\documentclass{amsart}
\usepackage{amsmath, amssymb, amsthm}

\newtheorem{theorem}{Theorem}[section]
\newtheorem{lemma}[theorem]{Lemma}

\theoremstyle{cor}
\newtheorem{cor}[theorem]{Corollary}

\newtheorem{Proposition}[theorem]{Proposition}

\theoremstyle{definition}
\newtheorem{definition}[theorem]{Definition}

\theoremstyle{remark}

\numberwithin{equation}{section}

\begin{document}

\title{Compact Support Cohomology 
\linebreak of Picard Modular Surfaces}

\author{Jukka Keranen}
\address{}
\curraddr{}
\email{}
\thanks{}

\subjclass[2010]{Primary 11R39}

\date{}

\begin{abstract}
We compute the cohomology with compact supports of a Picard modular surface as a virtual module over the product of the appropriate Galois group and the appropriate Hecke algebra. We use the method developed by Ihara, Langlands, and Kottwitz: comparison of the Grothendieck-Lefschetz formula and the Arthur-Selberg trace formula. Our implementation of this method takes as its starting point the works of Laumon and Morel.
\end{abstract}

\maketitle

\tableofcontents

\pagebreak
\section{Introduction}

\subsection{Goal of the Project}

Let $E$ be an imaginary quadratic extension of $\mathbb{Q}$. We define an algebraic group $\mathbf{GU}(2,1)$ over $\mathbb{Q}$ by setting\medskip
\begin{center}
$\mathbf{GU}(2,1)(A)= \lbrace g \in \mathbf{GL}_3(E \otimes_\mathbb{Q} A)\vert ^t\bar{g} J g = c(g)J, c(g) \in A^\times \rbrace$\medskip
\end{center}
for every $\mathbb{Q}$-algebra $A$, where\medskip
\begin{center}
$J=\left( \begin{array}{ccc}
0 & 0 & 1 \\
0 & 1 & 0 \\
1 & 0 & 0 \end{array}\right) \in \mathbf{GL}_3(\mathbb{Z}).$
\end{center}\medskip
Then $\mathbf{G} = \mathbf{GU}(2,1)$ is an algebraic group that is quasi-split over $\mathbb{Q}$.

Let $K$ be a compact open subgroup of $\mathbf{G}(\mathbb{A}_f)$, and let $S_K(\mathbf{G})$ be the Shimura variety of $\mathbf{GU}(2,1)$ at level $K$. We will assume $K$ to be neat, so that $S_K(\mathbf{G})$ is a smooth quasi-projective variety over its reflex field, which in this case is the imaginary quadratic field $E$. Any such $S_K(\mathbf{G})$ is called a \textit{Picard modular surface}.

Fix a good prime $p$. Assume, that is, that $p$ is not ramified in $E$, and that $K_p$ is hyperspecial, so that $K_p = \mathbf{G}(\mathbb{Z}_p)$. Let $\mathfrak{p}$ be a prime of $E$ dividing $p$.

\begin{definition}

The \textit{Hasse-Weil} $L$-\textit{function} at $\mathfrak{p}$ of $S_K(\mathbf{G})$ is defined by\medskip

\begin{center}
$\log L_{\mathfrak{p}} (S_K(\mathbf{G}), s) = \sum_{m \geqslant 1} \frac{N_m}{m} (q^{-s})^m$\medskip
\end{center}
where $O_{E_{\mathfrak{p}}}/\mathfrak{p} = \mathbb{F}_q$ and $s \in \mathbb{C}$, and $N_m$ the number of rational points on the special fiber at $\mathfrak{p}$ of $S_K(\mathbf{G})$ over the extension of degree $m$ of $\mathbb{F}_q$.
\end{definition}

The goal of this paper is to compute the parabolic part of the trace of certain correspondences on the cohomology with compact supports of $S_K(\mathbf{G})$ in terms of spectral traces. In an upcoming paper \cite{Keranen}, we will carry out the (much simpler) computation of the elliptic part of the said trace. Put together, these computations will allow us to give a spectral automorphic expression for the number $N_m$ of rational points. Hence, we will be able to express the Hasse-Weil $L$-function at $\mathfrak{p}$ of a Picard modular surface in terms of automorphic $L$-functions.

The distinctive challenge facing us is that the Picard modular surface $S_K(\mathbf{G})$ is not compact. Thus, in order to compute $N_m$, we shall have to compute the number of fixed points $N(j, f^p)$ of a suitable correspondence on the cohomology with compact supports of $S_K(\mathbf{G})$. Our basic tool will be the non-invariant twisted trace formula.

\subsection{Related work}

The study of $L$-functions of algebraic varieties is a vast discipline. The previous works most closely related to that of ours are due to Laumon and Morel.

(1) In 1997, Laumon computed the cohomology with compact supports for the group $\mathbf{GSp}(4)$; see \cite{L}. Our work is adapted from that of Laumon's.

(2) In 2010, Morel computed the intersection cohomology of the Baily-Borel compactification for $\mathbf{GU}(p, q)$ for arbitrary $p$ and $q$; see \cite{M}. We will follow Morel in our overall setup and notation. Her work generalizes the highly influential Montreal proceedings \cite{Montreal} from 1992 in which the case of $\mathbf{GU}(2,1)$ was worked out for the first time.\footnote{We would like to thank Paul Gunnells for reminding us to include this reference.} We should emphasize that there is no obvious way to deduce our results from the corresponding results for the Baily-Borel compactification.

\subsection{Structure of the Argument}

In this subsection, we will give a brief synopsis of our argument. In the interest of brevity, some of the standard notations employed here will only be defined in the subsequent sections.

\textsc{1. Stable Point Counting Formula}

Let $\mathbf{G} = \mathbf{GU}(2,1)$ and let $\mathbf{H} = \mathbf{G}(\mathbf{U}(1) \times \mathbf{U}(1,1))$, the only non-trivial elliptic endoscopic group of $\mathbf{G}$. In this situation, Kottwitz's stable point counting formula specializes to give\medskip
\begin{center}
$N(j,f^p)=ST_{e}^{\mathbf{G}}(f^{\mathbf{G}})+ \frac{1}{2}ST^{\mathbf{H}}_{e}(f^{\mathbf{H}})$\medskip
\end{center}
for suitable test functions $f^\mathbf{G}$ and $f^\mathbf{H}$. Our basic project is to give a sequence of different expressions for this quantity, finally to express it in terms of spectral traces.

It will be necessary to consider the following twisted sets and twisted groups:\medskip
\begin{center}
$G = R_{E/\mathbb{Q}}\mathbf{G}_E \rtimes \theta \subset R_{E/\mathbb{Q}}\mathbf{G}_E \rtimes \langle \theta \rangle = \tilde{G}$,

$H = R_{E/\mathbb{Q}}\mathbf{H}_E \rtimes \theta \subset R_{E/\mathbb{Q}}\mathbf{H}_E \rtimes \langle \theta \rangle = \tilde{H}$,\medskip
\end{center}
where $\theta$ is an automorphism on $\mathbf{G}$ and $\mathbf{H}$ induced by the non-trivial element of the Galois group Gal$(E/\mathbb{Q})$. We have the following equalities:\medskip

\begin{center}
$ST_{e}^{\mathbf{G}}(f^{\mathbf{G}}) = k_GT_{e}^{G}(\phi^{G})$

$ST_{e}^{\mathbf{H}}(f^{\mathbf{H}}) = k_HT_{e}^{H}(\phi^{H})$\medskip
\end{center}
where the test functions $\phi^{G}$ and $\phi^{H}$ are associated to the test functions $f^{\mathbf{G}}$ and $f^{\mathbf{H}}$ in a sense to be recalled, and $k_G$ and $k_H$ are constants to be determined. The first step in the computation will be to replace the distributions $ST^\mathbf{G}_e$ and $ST^\mathbf{H}_e$ on $\mathbf{G}$ and $\mathbf{H}$, respectively, by the distributions $T^{G}_e$ and $T^{H}_e$ on $G$ and $H$, respectively.

\textsc{2. Geometric Side of the Trace Formula}

Under suitable assumptions on the test functions $\phi^G$ and $\phi^H$, we have\medskip
\begin{center}
$T^{G}_{e}(\phi^{G})=T^{G}_{geom}(\phi^{G})$\medskip
\end{center}
and
\begin{center}
$T^{H}_{e}(\phi^{H})=T^{H}_{geom}(\phi^{H}),$\medskip
\end{center}
where $T^{G}_{geom}$ and $T^{H}_{geom}$ are the geometric sides of the non-invariant twisted trace formula for $G$ and $H$, respectively, and the test functions $\phi^{G}$ and $\phi^{H}$ are again associated with the test functions $f^\mathbf{G}$ and $f^\mathbf{H}$, respectively.
The second step in our computation is to replace the distributions $T^{G}_e$ and $T^{H}_e$ with the geometric sides of the trace formulas for $G$ and $H$, respectively. The point of the first step is that it makes this second step available to us.

\textsc{3. Spectral Side of the Trace Formula}

So far, up to certain coefficients we shall ignore in this introduction, we will have rewritten the stable point counting formula as\medskip
\begin{center}
$N(j,f^p)=T^{G}_{geom}(\phi^{G})+ \frac{1}{2} T^{H}_{geom}(\phi^{H}).$\medskip
\end{center}
The third main step in our computation is to apply the trace formula individually to each term above, so as to arrive at\medskip
\begin{center}
$N(j,f^p)=T^{G}_{spec}(\phi^{G})+ \frac{1}{2} T^{H}_{spec}(\phi^{H}).$\medskip
\end{center}
We can break up each term in this expression according to Levi subsets:\medskip
\begin{center}
$N(j,f^p)=J^{G}_{G}(\phi^{G}) + J^{G}_{T}(\phi^{G}) + \frac{1}{2} J^{H}_{H}(\phi^{H}) + \frac{1}{2} J^{H}_{T}(\phi^{H}),$\medskip
\end{center}
where $T$ is the diagonal Levi subset in $G$ and $H$. The rest of the argument consists of rewriting the distributions $J$ in successively more explicit forms.

\textsc{4. Stabilization of the Parabolic Part of the Trace}
 
Once written out explicitly, we will compare the parabolic terms $J^{G}_{T}(\phi^{G})$ and $J^{H}_{T}(\phi^{H})$. The expressions we will be working with are fairly complicated, and our project depends crucially on finding cancellations between certain terms in $J^G_T(\phi^G)$ and the corresponding terms in $J^H_T(\phi^H)$. After a lengthy computation, we are left with a relatively simple expression, which we will identify as the trace on a suitable virtual module for Gal$(\bar{E}/E) \times C^{\infty}_c(G(\mathbb{A}_f)//K)$. This identification will allow us to remove the assumptions on the test functions we had placed earlier. This is the core result of our computation. 

\textsc{5. The Elliptic Part of the Trace}

We will rewrite the elliptic terms $J^{G}_{G}(\phi^{G})$ and $J^{H}_{H}(\phi^{H})$ as traces on a suitable module for Gal$(\bar{E}/E) \times C^{\infty}_c(G(\mathbb{A}_f)//K)$. This is the most routine part of our computation. 

\textsc{6. The $L$-Function}

By using the expressions from steps 4. and 5., we will be able to express the $L$-function of $\mathbf{G}$ in terms of automorphic $L$-functions associated to $G$ and $H$. It is perhaps worth emphasizing that the final result will, therefore, express the $L$-function of the Shimura variety of a certain unitary group in terms of automorphic $L$-functions of certain general linear groups. This is a fundamental feature of the suite of techniques employed here.\footnote{In a similar fashion, Morel expresses the $L$-function of the intersection complex on the Baily-Borel compactification of a unitary group Shimura variety in terms of automorphic $L$-functions of general linear groups. See \cite[p.139]{M}.} 

\medskip

We will carry out steps 5. and 6. in a forthcoming paper \cite{Keranen}.\medskip

Since the choice of the test functions plays a large role in the actual execution of the argument outlined above, we pause here to explain the basic idea. First, the choice of the test functions on $\mathbf{G}$ is, of course, dictated by the trace we have resolved to compute in the first place, and the choice of the test functions on $\mathbf{H}$ is then dictated by the requirement that the appropriate instances of the Fundamental Lemma should hold and, in particular, that the Kottwitz point-counting formula should hold. In addition, Laumon's technique requires that certain assumptions be made about these functions, particularly at infinity. Second, the test functions on $G$ and $H$ are chosen so that we can pass from the distributions $ST$ on the unitary groups $\mathbf{G}$ and $\mathbf{H}$ to distributions $T$ on the linear groups $\tilde{G}$ and $\tilde{H}$, respectively. Again, in this new context, we can and will assume that the additional assumptions hold. We will then adapt the bulk of Laumon's method not for the original unitary groups but rather for the linear groups. Finally, we can discharge the additional assumptions on the test functions on $\tilde{G}$ and $\tilde{H}$ in the same way as Laumon does and hence, finally, we can discharge the corresponding assumptions on the original groups $\mathbf{G}$ and $\mathbf{H}$.

\section{Stable Point Counting Formula}

\subsection{Shimura Varieties and Integral Models}
In this paper, we will follow all the assumptions regarding Shimura varieties adopted by Morel in \cite{M}. In particular, the integral models of our Shimura varieties will be supplied by the work of Lan \cite{Lan}. We will recall the basic definitions and notations from chapter 1 of \cite{M} below; for more details, see \cite[pp.1-6]{M}.

Let $\mathbb{S} = R_{\mathbb{C}/\mathbb{R}} \mathbb{G}_{m, \mathbb{C}}$. Identify $\mathbb{S}(\mathbb{C}) = (\mathbb{C} \otimes_{\mathbb{R}} \mathbb{C})^{\times}$ by using the morphism \medskip
\begin{center}
$a \otimes 1 + b \otimes i \mapsto (a+ib,a-ib)$,\medskip
\end{center}
and write $\mu_0 : \mathbb{G}_{m, \mathbb{C}} \rightarrow \mathbb{S}_{\mathbb{C}}$ for the morphism $z \mapsto (z,1)$.

Following Morel, the definition of pure Shimura data that will be used here is that of \cite{P2} (3.1), up to condition (3.1.4). So a pure Shimura datum is a triple $(\mathbf{G}, \mathcal{X}, h)$ where $\mathbf{G}$ is a connected reductive linear algebraic group over $\mathbb{Q}$, $\mathcal{X}$ is a set with a transitive action of $\mathbf{G}(\mathbb{R})$, and $h: \mathcal{X} \rightarrow$ Hom$(\mathbb{S}, \mathbf{G}_{\mathbb{R}})$ is a $\mathbf{G}(\mathbb{R})$-equivarient morphism, satisfying conditions (3.1.1), (3.1.2), (3.1.3), and (3.1.5) of \cite{P2}, but not necessarily the condition (3.1.4) (the group $\mathbf{G}^{ad}$ may have a simple factor of compact type defined over $\mathbb{Q}$).

Let $(\mathbf{G}, \mathcal{X}, h)$ be a Shimura datum. The field of definition $F$ of the conjugacy class of cocharacters $h_x \circ \mu_0 : \mathbb{G}_{m,\mathbb{C}} \rightarrow \mathbf{G}_{\mathbb{C}}$, $x \in \mathcal{X}$, is called the \emph{reflex field} of the datum. If $K$ is a compact open subgroup of $\mathbf{G}(\mathbb{A}_f)$, there is an associated Shimura variety $M^K(\mathbf{G}, \mathcal{X})$, which is a quasi-projective variety over $F$ satisfying\medskip
\begin{center}
$M^K(\mathbf{G}, \mathcal{X})(\mathbb{C}) = \mathbf{G}(\mathbb{Q}) \setminus (\mathcal{X} \times \mathbf{G}(\mathbb{A}_f)/K).$\medskip
\end{center}
If moreover $K$ is \emph{neat} (see \cite[0.6]{P1}), then $M^K(\mathbf{G}, \mathcal{X})$ is smooth over $F$.

Suppose that we are given a Shimura datum as above, and that, in addition, all the assumptions in Morel's chapter 1, \cite{M}, are in force. Then the PEL moduli schemes from the work of Lan \cite{L} provide suitable integral models for our Shimura varieties; see Morel \cite[p.~9]{M}. Let $\mathit{S}_K$ denote the moduli scheme provided by Lan, and $S_K(\mathbf{G})$ the generic fiber of $\mathit{S}_K$.

\subsection{Kottwitz's Stable Point Counting Formula}

We now assume that $\mathbf{G}$ is one of the groups of unitary similitudes considered by Morel in chapter 2 of \cite{M}; concretely, we will only need the results of this section for $\mathbf{G} = \mathbf{GU}(2,1)$. Let $E$ be the imaginary quadratic field that is the reflex field of the Shimura datum associated to $\mathbf{G}$.

Fix a compact open subgroup $K$ of $\mathbf{G}(\mathbb{A}_f)$ and suppose that\medskip
\begin{center}
$K = K_N \subset K_{max} = \mathbf{G}(\hat{\mathbb{Z}}),$\medskip
\end{center}
where\medskip
\begin{center}
$K_N =$ Ker$(\mathbf{G}(\hat{\mathbb{Z}}) \twoheadrightarrow \mathbf{G}(\hat{\mathbb{Z}}/N\hat{\mathbb{Z}}))$,\medskip
\end{center}
for any integer $N \geqslant 3$.

Fix a prime $p$ that is good with respect to $K$, namely\medskip
\begin{center}
$K = K^pK_p$\medskip
\end{center}
where\medskip
\begin{center}
$K^p \subset \mathbf{G}(\mathbb{A}^p_f),$\medskip
\end{center}
and\medskip
\begin{center}
$K_p = \mathbf{G}(\mathbb{Z}_p) \subset \mathbf{G}(\mathbb{Q}_p),$\medskip
\end{center}
and such that $p \nmid N$, where $N$ is as above.

Fix an algebraic closure $\bar{\mathbb{Q}}$ of $\mathbb{Q}$ and an embedding of $\mathbb{Q}$ in $\bar{\mathbb{Q}}$. Fix an algebraic closure $\bar{\mathbb{Q}}_p$ of $\mathbb{Q}_p$ and an embedding of $\mathbb{Q}_p$ in $\bar{\mathbb{Q}}_p$. Also fix an embedding of $\bar{\mathbb{Q}}$ in $\bar{\mathbb{Q}}_p$.
Let $\mathbb{\bar{F}}_p$ be the residue field of the integral closure of $\mathbb{Z}_p$ in $\bar{\mathbb{Q}}_p$; then $\mathbb{\bar{F}}_p$ is an algebraic closure of $\mathbb{F}_p$, which is the one we fix. 

Fix an algebraic closure $\bar{E}$ of $E$ and an embedding of $E$ in $\bar{E}$. For any prime $\ell$, we will denote the $\ell$-adic cohomology with compact supports on $S_K(\mathbf{G})$ by\medskip
\begin{center}
$H^i_c(S_K(\mathbf{G}) \otimes_{E} \bar{E}, \mathbb{Q}_\ell).$\medskip
\end{center}
We will denote the convolution algebra of compactly supported, $K$-bi-invariant functions $f: \mathbf{G}(\mathbb{A}_f) \rightarrow \mathbb{C}$ by\medskip
\begin{center}
$C_c(\mathbf{G}(\mathbb{A}_f) // K),$\medskip
\end{center}
and the $\mathbb{Q}$-subspace consisting of the $\mathbb{Q}$-valued functions in $C_c(\mathbf{G}(\mathbb{A}_f) // K)$ by\medskip
\begin{center}
$C_c(\mathbf{G}(\mathbb{A}_f) // K)_{\mathbb{Q}}.$\medskip
\end{center}
There is a continuous action of Gal$(\bar{E}/E)$ on $H^i_c(S_K(\mathbf{G}) \otimes_{E} \bar{E}, \mathbb{Q}_l),$ and also an action of $C_c(\mathbf{G}(\mathbb{A}_f) // K)_{\mathbb{Q}}$, and the two actions commute.

For any prime $\ell$ such that $p \neq \ell$, we will consider the virtual Gal$(\bar{E}/E) \times C_c(\mathbf{G}(\mathbb{A}_f) // K)$ -module\medskip
\begin{center}
$W_\ell = \sum_{i\geqslant0} (-1)^{i} H^i_c(S_K(\mathbf{G}) \otimes_{E} \bar{E}, \mathbb{Q}_\ell).$\medskip
\end{center}

Fix a prime $\mathfrak{p}$ of $E$ dividing $p$. Fix an algebraic closure $\bar{E}_{\mathfrak{p}}$ of $E_{\mathfrak{p}}$ and an embedding of $E_{\mathfrak{p}}$ in $\bar{E}_{\mathfrak{p}}$. Also fix an embedding of $\bar{E}$ in $\bar{E}_{\mathfrak{p}}$.
Let $O_{E_{\mathfrak{p}}}/\mathfrak{p} = \mathbb{F}_q$, and let   $\mathbb{\bar{F}}_q$ be the residue field of the integral closure of $O_{E_{\mathfrak{p}}}$ in $\bar{E}_{\mathfrak{p}}$; then $\mathbb{\bar{F}}_q$ is an algebraic closure of $\mathbb{F}_q$, which is the one we fix. These choices, along with the choices made earlier, determine unique homomorphisms\medskip

\begin{center}
Gal$(\bar{E}/E) \hookleftarrow$ Gal$(\bar{E}_{\mathfrak{p}}/E_{\mathfrak{p}}) \twoheadrightarrow$ Gal$(\mathbb{\bar{F}}_q/ \mathbb{F}_q).$\medskip
\end{center}
For each $i$, we have Gal$(\mathbb{\bar{F}}_q/ \mathbb{F}_q) \times C_c(G(\mathbb{A}_f) // K)_{\mathbb{Q}}$ -equivariant isomorphisms\medskip
\begin{center}
$H^i_c(S_K(\mathbf{G}) \otimes_{E} \bar{E}, \mathbb{Q}_\ell) \cong H^i_c(S_K(\mathbf{G}) \otimes_{E} \bar{E}_{\mathfrak{p}} , \mathbb{Q}_\ell)\medskip \linebreak$\medskip
$\medskip\cong H^i_c(S_K(\mathbf{G}) \otimes_{O_{E_{\mathfrak{p}}}} \bar{\mathbb{F}}_q, \mathbb{Q}_\ell).$
\end{center}
We shall only have to consider the case where $p$ splits in $E$, since this will be enough to determine the $L$-function. For such a $p$, let Frob$_p$ denote the topological generator of Gal$(\bar{\mathbb{F}}_q/\mathbb{F}_q)$ that is given by\medskip
\begin{center}
$\alpha \mapsto \alpha^{1/p}.$\medskip
\end{center}
Let $\Phi_p$ be an arbitrary, fixed lift of Frob$_p$ to Gal$(\bar{E}/E)$.

By the above, we have\medskip
\begin{center}
$\mathrm{tr}(\Phi_{p}^{j} \times f^{p} \mathbf{1}_{K_{p}}, W_{\ell}) =\mathrm{tr}(\mathrm{Frob}_{p}^{j} \times f^{p}, R\Gamma_{c}(S_{K} \otimes_{O_{E_{\mathfrak{p}}}} \bar{\mathbb{F}}_q, \mathbb{Q}_{\ell}))$
\end{center}\medskip
for every integer $j \geqslant 0$ and every function $f^p \in C_c(\mathbf{G}(\mathbb{A}_f) // K)_{\mathbb{Q}}$, where $\mathbf{1}_{K_p}$ is the characteristic function of ${K_p}$ in $\mathbf{G}({\mathbb{Q}}_p)$; see \cite[p.~271]{L}.

\subsubsection{Number of Fixed Points}

If $f^p$ is the characteristic function of the double coset $K^pgK^p$ in $\mathbf{G}(\mathbb{A}^p_f)$, with $g \in \mathbf{G}(\mathbb{A}^p_f)$, and if\medskip
\begin{center}
$p^j > [K^p : K^p \cap gK^pg^{-1}]$,\medskip
\end{center}
then the fixed points of the correspondence Frob$^j_p \times f^p$ on the $\mathbb{F}_q$-scheme $\mathit{S}_K \otimes_{O_{E_{\mathfrak{p}}}} \bar{\mathbb{F}}_q$ are isolated; see  Zink \cite{Z}. We will write\medskip
\begin{center}
$N(j, f^p)$ \medskip
\end{center}
for the number of fixed points counted with multiplicity. By $\mathbb{Q}$-linearity, we define $N(j, f^p)$ for any $f^p \in C_c(\mathbf{G}(\mathbb{A}_f) // K)_{\mathbb{Q}}$ and $j$ sufficiently large.

The following theorem of Pink's is still often called \textit{Deligne's Conjecture}.

\begin{theorem}
For every function $f^p \in C_c(\mathbf{G}(\mathbb{A}^p_f)//K^p)_{\mathbb{Q}}$, there exists an integer $j(f^p) > 0$ with the following property. For every integer $j \geqslant j(f^p)$, the number of fixed points $N(j, f^p)$ of the correspondence $\mathrm{Frob}^j_p \times f^p$ is well-defined, and we have\medskip
\begin{center}

$N(j, f^p) = \mathrm{tr}(\mathrm{Frob}^j_p \times f^p, R\Gamma_c(S_K \otimes_{O_{E_{\mathfrak{p}}}} \bar{\mathbb{F}}_q, \mathbb{Q}_\ell)).$
\end{center}

\end{theorem}

\medskip
\begin{proof}
This is a direct consequence of theorem 7.2.2 of \cite{P}.
\end{proof}

\begin{cor}
\textit{For every function $f^p \in C_c(\mathbf{G}(\mathbb{A}^p_f)//K^p)_{\mathbb{Q}}$ and every integer $j \geqslant j(f^p)$, we have}\medskip
\begin{center}
$N(j, f^p) = \mathrm{tr}(\Phi_{p}^{j} \times f^{p} \mathbf{1}_{K_{p}}, W_{\ell}).$\medskip
\end{center}
\end{cor}

That is, the number of fixed points of a correspondence on the geometric special fiber is equal to the trace of the corresponding operator on the $l$-adic cohomology of $S_K(\mathbf{G})$. We now recall Kottwitz's well-known expression for this quantity.

\begin{theorem}
There is an equality\medskip
\begin{center}
$N(j, f^p) = \sum_{(\mathbf{H}, s, \eta_0)} \iota(\mathbf{G}, \mathbf{H}) ST^{\mathbf{H}\ast}_e(f^\mathbf{H})$\medskip
\end{center} 
where the sum is taken over the elliptic endoscopic triples of $\mathbf{G}$, \and for each $\mathbf{H}$, the function $f^\mathbf{H}$ is a transfer of $f^\mathbf{G}$.
\end{theorem}

We will take the basic concepts and results from the theory of endoscopy for granted, and we will adopt the various normalizations thereof from Morel \cite{M}. In particular, see  \cite[p.88]{M} for the notation and the references for theorem 2.3. Specializing to the case we are concerned with in this paper, we have the following

\begin{cor}
\textit{For} $\mathbf{G}=\mathbf{GU}(2,1)$ \textit{and} $\mathbf{H}=\mathbf{G}(\mathbf{U}(1) \times \mathbf{U}(1,1))$, we have\medskip

\begin{center}
$N(j,f^p)=ST_{e}^{\mathbf{G}}(f^{\mathbf{G}})+\frac{1}{2}ST^{\mathbf{H}}_{e}(f^{\mathbf{H}}).$\medskip
\end{center}
\end{cor}

\begin{proof}In this case, $\mathbf{G}$ and $\mathbf{H}$ are the only elliptic endoscopic groups. Further, in this case, only $(\mathbf{G}, \mathbf{H})$-regular orbits will contribute; see \cite[p.~89]{M}. Thus, the restriction to such orbits in theorem 2.3, indicated with the asterisk, can be ignored.
\end{proof}

In particular, we have the following expression for the number of rational points $N_j$ in the definition of the $L$-function.

\begin{cor}
When $f^p = \mathbf{1}_{K^p}$, we have\medskip

\begin{center}
$N_j = N(j,f^p)=ST_{e}^{\mathbf{G}}(f^{\mathbf{G}})+\frac{1}{2}ST^{\mathbf{H}}_{e}(f^{\mathbf{H}}).$\medskip
\end{center}
\end{cor}

\subsubsection{Choice of Test Functions}
We now focus on the groups $\mathbf{G} = \mathbf{GU}(2,1)$ and $\mathbf{H}=\mathbf{G}(\mathbf{U}(1) \times \mathbf{U}(1,1))$.
The test function $f^\mathbf{G}$ on $\mathbf{G}$ can be chosen as follows:\medskip
\begin{center}
$f^\mathbf{G} = f^\mathbf{G}_\mathbb{R} \times b^\mathbf{G}_j(\varphi_j) \times f^p,$\medskip
\end{center}
where

(1) $f^\mathbf{G}_\mathbb{R} = \frac{1}{3}(f_{\pi_1} + f_{\pi_2} + f_{\pi_3})$, the sum of pseudocoefficients of the discrete series $L$-packet $\Pi = \lbrace \pi_1, \pi_2, \pi_3 \rbrace$ of $\mathbf{G}(\mathbb{R})$ associated to the trivial representation of $\mathbf{G}$.

(2) $\varphi_j \in C_c(\mathbf{G}(\mathbb{\mathbb{Q}}_{p^j})//K_{p^j})$ is the characteristic function of the double coset $K_{p^j} \mu(p) K_{p^j}$, with $\mu$ the cocharacter associated to the Shimura datum, \cite[p.33]{M};\medskip
\begin{center}
$b^{\mathbf{G}}_j: C_c(\mathbf{G}(\mathbb{Q}_{p^j})//K_{p^j}) \rightarrow C_c(\mathbf{G}(\mathbb{Q}_p)//K_p)$ \medskip
\end{center}
the change of base map, with $\mathbb{Q}_{p^j}$ the unramified extension of $\mathbb{Q}_{p}$ of degree $j$ contained in $\mathbb{\bar{Q}}_p$.

(3) Pick a prime $q \neq p$ such that $K^p = K^{p,q}\mathbf{G}(\mathbb{Z}_q)$ and $f^p = f^{p,q}\mathbf{1}_{K_q}$, where $f^{p,q} \in C^{\infty}_{c}(\mathbf{G}(\mathbb{A}^{p,q}_f)//K^{p,q})$ is for now arbitrary.

The test function $f^{\mathbf{H}}$ is then chosen as follows:\medskip
\begin{center}
$f^\mathbf{H} = f^\mathbf{H}_\mathbb{R} \times b^\mathbf{H}_j(\varphi_j) \times h^p,$\medskip
\end{center}
where

(1) $f^{\mathbf{\mathbf{H}}}_{\mathbb{R}} = f_{\rho^+} + f_{\rho^-} - f_{\rho^0}$, the (weighted) sum of pseudocoefficients of the three discrete series $L$-packets of $\mathbf{H}(\mathbb{R})$ that transfer to the given $L$-packet $\Pi$ of $\mathbf{G}(\mathbb{R})$.

(2) $b^{\mathbf{H}}_j: C_c(\mathbf{G}(\mathbb{Q}_{p^j})//K_{p^j}) \rightarrow C_c(\mathbf{H}(\mathbb{Q}_p)//K^{\mathbf{H}}_p)$ is defined in the same way as in Laumon, see \cite[p.~288]{L};

(3) $h^p$ is an arbitrary transfer of $f^p$.

We will also need test functions on $G = R_{E/\mathbb{Q}}\mathbf{G}_E \rtimes \theta$ and $H = R_{E/\mathbb{Q}}\mathbf{H}_E \rtimes \theta$ that are associated to the given functions on $\mathbf{G}$ and $\mathbf{H}$ at all places in the sense of \cite[3.2]{Lab}. For the details of these choices and the underlying normalizations, we refer to Morel \cite[p.~138]{M}. Briefly, the test function $\phi^{G}$ and $\phi^{H}$ are chosen as follows:\medskip
\begin{center}
$\phi^{G} = \phi^{G}_\mathbb{R} \times \phi^G_j \times \phi^{G,p},$\medskip
\end{center}
and
\begin{center}
$\phi^{H} = \phi^{H}_\mathbb{R} \times \phi^H_j \times \phi^{H,p},$\medskip
\end{center}
where, for any field $F$ and any test function $\phi^0$ on $G^0(F)$, we associate to $\phi^0$ a test function $\phi$ on $G(F) = G^0(F) \rtimes \theta$ by setting\medskip
\begin{center}
$\phi(x \rtimes \theta) := \phi^0(x)$\medskip
\end{center} 
for $x \in G^0(F)$, and similarly for $H$; see \cite[p.~98]{Lab}. The connected components of $\phi^{G}$ and $\phi^{H}$ are then given as follows (where we will ignore the superscript):

(1) $\phi^G_{\mathbb{R}}$ is a pseudocoefficient of the $\theta$-discrete representation on $G^0(\mathbb{R})$ that corresponds to the discrete series $L$-packet $\Pi$ of $\mathbf{G}(\mathbb{R})$ above; see \cite[p.~124]{M}. Similarly,  $\phi^{H}_{\mathbb{R}} = \phi_{\rho^+} + \phi_{\rho^-} - \phi_{\rho^0}$, the sum of pseudocoefficients of $\theta$-discrete representations of $H^0(\mathbb{R})$ that correspond to the three $L$-packets of $\mathbf{H}(\mathbb{R})$ indicated above.

(2) $\phi^G_j$ and $\phi^H_j$ are chosen as on p.138 of \cite{M}.

(3) $\phi^{G,p}$ is associated at every place to $f^p$, and $\phi^{H,p}$ is associated at every place to $h^p$.

\pagebreak
\subsection{First Major Transition}
Now let $G$ be a connected component of a reductive group $\tilde{G}$ over $\mathbb{Q}$. Again, the cases we are interested in this paper are\medskip

\begin{center}
$G = R_{E/\mathbb{Q}}\mathbf{G}_E \rtimes \theta \subset R_{E/\mathbb{Q}}\mathbf{G}_E \rtimes \langle \theta \rangle = \tilde{G}$,\medskip

$H = R_{E/\mathbb{Q}}\mathbf{H}_E \rtimes \theta \subset R_{E/\mathbb{Q}}\mathbf{H}_E \rtimes \langle \theta \rangle = \tilde{H}$,\medskip
\end{center}
where $\theta$ is an automorphism on $\mathbf{G}$ and $\mathbf{H}$ induced by the non-trivial element of the Galois group Gal$(E/\mathbb{Q})$.

Let $T_e$ be a torus of $G^0_{\mathbb{R}}$ such that $T_e(\mathbb{R})$ is a maximal torus of the set of fixed points of a Cartan involution of $G^0(\mathbb{R})$ that commutes with $\theta$; see \cite[p.~121]{M}. Set\medskip
\begin{center}
$d(G) = \vert $Ker$(\mathbf{H}^1(\mathbb{R}, T_e) \rightarrow \mathbf{H}^1(\mathbb{R}, G^0))\vert$.\medskip
\end{center}
In our situation, $G^0$ comes from a complex group by restriction of scalars, and hence $\mathbf{H}^1(\mathbb{R}, G^0) = \lbrace1\rbrace$ and  $d(G) = \vert \mathbf{H}^1(\mathbb{R}, T_e) \vert$. For a quasi-split unitary group $\mathbf{G}(\mathbf{U}^\ast(n_1) \times ... \times \mathbf{U}^\ast(n_r))$, with $ n := n_1 + ... + n_r$, we have $T_e = \mathbf{G}(\mathbf{U}(1)^n)$. Thus, $d(G) = 2^{n-1}$; see \cite[p.~123]{M}.

\begin{theorem}
We have\medskip
\begin{center}
$T_{e}^{G}(\phi^{G}) = C_{G}\frac{d(G)}{\tau(\mathbf{G})}ST_{e}^{\mathbf{G}}(f^\mathbf{G})$\medskip
\end{center}
and\medskip
\begin{center}
$T_{e}^{H}(\phi^{H}) = C_{H}\frac{d(H)}{\tau(\mathbf{H})}ST_{e}^{\mathbf{H}}(f^\mathbf{H}),$\medskip
\end{center}
where the test functions are as chosen in subsection 2.2.2 above.
\end{theorem}

\begin{proof} This is a special case of proposition 8.3.1 in \cite[pp.~130-1]{M}.
\end{proof}

\begin{lemma}
\textit{In the situation of the theorem, we have}

(1) $\tau(\mathbf{G}) = 1$ \textit{and} $\tau(\mathbf{H}) = 2;$

(2) $d(G) = d(H) = 4;$

(3) $C_G = C_H = 1/4.$
\end{lemma}

\begin{proof} Part (1) follows from lemma 2.3.3 in \cite[p.~40]{M}. Part (2) follows directly from the remarks made before theorem 2.6. For part (3), by comparing Morel's statement of theorem 2.6 (proposition 8.3.1 in \cite{M}) with her source for this theorem, namely Labesse's theorem 4.3.4 in \cite{Lab}, we deduce that\medskip
\begin{center}
$C_G = \frac{\tau(G^0)J_Z(\theta)}{2^k d(G)},$\medskip
\end{center}
where $k =$ dim $\mathfrak{a}_G$, and $J_Z(\theta) = \vert $det$(1 - \theta \vert \mathfrak{a}_{G^0} / \mathfrak{a}_G)\vert$; see \cite[p.~97]{Lab}. Note that our choices concerning the spaces of test functions differ from those of Morel, and hence, the two simplifications she notes at the top of page 131 of \cite{M} do not apply in our situation. Instead, we compute directly that $J_Z(\theta) = 4$ and dim $\mathfrak{a}_G = 2$, while $\tau(G^0) = 1$ by lemma 2.3.3 of \cite{M}. Thus, we have\medskip
\begin{center}
$C_G = \frac{1 \times 4}{4 \times 4} = \frac{1}{4}.$\medskip
\end{center}
A similar computation shows that $C_H = 1/4$.
\end{proof}
\noindent
\emph{Remark.} The test functions $\phi^G_{\mathbb{R}}$ and $\phi^H_{\mathbb{R}}$ are associated with the test functions $d(G)f^{\mathbf{G}}_{\mathbb{R}}$ and $d(H)f^{\mathbf{H}}_{\mathbb{R}}$ by \cite[p.~124]{M}. Thus, the appearance of the factors $d(G)$ and $d(H)$ in the foregoing identities is due to our choice of test functions on $\mathbf{G}(\mathbb{R})$ and $\mathbf{H}(\mathbb{R})$, and likewise for their reciprocals. Thus, the constants $C_G$ and $C_H$ are, essentially, equal to 1.\medskip

Putting together everything we have done thus far, we can now write our central object of interest, the number of fixed points $N(j, f^p)$ as follows.  

\begin{theorem}
In the situation of theorem 2.6, we have\medskip
\begin{center}
$N(j,f^p) = T_{e}^{G}(\phi^{G})+ T_{e}^{H}(\phi^{H}).$\medskip
\end{center}
\end{theorem}

\begin{proof} This is a direct consequence of corollary 2.4, theorem 2.6, and lemma 2.7. 
\end{proof}

The next step in our project is to write the distributions $T^G_e(\phi^G)$ and $T^H_e(\phi^H)$ as explicitly as we can.

\section{Geometric Side of the Trace Formula}

In this section, we will take $G$ to be a reductive algebraic group, not necessarily connected, over $\mathbb{Q}$. Let $\mathcal{L}_{\mathbb{R}}$ be the set of Levi subsets of $G$ defined over $\mathbb{R}$; we will recall the definition of Levi subset in subsection 4.1 below.

\subsection{Definitions}

A function $f_{\mathbb{R}} \in \mathcal{H} (A_G(\mathbb{R})^0 \setminus G(\mathbb{R}))$ is called \textit{cuspidal} if for every $M \in \mathcal{L}_{\mathbb{R}}$, $M \neq G$, and for every $P \in \mathcal{P}(M)$, we have\medskip
\begin{center}
$\mathrm{tr} \pi(f_{\mathbb{R},P}) = 0$\medskip
\end{center}
for every irreducible tempered representation $\pi$ of $A_M(\mathbb{R})^0 \setminus M(\mathbb{R})$, where\medskip
\begin{center}
$f_{\mathbb{R},P}(m) := \delta_{P(\mathbb{R})}(m)^{1/2} \int_{K_{max,\mathbb{R}}} \int_{N_P(\mathbb{R})} f_{\mathbb{R}}(k^{-1}mnk)dndk.$\medskip
\end{center}

Further, a function $f_{\mathbb{R}} \in \mathcal{H}(A_G(\mathbb{R})^0 \setminus G(\mathbb{R}))$ is called \textit{stable cuspidal} if it is cuspidal and, in addition, we have\medskip
\begin{center}
$\mathrm{tr} \pi(f_{\mathbb{R}}) = 0$\medskip
\end{center}
for every irreducible tempered representation $\pi$ of $A_M(\mathbb{R})^0 \setminus M(\mathbb{R})$ that is not square-integrable, and\medskip
\begin{center}
$\mathrm{tr} \pi_{\mathbb{R},1}(f_{\mathbb{R}}) = \mathrm{tr} \pi_{\mathbb{R},2}(f_{\mathbb{R}})$\medskip
\end{center}
for any two square integrable representations $\pi_{\mathbb{R},1}$ and $\pi_{\mathbb{R},2}$ of $A_G(\mathbb{R})^0 \setminus G(\mathbb{R})$ that belong to the same $L$-packet.

A function $f_{\mathbb{R}} \in \mathcal{H} (A_G(\mathbb{R})^0 \setminus G(\mathbb{R}))$ is called \textit{very cuspidal} if it is invariant under conjugation by $K_{max, \mathbb{R}}$ and if for every $M \in \mathcal{L}_{\mathbb{R}}$ and every $P \in \mathcal{P}(M)$, we have\medskip
\begin{center}
$f_{\mathbb{R},P}(m) = \delta_{P(\mathbb{R})}(m)^{1/2} \int_{N_P(\mathbb{R})} f_{\mathbb{R}}(mn)dn = 0$\medskip
\end{center}
for all $m$ in $M(\mathbb{R})$.

A function $f_{\mathbb{R}}$ that is very cuspidal is also cuspidal, but need not be stable cuspidal.

Let $C \in \mathbb{R}_+$. A function $f' \in C^{\infty}_c(G(\mathbb{A}_f))$ is called $C$-\textit{regular} (resp. \textit{strongly $C$-regular}) if for every $M \in \mathcal{L}$, $M \neq G$, every $P \in \mathcal{P}(M)$, and $m \in M(\mathbb{A}_f)$ such that\medskip
\begin{center}
$f_P'(m) := \delta_{P(\mathbb{A}_f)}(m)^{1/2} \int_{K_{max, f}} \int_{N_P(\mathbb{A}_f)} f'(k^{-1}mnk)dndk \neq 0,$\medskip

\end{center}
there exists at least one root $\alpha$ of $A_M$ in $G$ for which (resp. for all roots $\alpha$ of $A_M$ in $G$) we have\medskip
\begin{center}
$\vert \alpha(H_{M,f}(m)) \vert > C$;\medskip
\end{center}
see subsection 4.1 for the definition of the morphism $H$.

Clearly, a strongly $C$-regular function is also $C$-regular.

We will occasionally employ `twisted' versions of these concepts. That is, we will use concepts like $C$-regularity in the context of the twisted trace formula for $G^0$. In that situation, Levi and parabolic subsets are assumed to range over the set of $\theta$-stable ones, which are the ones that correspond to the ones on $G$.

\subsection{Second Major Transition}

\begin{theorem} 
Let $\phi^G_{\mathbb{R}}$ be a function in $\mathcal{H}(A_G(\mathbb{R})^0 \setminus G(\mathbb{R}))$ that is stable cuspidal and very cuspidal. Then there exists a constant $C \in \mathbb{R}_+$ that only depends on the support of $\phi^G_{\mathbb{R}}$ and has the following property. For every function $\phi' \in C^{\infty}_c (G(\mathbb{A}_f))$ that is $C$-regular, we have\medskip
\begin{center}
$T^G_e(\phi^G_\mathbb{R}\phi') = J^G_{geom}(\phi^G_\mathbb{R}\phi').$\medskip
\end{center}
\end{theorem}

\begin{proof} The proof Laumon gives in the connected case, \cite[p.~302]{L} still works in the non-connected case.
\end{proof}

\section{Spectral Side of the Trace Formula}

The goal of this section is to write the spectral side of the non-invariant trace formula for $G$ and $H$ as explicitly as possible. This will permit us to combine the terms corresponding to the diagonal Levi subset of $G$ with those corresponding to the diagonal Levi subset of $H$.

The arguments in this section are adapted from the corresponding arguments in Laumon \cite{L}. The difference is that we are working in the setting of the non-connected sets $G$ and $H$, whereas Laumon is working in the setting of the connected groups $\mathbf{G}$ and $\mathbf{H}$. As it turns out, virtually all of Laumon's arguments carry over to our setting with little or no change. We will begin by giving an overview of how this comes about.

Laumon has two basic references for his work on the spectral side of the non-invariant trace formula. First, he adopts the basic expression from Arthur's ``The Invariant Trace Formula II. Global Theory'' \cite{A88II}. This paper gives a uniform treatment of both the connected and non-connected cases, and so the results there are directly applicable to our setting. By way of clarification, the principal result of \cite{A88II} is an invariant form of the trace formula. However, following Laumon, we are only interested in the non-invariant results that Arthur compiles as a preparation for his final derivation.
 
Second, Laumon makes a crucial use of certain distributions defined in terms of residues of Eisenstein series, and for this part of the argument, his basic reference is Arthur's 1993 paper ``On elliptic tempered characters'' \cite{A93}. Now, this paper is written in the connected setting, and so some remarks must be made about why the results in it are applicable in our setting. First, the definition of these crucial distributions is given in Arthur's earlier papers, such as ``Intertwining operators and residues I. Weighted character'' \cite{A89I}, which is the basic reference in the 1993 paper. Happily, \cite{A89I} is written in the general setting that includes the non-connected case. In particular, the definition and basic properties of these distributions are available to us. The second basic reference for the 1993 paper is ``Intertwining operators and residues II. Invariant distributions'' \cite{A89II}. This paper is again written in the connected setting. However, the reason for this restriction is well-understood: at the time, the trace Paley-Wiener theorem was not yet available for non-connected groups. It was established by Delorme-Mezo in 2008 \cite{DM}. In fact, Morel analyzes the situation completely in \cite[p.~128]{M}: all the results of Arthur's \cite{A89II} carry over to the non-connected setting, provided that one replaces characters with twisted characters and discrete representations at the real place with $\theta$-discrete ones. Consequently, it is easy to check that all the results we need from Arthur's 1993 paper are also available in the non-connected setting.

None of this should seem particularly surprising. From the start, Arthur took great pains to write his papers in a way that applies uniformly to the connected and non-connected cases. In fact, the only slightly subtle point concerns Laumon's Lemma 4.13 \cite[p.313]{L} which gives a finite Fourier expansion of the aforementioned distributions. The proof of this lemma uses the Harish-Chandra character formula, and to our knowledge, this formula has not been proven in the non-connected case for real groups. However, we can circumvent this problem by using a character identity that allows us to pass from twisted characters on the non-connected side to stable characters on the connected side, and this allows us to run Laumon's argument as before; see Lemma 4.8 below.

\subsection{Definitions}
\subsubsection{Parameter Spaces}
Let $F$ be a field, local or global, and let $\tilde{G}$ be a reductive algebraic group, not necessarily connected, over $F$. Fix a connected component $G$ of $\tilde{G}$, and assume that $G$ generates $\tilde{G}$ and that $G(F) \neq \varnothing$. Let $G^0$ be the connected component of 1 in $\tilde{G}$. Our focus shall be on $G$ and $G^0$ rather than on $\tilde{G}$.
We will consider two instances of these definitions:\medskip

\begin{center}
$G = R_{E/\mathbb{Q}}\mathbf{GU}(2,1)_{E} \rtimes \theta \subset R_{E/\mathbb{Q}}\mathbf{GU}(2,1)_{E} \rtimes \langle \theta \rangle = \tilde{G},$\medskip

$G^0 = R_{E/\mathbb{Q}}\mathbf{GU}(2,1)_{E} \rtimes 1,$\medskip
\end{center}
and\medskip
\begin{center}
$H = R_{E/\mathbb{Q}} \mathbf{G}(\mathbf{U}(1) \times \mathbf{U}(1,1))_{E} \rtimes \theta \subset R_{E/\mathbb{Q}} \mathbf{G}(\mathbf{U}(1) \times \mathbf{U}(1,1))_{E} \rtimes \langle \theta \rangle = \tilde{H},$\medskip

$H^0 = R_{E/\mathbb{Q}} \mathbf{G}(\mathbf{U}(1) \times \mathbf{U}(1,1))_{E} \rtimes 1$\medskip
\end{center}
where $\theta$ is the automorphism induced by the non-trivial element on Gal$(E/\mathbb{Q})$. Since $G$ is not a group, some care must be taken in defining the various group-theoretic notions such as Levi components and representations for $G$.

Following Arthur and Morel, we will adopt the following definitions; see, for example, \cite[p.~228]{A88W} and also \cite[p.~119]{M}. A \emph{parabolic subgroup} of $\tilde{G}$ is the normalizer in $\tilde{G}$ of a parabolic subgroup of $G^0$. A \emph{parabolic subset} of $G$ is a nonempty subset of $G$ that is equal to the intersection of $G$ with a parabolic subgroup of $\tilde{G}$. If $P$ is a parabolic subset of $G$, write $\tilde{P}$ for the subgroup of $\tilde{G}$ generated by $P$, and write $P^0$ for the intersection $\tilde{P} \cap G^0$.

Let $P$ be a parabolic subset of $G$. The \emph{unipotent radical} $N_P$ of $P$ is defined to be the unipotent radical of $P^0$. A \emph{Levi component} $M$ of $P$ is a subset of $P$ that is equal to $\tilde{M} \cap P$, where $\tilde{M}$ is the normalizer in $\tilde{G}$ of a Levi component $M^0$ of $P^0$. If $M$ is a Levi component of $P$, then $P = MN_P$.

A \emph{Levi subset} $M$ of $G$ is a Levi component of a parabolic subset of $G$. Let $\tilde{M}$ be the subgroup of $\tilde{G}$ generated by $M$, let $M^0 = \tilde{M} \cap G^0$. Let $A_M$ be the maximal split subtorus of the centralizer of $M$ in $M^0$. Then $A_M \subseteq A_{M^0}$, and in general, the containment is proper.

Let $M$ be a Levi subset of $G$. Arthur defines the $\mathbb{R}$-vector space\medskip
\begin{center}
$\mathfrak{a}_{M} = Hom(X^\ast(\tilde{M})_F, \mathbb{R}),$ \medskip
\end{center}
where $X^{\ast}(\tilde{M})_F$ is the group of characters of $\tilde{M}$ defined over $F$; \cite[p.~24]{A89I}. Note that
\begin{center}
$\mathfrak{a}_M \subseteq \mathfrak{a}_{M^0},$\medskip
\end{center}
and in general, the containment is proper. In the two cases we are concerned with, $F = \mathbb{Q}$, and
\begin{center}
$\mathfrak{a}_M \simeq \mathfrak{a}_{M^0}^\theta;$\medskip
\end{center}
the Arthur space of a Levi subset of $G$ is obtained by taking the $\theta$-invariants of the Arthur space of the corresponding Levi subgroup of $G^0$, and similarly for $H$; \cite[p.~120]{M}. In particular, if $T = T^0 \rtimes \theta$ is the diagonal Levi subset in either $G$ or $H$, we have\medskip
\begin{center}
$\mathfrak{a}_{T} \simeq \lbrace (y,(x, 0, -x)) \in \mathbb{R} \times \mathbb{R}^3 \rbrace.$\medskip
\end{center}
This is also the Arthur space of the diagonal Levi subgroup of $\mathbf{GU}(2,1)$ and of $\mathbf{G}(\mathbf{U}(1) \times \mathbf{U}(1,1))$. 

Now suppose that $F$ is a global field, and fix a finite set $S$ of inequivalent valuations on $F$. Then\medskip
\begin{center}
$F_S = \prod_{q \in S} F_q$\medskip
\end{center}
is a locally compact ring. We can regard $M$, $\tilde{M}$ and $M^0$ as schemes over $F_S$, and since $F$ embeds diagonally in $F_S$, we can form the corresponding sets of $M(F_S)$, $\tilde{M}(F_S)$, and $M^0(F_S)$ of $F_S$-valued points. Consider the homomorphism\medskip
\begin{center}
$H_{M, S}: \tilde{M}(F_S) \rightarrow \mathfrak{a}_M,$\medskip
\end{center}
which is defined by\medskip
\begin{center}
$e^{\langle H_{M,S}(x), \chi \rangle} = \vert \chi(x) \vert = \prod_{q \in S} \vert \chi(x_q) \vert_q,$\medskip
\end{center}
for any $x = \prod_{q \in S} x_q$ in $\tilde{M}(F_S)$ and $\chi$ in $X^\ast(\tilde{M})_F$.\medskip We will write\medskip
\begin{center}
$\mathfrak{a}_{M,S} = H_{M,S}(\tilde{M}(F_S))$\medskip
\end{center}
for the image of this homomorphism. If the set $S$ contains any of the infinite places of $F$, then $\mathfrak{a}_{M,S} = \mathfrak{a}_M$. If $S = \lbrace q \rbrace$ for a finite place $q$ of $F$, then $\mathfrak{a}_{M,S}$ is a lattice in $\mathfrak{a}_M$; see \cite[p.~50]{A89I}.

Further, we set\medskip
\begin{center}
$\mathfrak{a}^\vee_{M,S} = Hom(\mathfrak{a}_{M,S}, 2\pi \mathbb{Z}).$\medskip
\end{center}
Then\medskip
\begin{center}
$\mathfrak{a}^{\ast}_{M,S} := \mathfrak{a}^{\ast}_{M} / \mathfrak{a}^{\vee}_{M,S}$ \medskip
\end{center}
is the additive character group of $\mathfrak{a}_{M,S}$. It is a compact quotient of $\mathfrak{a}^{\ast}_M$ if $S$ does not contain any of the infinite places of $F$, and is equal to $\mathfrak{a}^{\ast}_M$ otherwise; \cite[pp.~24-5]{A89I}.

If $S$ is the set of all the finite places of $F$, we will write $\mathfrak{a}_{M,f}$ and $\mathfrak{a}^{\ast}_{M,f}$ for these two groups.

\subsubsection{Representations}

Let $\Pi(\tilde{M}(F_S))$ be the set of equivalence classes of irreducible admissible representations of $\tilde{M}(F_S)$. We define $\Pi(M(F_S))$ to be the subset of $\Pi(\tilde{M}(F_S))$ consisting of those classes $\pi$ whose restriction $\pi^0$ to $M^0(F_S)$ remains irreducible. Note that $\pi^0$ is invariant under the finite group $\tilde{M}(F_S)/M^0(F_S)$ and conversely, any irreducible representation of $M^0(F_S)$ that is invariant by this group is the $\pi^0$ for some $\pi$ in $\Pi(M(F_S))$; \cite[p.~27]{A89I}.

An admissible representation $\pi^0$ of $\mathbf{M}^0(\mathbb{A})$ is called $\theta$-\emph{stable} if $\pi^0 \simeq \pi^0 \circ \theta$. In that case, there exists an intertwining operator $A_{\pi^0}: \pi^0 \rightarrow \pi^0 \circ \theta$. We say that $A_{\pi^0}$ is \emph{normalized} if $A^2_{\pi^0} = 1$. If $\pi^0$ is irreducible and $\theta$-stable, then it has a normalized intertwining operator, by Schur's lemma. The data of a normalized intertwining operator $A_{\pi^0}$ on $\pi^0$ are equivalent to the data of a representation $\pi$ of $\tilde{\mathbf{M}}(\mathbb{A})$ extending $\pi^0$; \cite[p.~121]{M}.

Thus, in our case, the restriction of any $\pi$ in $\Pi(\mathbf{M}(\mathbb{A}))$ to $\mathbf{M}^0(\mathbb{A})$ is $\theta$-invariant, and any $\theta$-invariant irreducible representation of $\mathbf{M}^0(\mathbb{A})$ extends to a representation in $\Pi(\mathbf{M}(\mathbb{A}))$. In practice, we will start with a $\theta$-stable irreducible representation $\pi^0$ of $\mathbf{M}^0$ and choose a normalized intertwining operator $A_{\pi^0} : \pi^0 \rightarrow \pi^0 \circ \theta$. Accordingly, we can express the spectral side of the trace formula for $\mathbf{G}$ (for $\mathbf{H}$) in terms of representations in $\Pi(\mathbf{M}(\mathbb{A}))$, and also in terms of $\theta$-stable representations of $\mathbf{M}^0(\mathbb{A})$ for the various Levi subsets $\mathbf{M}$ of $\mathbf{G}$ (of $\mathbf{H}$).

We shall also have to consider certain equivalence classes of representations in $\Pi(M(F_S))$. First, there is an action of the finite group\medskip
\begin{center}
$\Xi_{M, S} = Hom(\tilde{M}(F_S)/M^0(F_S), \mathbb{C}^{\times})$\medskip
\end{center}
on $\Pi(\tilde{M}(F_S))$, which is given by\medskip
\begin{center}
$\pi_{\xi}(m) = \pi(m)\xi(\bar{m})$,\medskip
\end{center}
where $\pi \in \Pi(\tilde{M}(F_S))$, $\xi \in \Xi_{M,S}$, and $\bar{m}$ is the projection of $m$ onto $\tilde{M}(F_S)/M^0(F_S)$. This action preserves $\Pi(M(F_S))$, and we shall write $\lbrace \Pi(M(F_S)) \rbrace$ for the space of $\Xi_{M,S}$-orbits in $\Pi(M(F_S))$. In fact, $\Pi(M(F_S))$ is the subset of $\Pi(\tilde{M}(F_S))$ on which $\Xi_{M,S}$ acts freely; see \cite[p.~28]{A89I}. Further, if $\lbrace \pi \rbrace$ is a set of representatives of $\Xi_{M,S}$-orbits in $\Pi(M(F_S))$, the map\medskip
\begin{center}
$\lbrace \pi \rbrace \rightarrow \Pi(M^0(F_S)), \pi \mapsto \pi^0$\medskip
\end{center}
is a bijection from the set of fixed-point free orbits of $\Xi_{M,S}$ in $\Pi(\tilde{M}(F_S))$ onto the set of elements in $\Pi(M^0(F_S))$ that are invariant under $\tilde{M}(F_S)/M^0(F_S)$.

We shall write $\Pi_{temp}(M(F_S))$ and $\Pi_{unit}(M(F_S))$ for the subset of representations $\pi$ in $\Pi(M(F_S))$ such that $\pi^0$ is tempered and unitary, respectively, and similarly for $\Pi_{temp}(M(\mathbb{A}))$ and $\Pi_{unit}(M(\mathbb{A}))$.

Note that the disconnected group\medskip
\[\Xi_{M, \mathbb{A}} = \lim_{\substack{ \longrightarrow\\ S}}\Xi_{M,S}\]
acts freely on each of $\Pi_{temp}(M(\mathbb{A}))$ and $\Pi_{unit}(M(\mathbb{A}))$. We shall write  $\lbrace \Pi_{unit}(M(\mathbb{A})) \rbrace$ and  $\lbrace \Pi_{temp}(M(\mathbb{A})) \rbrace$ for the sets of orbits. Again, they correspond bijectively to the sets of representations of $M^0(\mathbb{A})$ obtained by restriction.

Similar definitions can be made for $M(\mathbb{A})^1$ in place of $M(\mathbb{A})$. The terms on the spectral side of the trace formula will depend on representations in $\Pi_{unit}(M(\mathbb{A})^1)$. We shall identify a representation $\pi \in \Pi_{unit}(M(\mathbb{A})^1)$ with the corresponding orbit\medskip
\begin{center}
$\lbrace \pi_{\mu} : \mu \in i\mathfrak{a}^\ast_M \rbrace$\medskip
\end{center}
of $i\mathfrak{a}^\ast_M$ in $\Pi_{unit}(M(\mathbb{A}))$; see \cite[p.~506]{A88II} for more details.

To any representation $\pi \in \Pi(M(\mathbb{A}))$ we can associate the induced representation $\pi^G$ of $\tilde{G}(\mathbb{A})$. Let $\nu_{\pi}$ denote the infinitesimal character of its Archimedean component. We shall write $\Pi_{unit}(M(\mathbb{A})^1, t)$ for the set of representations $\pi \in \Pi_{unit}(M(\mathbb{A})^1)$ such that\medskip
\begin{center}
$\Vert Im(\nu_{\pi}) \Vert = t$,\medskip
\end{center}
where $Im(\nu_{\pi})$ is the imaginary part of $\nu_{\pi}$; see \cite[p.~515]{A88II} for more details.

As Arthur explains in \cite[p.~517]{A88II}, formula (4.4), the part of the spectral side of the trace formula for $G$ indexed by a fixed $t$ can be written as\medskip
\[ \sum_{\pi \in \lbrace \Pi_{unit}(G(\mathbb{A})^1, t)\rbrace} a^G_{disc}(\pi) f_G(\pi), \]\medskip
a finite linear combination of characters. This defines a complex-valued function\medskip
\begin{center}
$a_{disc}(\pi) := a^G_{disc}(\pi), \pi \in \lbrace\Pi_{unit}(G(\mathbb{A})^1, t)\rbrace$,\medskip
\end{center}
which is the primary global datum for the spectral side of the trace formula. Arthur then focuses on a subset of $\lbrace \Pi_{unit}(G(\mathbb{A})^1, t)\rbrace$ that contains the support of $a^G_{disc}$, defined as follows. We shall write $\Pi_{disc}(G, t)$ for the subset of $\Xi_{\mathbb{A}}$-orbits in $\lbrace \Pi_{unit}(G(\mathbb{A})^1, t)\rbrace$ that are represented by irreducible constituents of induced representations\medskip
\begin{center}
$\sigma^G_{\lambda}, M \in \mathcal{L}, \sigma \in \Pi_{unit}(M(\mathbb{A})^1, t), \lambda \in i\mathfrak{a}^\ast_M / i\mathfrak{a}^\ast_G$,\medskip
\end{center}
where $\sigma_{\lambda}$ satisfies the following two conditions:

(i) $a^M_{disc}(\sigma) \neq 0$,

(ii) There is an element $s \in W^G(\mathfrak{a}_M)_{reg}$ such that $s\sigma_{\lambda} = \sigma_{\lambda}.$

For each $L \in \mathcal{L}_{\mathbb{R}}$, we will write $\Pi_2(A_L(\mathbb{R})^0 \setminus L(\mathbb{R})) \subset \Pi_{temp}(A_L(\mathbb{R})^0 \setminus L(\mathbb{R}))$ for the set of equivalence classes of square-integrable representations of $A_L(\mathbb{R})^0 \setminus L(\mathbb{R})$. Finally, for each $M \in \mathcal{L}$, we will write $\Pi_{disc}(A_M(\mathbb{R})^0 \setminus M(\mathbb{R}))$ for the set of equivalence classes of irreducible unitary representations $\pi$ of $A_M(\mathbb{R})^0 \setminus M(\mathbb{R})$ that occur discretely, with multiplicity $m^M_{disc}(\pi)$ in\medskip
\begin{center}
$L^2(A_M(\mathbb{R})^0M(\mathbb{Q}) \setminus M(\mathbb{A})).$
\end{center}
%Laumon, p.307 and p.316.%

\subsubsection{Weighted Characters}

Let $\mathcal{H}(G(F_S))$ be the space of smooth, compactly supported functions on $G(F_S)$ whose left and right translates by $K$ span a finite dimensional space; \cite[p.~44]{A89I}.
%A 1989 I, p.44%
Weighted characters are certain linear functionals on $\mathcal{H}(G(F_S))$, defined by Arthur in terms of normalized intertwining operators; see \cite{A81}. It turns out that they describe the terms in the trace formula arising from Eisenstein series \cite{A82II}.

Let $M$ be a Levi subset of $G$.  For every $f \in \mathcal{H}(G(F_S))$ and every $\pi$ in $\Pi(M(F_S))$, Arthur defines the \emph{weighted character} $J_M(\pi_{\lambda})$ by\medskip
\begin{center}
$J_M(\pi_{\lambda}, f) =$ tr$(\mathcal{R}_M(\pi_{\lambda}, P_0) \mathcal{I}_{P_0}(\pi_{\lambda},f)),$
\end{center}
where\medskip
\begin{center}
$\mathcal{R}_M(\pi_{\lambda}, P_0) = lim_{\nu \rightarrow 0} \sum_{P \in \mathcal{P}(M)} \mathcal{R}_P(\nu, \pi_{\lambda}, P_0) \theta_P(\nu)^{-1},$\medskip
\end{center}
in the notation of \cite[Sect.~6]{A81}. They are independent of $P_0$.

It is not actually the distributions $J_M(\pi_{\lambda})$ that occur in the trace formula, but rather their integrals over $\pi$.

Suppose that $\pi \in \Pi(M(F_S))$ is such that $J_M(\pi_{\lambda}, f)$ is regular for $\lambda \in i \mathfrak{a}^\ast_M$. This holds, for example, when $\pi$ is unitary; \cite[p.~50]{A89I}. Then if $X \in \mathfrak{a}_{M,S}$, Arthur defines\medskip
\begin{center}
$J_M(\pi, f, X) = \int_{i\mathfrak{a}_{M,S}^\ast} J_M(\pi_{\lambda}, f)e^{-\lambda(X)}d\lambda$;
\end{center}
see \cite[p.~50]{A89I}.

We will focus on the case where $F = \mathbb{Q}$ and $S = \lbrace \infty \rbrace$, so that $F_S = \mathbb{R}$, and where $\pi = \pi_{\mathbb{R}}$ is unitary. Following Laumon, we will take our test functions $f_{\mathbb{R}}$ to lie in $\mathcal{H}(A_G(\mathbb{R})^0 \setminus G(\mathbb{R}))$, and so we may write\medskip
\begin{center}
$J^G_M(\pi_\mathbb{R}, f_\mathbb{R}, X) := J_M(\pi_\mathbb{R}, f_\mathbb{R}, X) = \int_{i(\mathfrak{a}^G_M)^\ast} J^G_M(\pi_{\mathbb{R},\lambda}, f_\mathbb{R})e^{-\lambda(X)}d\lambda,$\medskip
\end{center}
where we take $X \in \mathfrak{a}^G_M$. Finally, for $f' \in C^{\infty}_c(G(\mathbb{A}_f))$, we set\medskip
\begin{center}
$J^G_{M,\pi}(f_\mathbb{R}f') = \sum_{X \in \mathfrak{a}_{M,f}} J^G_M(\pi_\mathbb{R}, f_\mathbb{R}, X)f'_M(\pi_f, X);$\medskip
\end{center}
see \cite[pp.~304-5]{L}.

\subsection{Simplifications on the Spectral Side}

Our starting point is the following expression for the spectral side of the non-invariant trace formula. This expression is available for a connected $\mathbf{G}$ as well as a non-connected $G$ of the form we are considering in this paper; see \cite[Sect.~4]{A88II}. In the connected case, the summation below is over Levi subgroups $\mathbf{M}$ of $\mathbf{G}$, while in the non-connected case, the summation is over Levi subsets $M$ of $G$.

\begin{theorem} 
Let $\phi_\mathbb{R}$ be a function in $\mathcal{H}(A_G(\mathbb{R})^0 \setminus G(\mathbb{R}))$ that is very cuspidal, and let $\phi'$ be any function in $C^{\infty}_{c}(G(\mathbb{A}_f))$. Then\medskip

\begin{center}
$J^G_{spec}(\phi_\mathbb{R}\phi') = \sum_{t \geqslant 0} \sum_{M \in \mathcal{L}} \vert W^M_0\vert \vert W^G_0 \vert^{-1} \sum_{\pi \in \Pi_{disc}(M, t)} a^{M}_{disc}(\pi) J^G_{M,\pi}(\phi_\mathbb{R}\phi')$\medskip
\end{center}
where $\Pi_{disc}(M,t) \subset \Pi_{unit}(A_M(\mathbb{R})^0 \setminus M(\mathbb{A}), t)$ and\medskip
\begin{center}
$J^G_{M, \pi}(\phi_{\mathbb{R}}\phi') = \sum_{X \in \mathfrak{a}_{M,f}}J^G_M(\pi_{\mathbb{R}}, \phi_{\mathbb{R}}, s(X))\phi'_M(\pi_{f}, X)$\medskip
\end{center}
for every $M \in \mathcal{L}$ and every $\pi = \pi_{\mathbb{R}} \otimes \pi_{f} \in \Pi_{disc}(M,t)$, where\begin{center}\medskip
$s(X) = \left( \sum_qX_q \right)^G  \in \mathfrak{a}^G_M$\medskip
\end{center}
for every $X \in \mathfrak{a}_{M,f}$.
\end{theorem}

\begin{proof} Laumon \cite[pp.~305-6]{L}.
\end{proof}

We now want to write the spectral side in a more explicit form.

In \cite[Sect.~8]{A89I}, Arthur introduces the distributions\medskip
\begin{center}
$^{c}D^G_M(\pi_{\mathbb{R}}, X, \phi_\mathbb{R})$\medskip
\end{center}
for $M \in \mathcal{L}_\mathbb{R}$, $\pi_\mathbb{R} \in \Pi_{temp}(A_G(\mathbb{R})^0 \setminus M(\mathbb{R}))$, $X \in \mathfrak{a}^G_M$, and $\phi_\mathbb{R} \in \mathcal{H}(A_G(\mathbb{R})^0 \setminus G(\mathbb{R}))$.

These distributions are defined in terms of residues of certain integrals, ultimately in terms of residues of Eisenstein series. They are defined in exactly the same way whether $G$ is connected or not, and their basic properties are valid in both cases. In this paper, their usefulness stems from the following two results.

\begin{Proposition}
Let $M \in \mathcal{L}_{\mathbb{R}}$ be a Levi subset of $G$, and let $\phi_\mathbb{R}$ be a function in $\mathcal{H}(A_G(\mathbb{R})^0 \setminus G(\mathbb{R}))$ that is cuspidal. Then there exists a family of constants $C_L \in \mathbb{R}_+ (L \in \mathcal{L}_\mathbb{R}(M), L\neq G)$ that only depend on the support of $\phi_\mathbb{R}$ and have the following property. For each $\pi_\mathbb{R} \in \Pi_{temp}(A_G(\mathbb{R})^0 \setminus M(\mathbb{R}))$ and for each $X \in \mathfrak{a}^G_M$ such that $\Vert X_L\Vert > C_L$ for all $L \in \mathcal{L}_\mathbb{R}(M)$, $L \neq G$, we have\medskip
\begin{center}
$J^G_M(\pi_\mathbb{R}, \phi_\mathbb{R}, X) =$ $^{c}D^G_M(\pi_\mathbb{R}, X, \phi_\mathbb{R}).$\medskip
\end{center}
\end{Proposition}

\begin{proof} Laumon \cite[pp.306-7]{L}.
\end{proof}

The interest of replacing the distributions $J^G_M$ with the distributions $^cD^G_M$ is that the latter can be computed explicitly in virtue of the following proposition. The heart of the result is a recursion relation which shows that the numbers $^cD^G_M$ are the coefficients in a linear relation among the normalized characters of $G$ and those of its Levi subsets. The next proposition is stated in Laumon for a connected group $G$, but it remains true for a non-connected set. The only change required is that the summation will be over Levi subsets rather than Levi subgroups.

\begin{Proposition}
Let $\phi_\mathbb{R}$ be a function in $\mathcal{H}(A_G(\mathbb{R})^0 \setminus G(\mathbb{R}))$ that is cuspidal. Suppose that the support of\medskip
\begin{center}
$\phi_{\mathbb{R}, G}: \Pi_{temp}(A_G(\mathbb{R})^0 \setminus G(\mathbb{R})) \rightarrow \mathbb{C}, \sigma_\mathbb{R} \mapsto  \mathrm{tr}( \sigma_\mathbb{R}(\phi_\mathbb{R})),$\medskip
\end{center}
is contained in $\Pi_2(A_G(\mathbb{R})^0 \setminus G(\mathbb{R}))$. Then for every $M \in \mathcal{L}_\mathbb{R}$, $X \in \mathfrak{a}^G_M$, and $\pi_\mathbb{R} \in \Pi_{temp}(A_M(\mathbb{R})^0 \setminus M(\mathbb{R}))$, we have\medskip
\begin{center}
$^{c}D^G_M(\pi_\mathbb{R}, X, \phi_\mathbb{R}) \neq 0$\medskip
\end{center}
only if $\pi_\mathbb{R} \in \Pi_2(A_M(\mathbb{R})^0 \setminus M(\mathbb{R}))$. 
In that case,\medskip
\begin{center}
$^{c}D^G_M(\pi_\mathbb{R}, X, \phi_\mathbb{R}) = \sum_{\sigma_{\mathbb{R}}}$ $^{c}D^G_M(\pi_\mathbb{R}, X, \sigma_\mathbb{R}) \mathrm{tr}(\sigma_\mathbb{R}(\phi_\mathbb{R}))$\medskip
\end{center}
where $\sigma_\mathbb{R}$ ranges over $\Pi_2(A_G(\mathbb{R})^0 \setminus G(\mathbb{R}))$, and the expressions\medskip
\begin{center}
$^{c}D^G_M(\pi_\mathbb{R}, X, \sigma_\mathbb{R})$\medskip
\end{center}
for a fixed $\sigma_\mathbb{R} \in \Pi_2(A_G(\mathbb{R})^0 \setminus G(\mathbb{R}))$ and for $M \in \mathcal{L}_\mathbb{R}$, $\pi_\mathbb{R} \in \Pi_2(A_M(\mathbb{R})^0 \setminus M(\mathbb{R}))$ and $X \in \mathfrak{a}^G_M$ satisfy the following properties:

(i) $^cD^G_G(\pi_\mathbb{R}, 0, \sigma_\mathbb{R}) = \delta_{\pi_\mathbb{R}, \sigma_\mathbb{R}}$.

(ii) If $M \subset G$, $M \neq G$, for every $\gamma \in M(\mathbb{R})^1 \cap G_{reg}(\mathbb{R})$ and for every $X \in \mathfrak{a}^G_M$, we have\medskip
\begin{center}
$\sum_L \sum_{\rho_\mathbb{R}} (-1)^{dim(A_L/A_G)}$ $^cD^G_L(\rho_\mathbb{R}, X_L, \sigma_\mathbb{R})\Phi^L_M(\rho_\mathbb{R}^\vee, \gamma e^{X^L})=0,$\medskip
\end{center}
where $L$ ranges over $\mathcal{L}_\mathbb{R}(M)$, $\rho_\mathbb{R}$ ranges over $\Pi_2(A_L(\mathbb{R})^0 \setminus L(\mathbb{R}))$, and\medskip
\begin{center}
$X = X^L + X_L \in \mathfrak{a}^L_M \oplus \mathfrak{a}^G_L = \mathfrak{a}^G_M$.
\end{center}
\end{Proposition}

\begin{proof} Laumon, \cite[p.~307-8]{L}.
\end{proof}

Since the trace formula for the non-connected set $G^0 \rtimes \theta$ is just a way of writing the twisted trace formula for $G^0$, we can reformulate the foregoing proposition in terms of twisted characters. Essentially, we will replace representations of $G$ with $\theta$-stable representations of $G^0$ of the same type, traces with twisted traces, and characters with twisted characters.
This reformulation allows us to utilize the results in Morel \cite{M}.

\begin{cor}
Let $f_\mathbb{R}$ be a function in $\mathcal{H}(A_{G^0}(\mathbb{R})^0 \setminus G^0(\mathbb{R}))$ that is cuspidal. Suppose that the support of\medskip
\begin{center}
$f_{\mathbb{R}, G}: \Pi_{\theta -temp}(A_{G^0}(\mathbb{R})^0 \setminus G^0(\mathbb{R})) \rightarrow \mathbb{C}, \sigma^0_\mathbb{R} \mapsto  \mathrm{tr}(\sigma^0_\mathbb{R}(f_\mathbb{R})A_{\sigma_{\mathbb{R}}}),$\medskip
\end{center}
is contained in $\Pi_{\theta -2}(A_{G^0}(\mathbb{R})^0 \setminus G^0(\mathbb{R}))$. Then for every $M^0 \in \mathcal{L}^{\theta}_\mathbb{R}$, the set of $\theta$-stable Levi subgroups of $G^0$, for $X \in (\mathfrak{a}^{G^0}_{M^0})^{\theta}$, and for $\pi^0_\mathbb{R} \in \Pi_{\theta -temp}(A_{M^0}(\mathbb{R})^0 \setminus M^0(\mathbb{R}))$, we have\medskip
\begin{center}
$^cD^{G^0}_{M^0}(\lbrace \pi^0_\mathbb{R}, A_{\pi^0_\mathbb{R}} \rbrace, X, f_\mathbb{R}) \neq 0$\medskip
\end{center}
\textit{only if $\pi^0_\mathbb{R} \in \Pi_2(A_{M^0}(\mathbb{R})^0 \setminus M^0(\mathbb{R}))$. 
In that case},\medskip
\begin{center}
$^cD^{G^0}_{M^0}(\lbrace \pi^0_\mathbb{R}, A_{\pi^0_\mathbb{R}} \rbrace, X, f_\mathbb{R}) = \sum_{\sigma^0_{\mathbb{R}}} \sum_{A_{\sigma^0_{\mathbb{R}}}}$ $^cD^{G^0}_{M^0}(\pi_\mathbb{R}, X, \lbrace\sigma^0_\mathbb{R}, A_{\sigma^0_{\mathbb{R}}}\rbrace)  \mathrm{tr}(\sigma^0_\mathbb{R}(f_\mathbb{R}) A_{\sigma^0_{\mathbb{R}}})$\medskip
\end{center}
where $\sigma^0_\mathbb{R}$ ranges over $\Pi_{\theta -2}(A_{G^0}(\mathbb{R})^0 \setminus G^0(\mathbb{R}))$, and the expressions\medskip
\begin{center}
$^cD^{G^0}_{M^0}(\lbrace \pi^0_\mathbb{R}, A_{\pi^0_\mathbb{R}} \rbrace, X, \lbrace\sigma^0_\mathbb{R}, A_{\sigma^0_{\mathbb{R}}}\rbrace)$\medskip
\end{center}
for a fixed $\sigma^0_\mathbb{R} \in \Pi_{\theta -2}(A_{G^0}(\mathbb{R})^0 \setminus G^0(\mathbb{R}))$, fixed normalized intertwining operator $A_{\sigma^0_{\mathbb{R}}}$, $M^0 \in \mathcal{L}^{\theta}_\mathbb{R}$, $\pi_\mathbb{R} \in \Pi_{\theta -2}(A_{M^0}(\mathbb{R})^0 \setminus M^0(\mathbb{R}))$, and $X \in (\mathfrak{a}^{G^0}_{M^0})^{\theta}$ satisfy the following properties:

(i) $^cD^{G^0}_{G^0}(\lbrace \pi^0_\mathbb{R}, A_{\pi^0_\mathbb{R}} \rbrace, 0, \lbrace \sigma^0_\mathbb{R}, A_{\sigma^0_\mathbb{R}} \rbrace) = \delta_{\pi_\mathbb{R}, \sigma_\mathbb{R}}$, where $\pi_{\mathbb{R}}$ and $\sigma_{\mathbb{R}}$ are the representations of $G$ that correspond to the data $\lbrace \pi^0_\mathbb{R}, A_{\pi^0_\mathbb{R}} \rbrace$ and $\lbrace \sigma^0_\mathbb{R}, A_{\sigma^0_\mathbb{R}} \rbrace$, respectively.

(ii) If $M^0 \subset G^0$, $M^0 \neq G^0$, for every $\gamma \in M^0(\mathbb{R})^1 \cap G^0_{reg}(\mathbb{R})$ and for every $X \in (\mathfrak{a}^{G^0}_{M^0})^{\theta}$, we have\medskip
\begin{center}
$\sum_{L^0} \sum_{\rho^0_\mathbb{R}} \sum_{A_{\rho^0_{\mathbb{R}}}} (-1)^{\dim(A_{L^0}/A_{G^0})}$ $^cD^{G^0}_{L^0}(\rho^0_\mathbb{R}, X_{L^0}, \sigma^0_\mathbb{R})\Phi^{L^0}_{M^0}(\lbrace \rho_\mathbb{R}^{0\vee}, A_{\rho_{\mathbb{R}}} \rbrace, \gamma e^{X^{L^0}})=0,$\medskip
\end{center}
where $L^0$ ranges over $\mathcal{L}^{\theta}_\mathbb{R}(M^0)$, $\rho^0_\mathbb{R}$ ranges over $\Pi_{\theta -2}(A_{L^0}(\mathbb{R})^0 \setminus L^0(\mathbb{R}))$, and\medskip
\begin{center}
$X = X^{L^0} + X_{L^0} \in (\mathfrak{a}^{L^0}_{M^0})^{\theta} \oplus (\mathfrak{a}^{G^0}_{L^0})^{\theta} = (\mathfrak{a}^{G^0}_{M^0})^{\theta}$,\medskip
\end{center}
\textit{and where $\Phi^{L^0}_{M^0}$ now stands for the twisted normalized character, as defined by Morel in} \cite[p.~125]{M}.
\end{cor}

\begin{proof} This is a restatement of proposition 4.3.
\end{proof}

Note that the summation over $A_{\sigma^0_{\mathbb{R}}}$ ($A_{\rho^0_{\mathbb{R}}}$) amounts to the following. For each $\sigma^0_{\mathbb{R}}$ ($\rho^0_{\mathbb{R}}$), we choose a normalized intertwining operator $A_{\sigma^0_{\mathbb{R}}}$ ($A_{\rho^0_{\mathbb{R}}}$), and then the sum is taken over $A_{\sigma^0_{\mathbb{R}}}$ and $-A_{\sigma^0_{\mathbb{R}}}$ ($A_{\rho^0_{\mathbb{R}}}$ and $-A_{\rho^0_{\mathbb{R}}}$).

The previous results only work for tempered representations. We extend them to general unitary representations by using the following proposition of Arthur's, which is an extension of a well-known result of Vogan.

Let $\Sigma(G(F_S))$ denote the set of equivalence classes of representations of $\tilde{G}(F_S)$ that are equal to $\sigma^G_{\Lambda}$ for some $\prod_{v \in S}M_v$, with $\sigma$ a representation in\medskip
\begin{center}
$\lbrace \otimes_v \sigma_v : \sigma_v \in \Pi_{temp}(M_v(F_v)) \rbrace$\medskip
\end{center} 
and $\Lambda$ a point in $\oplus_v \mathfrak{a}^\ast_{M_v}$ that is regular in the sense that $\Lambda(\beta) \neq 0$ for every root $\beta$ of $(G, \prod_v A_{M_v})$. The elements of $\Sigma(G(F_S))$ are called \emph{standard representations}. For $\rho \in \Sigma(G(F_S))$, we have $\rho^0 \in \Sigma(G^0(F_S))$. It is well-known that $\rho^0$ has a unique irreducible quotient, and hence, $\rho$ also has a unique irreducible quotient, which we will denote $\bar{\rho}$. It is a representation in $\Pi(G(F_S))$, and $\rho \mapsto \bar{\rho}$ is a bijection from $\Sigma(G(F_S))$ onto $\Pi(G(F_S))$; see \cite[p.~40]{A89I} Thus, standard representations are fit to play the same role they play for connected groups.

\begin{Proposition}
Let $\lbrace \Pi(G(F_S))\rbrace$ and $\lbrace\Sigma(G(F_S))$ denote the set of $\Xi_{G,S}$-orbits in $\Pi(G(F_S))$ and $\Sigma(G(F_S))$, respectively. Then there are uniquely determined complex numbers\medskip
\begin{center}
$\lbrace \Delta(\pi, \rho) : \pi \in \Pi(G(F_S)), \rho \in \Sigma(G(F_S))\rbrace,$\medskip
\end{center}
such that\medskip
\begin{center}
$\mathrm{tr}(\pi) = \sum_{\rho \in \lbrace \Sigma(G(F_S))\rbrace} \Delta(\pi, \rho) \mathrm{tr}(\rho),$\medskip
\end{center}
where $\pi \in \Pi(G(F_S))$ and $\rho \in \Sigma(G(F_S))$.
\end{Proposition}

\begin{proof} See Arthur, \cite[p.~41]{A89I}.
\end{proof}
%For Laumon, these Deltas were basically just 1 or 0, except in one special case.%

\begin{Proposition}
Let $M \in \mathcal{L}_{\mathbb{R}}$ and $\phi_{\mathbb{R}} \in \mathcal{H}(A_G(\mathbb{R})^0 \setminus G(\mathbb{R}))$. Suppose that $\phi_{\mathbb{R}}$ is very cuspidal and that the support of $\phi_{\mathbb{R}, G}$ is contained in $\Pi_2(A_G(\mathbb{R})^0 \setminus G(\mathbb{R}))$. Choose a family of constants $C_L \in \mathbb{R}_+ (L \in \mathcal{L}_{\mathbb{R}}(M), L \neq G)$ as in proposition $4.2.$ Then, for each $\pi_{\mathbb{R}} \in \Pi_{unit}(A_M(\mathbb{R})^0 \setminus M(\mathbb{R}))$, and each $X \in \mathfrak{a}^G_M$ such that $\Vert X_L\Vert > C_L$ for all $L \in \mathcal{L}_\mathbb{R}(M)$, $L \neq G$, we have\medskip
\begin{center}
$J^G_M(\pi_\mathbb{R}, \phi_\mathbb{R}, X) =$ $^{c}D^G_M(\pi_\mathbb{R}, X, \phi_\mathbb{R}),$\medskip
\end{center}
where\medskip
\begin{center}
$^{c}D^G_M(\pi_\mathbb{R}, X, \phi_\mathbb{R}) := \sum_{\pi'_{\mathbb{R}}} \Delta(\pi_{\mathbb{R}}, \pi'_{\mathbb{R}})^{c}D^G_M(\pi'_\mathbb{R}, X, \phi_\mathbb{R}),$\medskip
\end{center}
\textit{where $\pi'_{\mathbb{R}}$ ranges over $\Pi_2(A_M(\mathbb{R})^0 \setminus M(\mathbb{R}))$.}
\end{Proposition}

\begin{proof} Laumon, \cite[p.~310]{L}.
\end{proof}

\begin{Proposition}
\textit{For each family of constants $C_L \in \mathbb{R}_+ (L \in \mathcal{L}_{\mathbb{R}}(M), L \neq G)$, there exits a constant $C \in \mathbb{R}_+$ with the following property. For each $\phi' \in C^{\infty}_c(G(\mathbb{A}_f))$ that is $C$-regular, each $M \ \in \mathcal{L}$, $M \neq G$, each $\pi_f \in \Pi(M(\mathbb{A}_f))$, and each $X \in \mathfrak{a}_{M,f}$ such that}
\begin{center}
$\Vert s(X)_L\Vert \leqslant C_L$
\end{center}
\textit{for at least one $L \in \mathcal{L}_{\mathbb{R}}(M)$, $L \neq G$, we have}
\begin{center}
$\phi'_M(\pi_f, X) = 0$.
\end{center}
\end{Proposition}

\begin{proof} Laumon, \cite[p.~311]{L}.
\end{proof}

So far, we have obtained the following expression for the spectral side of the trace formula.

\begin{theorem} 
Let $\phi_\mathbb{R}$ be a function in $\mathcal{H}(A_{G}(\mathbb{R})^{0} \setminus G(\mathbb{R}))$ that is highly cuspidal and stable cuspidal. Then there exists $C \in \mathbb{R}_+$ that only depends on the support of $\phi_\mathbb{R}$ and has the following property. For every function $\phi' \in C^{\infty}_{c}(G(\mathbb{A}_f))$ that is $C$-regular, we have\medskip

\begin{center}
$J^{G}_{spec}(\phi'\phi_{R}) = \sum_{t \geqslant 0} \sum_{M \in \mathcal{L}} \vert W^{M}_{0} \vert \vert W^{G}_{0} \vert^{-1} \sum_{\pi} a^{M}_{disc}(\pi) \sum_{\pi'} \Delta(\pi_{\mathbb{R}}, \pi_{\mathbb{R}}')$\medskip

$\times \sum_{X \in a_{M,f}} D^{G}_{M}(\pi_{\mathbb{R}}', s(X), \phi_{\mathbb{R}}) \phi_{M}'(\pi_{f}, X)$\medskip

\end{center}
where $\pi$ ranges over $\Pi_{disc}(M,t)$ and $\pi'$ ranges over $\Pi_{2}(A_{M}(\mathbb{R})^{0} \setminus M(\mathbb{R}))$, and the quantities $^cD^{G}_{M}$ are computed as in proposition 4.3.
\end{theorem}

Again, we can write down a version of the formula in the theorem that is indexed explicitly in terms of representations of $G^0$ and normalized intertwining operators.

\begin{cor}
Let $\phi_\mathbb{R}$ be a function in $\mathcal{H}(A_{G^0}(\mathbb{R})^{0} \setminus G^0(\mathbb{R}))$ that is highly cuspidal and stable cuspidal. Then there exists $C \in \mathbb{R}_+$ that only depends on the support of $\phi_\mathbb{R}$ and has the following property. For every function $\phi' \in C^{\infty}_{c}(G^0(\mathbb{A}_f))$ that is $C$-regular, we have\medskip

\begin{center}
$J^{G}_{spec}(\phi'\phi_{R}) = \sum_{t \geqslant 0} \sum_{M^0 \in \mathcal{L}^0} \vert W^{M}_{0} \vert \vert W^{G}_{0} \vert^{-1}$\medskip

$\times \sum_{\pi^0} a^{M^0}_{disc}(\pi^0) \sum_{A_{\pi^0_{\mathbb{R}}}}  \sum_{\pi'^0_{\mathbb{R}}} \sum_{A_{\pi'^0_{\mathbb{R}}}} \Delta(\lbrace \pi^0_{\mathbb{R}},A_{\pi^0_\mathbb{R}}\rbrace, \lbrace \pi'^0_{\mathbb{R}},A_{\pi'^0_\mathbb{R}}\rbrace ) $\medskip

$\times \sum_{X \in a_{M,f}}$ $^cD^{G^0}_{M^0}(\pi_{\mathbb{R}}', s(X), \phi_{\mathbb{R}}) \phi_{M^0}'(\pi_{f}, X)$\medskip

\end{center}
where $\pi^0$ ranges over $\Pi_{disc}(M^0,t)$ and $\pi'^0$ ranges over $\Pi_{2}(A_{M^0}(\mathbb{R})^{0} \setminus M^0(\mathbb{R}))$, and where $\Delta(\lbrace \pi^0_{\mathbb{R}},A_{\pi^0_\mathbb{R}}\rbrace, \lbrace \pi'^0_{\mathbb{R}},A_{\pi'^0_\mathbb{R}}\rbrace ):=\Delta(\pi_{\mathbb{R}},\pi'_{\mathbb{R}})$ for $\pi_{\mathbb{R}}$ and $\pi'_{\mathbb{R}}$ the representations of $M$ corresponding to the data $\lbrace\pi^0_{\mathbb{R}}, A_{\pi^0_{\mathbb{R}}} \rbrace$ and $\lbrace\pi'^0_{\mathbb{R}}, A_{\pi'^0_{\mathbb{R}}} \rbrace$, respectively, and where the quantities $^cD^{G^0}_{M^0}$ are computed as in corollary 4.4.
\end{cor}

\begin{proof} This is a restatement of theorem 4.7
\end{proof}

The final step in manipulating the spectral side is to write the distributions $^cD^G_M$ as finite Fourier expansions. This will render the spectral terms explicit enough to permit us to combine the parabolic terms for $G$ with the parabolic terms for $H$. The following lemma requires the Harish-Chandra character formula for $G$. In our case, however, we are only using it in a situation where the (twisted) character on $G^0$ is equal to the (ordinary) stable character of the norm on $\mathbf{G}$.

\begin{lemma}
For $P \in P(M)$, $\pi_\mathbb{R} \in \Pi_2(A_M(\mathbb{R})^0 \setminus M(\mathbb{R}))$ , $\sigma_\mathbb{R} \in \Pi_2(A_G(\mathbb{R})^0 \setminus G(\mathbb{R}))$, and $\Lambda \in (\mathfrak{a}^{G}_{M})^{\ast}_{\mathbb{C}}$, there exist constants $d^{G}_{P}(\pi_\mathbb{R}, \Lambda, \sigma_\mathbb{R})$ with the following properties:

(i) $d^{G}_{P}(\pi_\mathbb{R}, \Lambda, \sigma_\mathbb{R}) = 0$ for each $\Lambda$ such that the intersection of $\Lambda + \chi_{\pi^{\vee}_{\mathbb{R}}}$ with $(\chi_{\sigma_\mathbb{R}})_M$ is empty in $(\mathfrak{a}^G_M)^\ast_{\mathbb{C}} \oplus t^{anis}_M(\mathbb{C})^{\ast}$;

(ii) for each $X \in {\mathfrak{a}^G_P}^+$, we have\medskip

\begin{center}
$^{c}D^G_M(\pi_\mathbb{R}, X, \sigma_\mathbb{R}) = \sum_{\Lambda \in (\mathfrak{a}^G_M)^\ast_\mathbb{C}} d^G_P(\pi_\mathbb{R}, \Lambda, \sigma_\mathbb{R})e^{-\Lambda(X)};$\medskip
\end{center}
\textit{in particular,}
\begin{center}
$^{c}D^G_M(\pi_\mathbb{R}, X, \phi_\mathbb{R}) = \sum_{\Lambda \in (\mathfrak{a}^G_M)^\ast_\mathbb{C}} d^G_P(\pi_\mathbb{R}, \Lambda, \phi_\mathbb{R})e^{-\Lambda(X)},$\medskip
\end{center}
where we have set\medskip
\begin{center}
$d^G_P(\pi_\mathbb{R}, \Lambda, \phi_\mathbb{R}) = \sum_{\sigma_\mathbb{R}} d^G_P(\pi_\mathbb{R}, \Lambda, \sigma_\mathbb{R}) \mathrm{tr} (\sigma_\mathbb{R}(\phi_\mathbb{R}))$.\medskip
\end{center}
\end{lemma}

\begin{proof} See Laumon, \cite[p.~316-7]{L}
\end{proof}

\begin{cor}
For $P^0 \in P(M^0)$, $\pi_\mathbb{R} \in \Pi_2(A_{M^0}(\mathbb{R})^0 \setminus M^0(\mathbb{R}))$ , $\sigma_\mathbb{R} \in \Pi_2(A_{G^0}(\mathbb{R})^0 \setminus G^0(\mathbb{R}))$, and $\Lambda \in (\mathfrak{a}^{G}_{M})^{\ast}_{\mathbb{C}}$, there exist constants $d^{G^0}_{P^0}(\lbrace \pi^0_\mathbb{R}, A_{\pi^0_\mathbb{R}}  \rbrace,\Lambda,\lbrace \sigma^0_\mathbb{R}, A_{\sigma^0_\mathbb{R}}  \rbrace)$ with the following properties:\medskip

(i) $d^{G^0}_{P^0}(\lbrace \pi^0_\mathbb{R}, A_{\pi^0_\mathbb{R}}  \rbrace,\Lambda,\lbrace \sigma^0_\mathbb{R}, A_{\sigma^0_\mathbb{R}}  \rbrace) = 0$ for each $\Lambda$ such that the intersection of $\Lambda + \chi_{\pi^{\vee}_{\mathbb{R}}}$ with $(\chi_{\sigma_\mathbb{R}})_M$ is empty in $(\mathfrak{a}^G_M)^\ast_{\mathbb{C}} \oplus t^{anis}_M(\mathbb{C})^{\ast}$;

(ii) for each $X \in {\mathfrak{a}^G_P}^+$, we have\medskip

\begin{center}
$^{c}D^{G^0}_{M^0}(\lbrace \pi^0_\mathbb{R}, A_{\pi^0_\mathbb{R}}  \rbrace, X, \lbrace \sigma^0_\mathbb{R}, A_{\sigma^0_\mathbb{R}}\rbrace) = \sum_{\Lambda \in (\mathfrak{a}^G_M)^\ast_\mathbb{C}} d^{G^0}_{P^0}(\lbrace \pi^0_\mathbb{R}, A_{\pi^0_\mathbb{R}}  \rbrace,\Lambda,\lbrace \sigma^0_\mathbb{R}, A_{\sigma^0_\mathbb{R}}  \rbrace)e^{-\Lambda(X)};$\medskip
\end{center}
particular,\medskip
\begin{center}
$^{c}D^{G^0}_{M^0}(\lbrace \pi^0_\mathbb{R}, A_{\pi^0_\mathbb{R}}  \rbrace, X, \phi_\mathbb{R}) = \sum_{\Lambda \in (\mathfrak{a}^G_M)^\ast_\mathbb{C}} d^{G^0}_{P^0}(\lbrace \pi^0_\mathbb{R}, A_{\pi^0_\mathbb{R}}  \rbrace, \Lambda, \phi_\mathbb{R})e^{-\Lambda(X)},$\medskip
\end{center}
where we have set\medskip
\begin{center}
$d^{G^0}_{P^0}(\lbrace \pi^0_\mathbb{R}, A_{\pi^0_\mathbb{R}}  \rbrace, \Lambda, \phi_\mathbb{R}) = \sum_{\sigma_\mathbb{R}} \sum_{A_{\sigma_\mathbb{R}}} d^{G^0}_{P^0}(\lbrace \pi^0_\mathbb{R}, A_{\pi^0_\mathbb{R}}  \rbrace, \Lambda, \sigma_\mathbb{R})\mathrm{tr} (\sigma_\mathbb{R}(\phi_\mathbb{R})A_{\sigma_{\mathbb{R}}})$.\medskip
\end{center}
\end{cor}

\begin{proof} See Laumon, \cite[p.~316-7]{L}.
\end{proof}

\begin{theorem} 
Let $\phi_{\mathbb{R}}$ be a function in $\mathcal{H}(A_G(\mathbb{R})^0 \setminus G(\mathbb{R}))$ that is very cuspidal and stable cuspidal. Then there exists a constant $C \in \mathbb{R}_+$ that only depends on the support of $\phi_{\mathbb{R}}$ and has the following property. For every function $\phi' \in C^{\infty}_c(G(\mathbb{A}_f))$ that is strongly $C$-regular, we have\medskip

\begin{center}
$J^{G}_{spec}(\phi'\phi_{R}) = \sum_{M \in L} \vert W^{M}_{0} \vert \vert W^{G}_{0} \vert^{-1} \sum_{\pi} m^{M}_{disc}(\pi) \sum_{\pi'} \Delta(\pi_{\mathbb{R}}, \pi_{\mathbb{R}}')$\medskip

$\times \sum_{P \in P(M)} \sum_{\Lambda \in (\mathfrak{a}^G_M)^\ast_{\mathbb{C}}} d^G_P(\pi'_{\mathbb{R}}, \Lambda, \phi_{\mathbb{R}})$\medskip

$\times \sum_{X \in a_{M,f}} e^{-\Lambda(s(X))} \phi_{M}'(\pi_{f}, X)$\medskip

\end{center}
where $s(X) \in \mathfrak{a}^{G+}_P$, $\pi$ ranges over $\Pi_{disc}(A_M(\mathbb{R})^0 \setminus M(\mathbb{A}))$ and $\pi'_{\mathbb{R}}$ ranges over $\Pi_{2}(A_{M}(\mathbb{R})^{0} \setminus M(\mathbb{R})).$
\end{theorem}

\begin{cor} 
Let $f_{\mathbb{R}}$ be a function in $\mathcal{H}(A_{G^0}(\mathbb{R})^0 \setminus G^0(\mathbb{R}))$ that is very cuspidal and stable cuspidal. Then there exists a constant $C \in \mathbb{R}_+$ that only depends on the support of $f_{\mathbb{R}}$ and has the following property. For every function $f' \in C^{\infty}_c(G^0(\mathbb{A}_f))$ that is strongly $C$-regular, we have\medskip

\begin{center}
$J^{G}_{spec}(f'f_{R})$ \medskip

$= \sum_{M^0 \in L^0} \vert W^{M}_{0} \vert \vert W^{G}_{0} \vert^{-1}$\medskip

$\times \sum_{\pi} m^{M^0}_{disc}(\pi) \sum_{A_{\pi^0_{\mathbb{R}}}}  \sum_{\pi'^0_{\mathbb{R}}} \sum_{A_{\pi'^0_{\mathbb{R}}}} \Delta(\lbrace \pi^0_{\mathbb{R}},A_{\pi^0_\mathbb{R}}\rbrace, \lbrace \pi'^0_{\mathbb{R}},A_{\pi'^0_\mathbb{R}}\rbrace )$\medskip

$\times \sum_{P^0 \in P(M^0)} \sum_{\Lambda \in (\mathfrak{a}^{G^0}_{M^0})^\ast_{\mathbb{C}}} d^{G^0}_{P^0}(\lbrace \pi'^0_\mathbb{R}, A_{\pi'^0_\mathbb{R}}  \rbrace, \Lambda, \phi_\mathbb{R})$\medskip

$\times \sum_{X \in a_{M,f}} e^{-\Lambda(s(X))} \phi_{M}'(\pi_{f}, X)$\medskip

\end{center}
\textit{where $s(X) \in \mathfrak{a}^{G+}_P$, $\pi$ ranges over $\Pi_{disc}(A_{M^0}(\mathbb{R})^0 \setminus M^0(\mathbb{A}))$ and $\pi'_{\mathbb{R}}$ ranges over $\Pi_{2}(A_{M^0}(\mathbb{R})^{0} \setminus M^0(\mathbb{R}))$, and where the quantities $d^{G^0}_{P^0}$ are computed as in corollary} 4.11.
\end{cor}

The usefulness of Theorem 4.12 is due to the fact that it is easy to produce functions that are strongly $C$-regular, as the following lemma will show. The strategy is to fix a finite set of places $S$ of $\mathbb{Q}$ and make a suitable choice of a test function at the places in $S$. For any function $\phi'' \in C^{\infty}_c(G(\mathbb{A}^S_f))$, let\medskip
\begin{center}
$C(\phi'') = \sup_{M, \alpha, P, m} \lbrace \vert \alpha(H^S_{M,f}(m))\vert \rbrace,$\medskip
\end{center}
where $M$ ranges over $\mathcal{L} - \lbrace G \rbrace$, $\alpha$ over the roots of $A_M$ in $G$, $P$ over $\mathcal{P}(M)$, and $m$ over the support of $\phi_P''$ in $M(\mathbb{A}^S_f)$.

\begin{lemma}
Let $C \in \mathbb{R}_+$ and, for each $q \in S$, let $\phi_q \in C_c(G(\mathbb{Q}_q)//K_{max, q})$. Suppose that, for every collection $(\mu_q)_{q \in S}$ \textit{with} $\mu_q \in $ Supp$(f^\vee_q)$, every $M \in \mathcal{L}$, and every root $\alpha$ of $A_M$ in $G$, we have\medskip
\begin{center}
$\vert \sum_{q \in S} \alpha (\mu_{q, M}) \log q \vert > C + C(\phi'').$\medskip
\end{center}
Then, the function\medskip
\begin{center}
$\phi' = \phi'' \prod_{q \in S} \phi_q \in C^{\infty}_c(G(\mathbb{A}_f))$\medskip
\end{center}
\textit{is strongly $C$-regular. Further, for every $M \in \mathcal{L}$, every $P \in \mathcal{P}(M)$, and every $\Lambda \in (\mathfrak{a}^G_M)^\ast_{\mathbb{C}}$, we have}\medskip
\begin{center}
$\sum_{X \in \mathfrak{a}_{M,f}} e^{-\Lambda(s(X))} \phi'_M(\pi_f, X)$\medskip

$= \mathrm{tr}(I_P(\pi^S_{f, \Lambda}, \phi'')) \times \sum e^{-\Lambda(s(X_S))} f_{S,M} (\pi_S, X_S),$\medskip
\end{center}
\textit{where in the first sum $s(X) \in \mathfrak{a}^{G+}_P$}.
\end{lemma}

\begin{proof} Laumon, \cite[p.~318-9]{L}.
\end{proof}

The following theorem gives our final general expression for the spectral side of the trace formula.

Let $f^{\mathbf{G}}_{\mathbb{R}}$ be a pseudo-coefficient of the discrete series $L$-packet $\Pi$ as in subsection 2.2.2, and let $\phi^G_{\mathbb{R}}$ be a twisted pseudo-coefficient of the representation of $G^0(\mathbb{R})$ associated to $\Pi$. Then $f^{\mathbf{G}}_{\mathbb{R}}$ and $d(G)\phi^G_{\mathbb{R}} = 4\phi^G_{\mathbb{R}}$ are associated, by \cite[p.124]{M}.

Let $f^{\mathbf{H}}_{\mathbb{R}}$ be the linear combination of pseudo-coefficients of the three $L$-packets of $\mathbf{H}$ associated to $\Pi$ as in subsection 2.2.2, and let $\phi^H_{\mathbb{R}}$ be the same linear combination of twisted pseudo-coefficients of $H^0(\mathbb{R})$ associated to these three $L$-packets. Then $f^{\mathbf{H}}_{\mathbb{R}}$ and $d(H)\phi^H_{\mathbb{R}} = 4\phi^H_{\mathbb{R}}$ are associated.

We may and will assume that the functions $f^{\mathbf{G}}_{\mathbb{R}}, \phi^G_{\mathbb{R}},f^{\mathbf{H}}_{\mathbb{R}}$ and $\phi^H_{\mathbb{R}}$ are very cuspidal \cite[p.~320]{L}.

Given any function $f^p \in C_c(\mathbf{G}(\mathbb{A}^{p}_f)//K^{p})$, let $\phi^p$ be associated to $f^p$. Fix an auxiliary prime $q \neq p$ such that $K^p = K^{p,q}K_q$, where $K_q = G(\mathbb{Z}_q)$, and such that $\phi^p = \phi^{p,q}\mathbf{1}_{K_q}$, with $\phi^{p,q} \in C_c(G(\mathbb{A}^{p,q}_f)//K^{p,q})$. We may replace $\phi^p$ with $\phi^{p,q}\phi_q$, where $\phi_q$ is for now arbitrary.

Let $h^{p,q}$ be a transfer of $f^{p,q}$, as in \cite[p.~321]{L}, and let $\varphi^{p,q}$ be associated to $h^{p,q}$. Let $\varphi_q = b^H(\phi_q)$ be defined as in \cite[p.~321]{L}.

\begin{theorem} 
Fix prime numbers $p \neq q$ and test functions as above. Then there exists a constant $D \in \mathbb{R}_+$ satisfying the following property. For every function $\phi_q \in C_c(G(\mathbb{Q}_q) // K_q)$ and every integer $j > 0$ such that

(i) $\vert \alpha(\mu)\vert \log q > D$ $(\forall \alpha \in \Phi(A_T, G), \forall \mu \in$ Supp$(\phi^{\vee}_q)),$

(ii) $j > \frac{D+\vert\alpha(\mu)\vert \log q}{\log p}$ $(\forall \alpha \in \Phi(T,G), \forall \mu \in $Supp$(\phi^{\vee}_q)),$

\noindent the spectral side $J^G_{spec}(\phi)$ of the trace formula for $G$ can be written as\medskip
\begin{center}
$J^G_{spec}(\phi^G) = \sum_M \vert W^{M}_{0}\vert \vert W^{G}_{0}\vert ^{-1} J^G_M(\phi^G)$\medskip
\end{center}
\textit{where the sum is over the Levi subsets of G, and}\medskip

\begin{center}
$J^G_M(\phi^G) = \sum_{\pi}m^M_{disc}(\pi) \sum_{\pi_{\mathbb{R}}'} \Delta(\pi, \pi')$ \linebreak \medskip

$\times \sum_{P \in P(M)} \sum_{\Lambda \in (\mathfrak{a}^G_M)^{\ast}_{\mathbb{C}}} d^G_P(\pi_{\mathbb{R}}', \Lambda, \phi^G_{\mathbb{R}})\mathrm{tr}(I_P(\pi^{p,q}_{f, \Lambda}, \phi^{p,q}))$ \linebreak \medskip

$\times \sum_{\nu \in W \cdot \mu} p^{j(\Lambda(\nu^G_M)+1)} \nu(t_{\pi_p})^j$\linebreak

$\times \sum_{X_q \in \mathfrak{a}_{M, q}} e^{-\Lambda(X^G_q)} \phi_{q,M}(\pi_q, X_q)$,\medskip
\end{center}
\textit{where in the final sum, we require} $X^G_q \in j \nu^G_M logp + \mathfrak{a}^{G +}_P $.
\end{theorem}

\begin{proof} The conditions (i) and (ii) guarantee that the conditions (i) and (ii) of Laumon's remark (4.20) \cite[p.~320]{L} are satisfied for $\phi^{p,q}\phi_qb^G_j(\varphi_j)$ when $D = C + C(\phi^{p,q})$. This, in turn, guarantees that the assumptions of his lemma 4.19, which is our lemma 4.14, are satisfied.
\end{proof}

\pagebreak
We also need the analogue of theorem 4.15 for the base change $H$ of the endoscopic group $\mathbf{H}$.

\begin{theorem}
In the setting of theorem $4.15$, the spectral side $J^H_{spec}(\varphi)$ of the trace formula for $H$ can be written as\medskip
\begin{center}
$J^H_{spec}(\varphi^H) = \sum_M \vert W^{M}_{0}\vert \vert W^{H}_{0}\vert ^{-1} J^H_M(\varphi^H)$\medskip
\end{center}
\textit{where the sum is over the Levi subsets of H, and}\medskip

\begin{center}
$J^H_M(\varphi^H) = \sum_{\pi}m^M_{disc}(\pi) \sum_{\pi_{\mathbb{R}}'} \Delta(\pi, \pi')$ \linebreak

$\times \sum_{P \in \mathcal{P}(M)} \sum_{\Lambda \in (\mathfrak{a}^H_M)^{\ast}_{\mathbb{C}}} d^H_P(\pi_{\mathbb{R}}', \Lambda, \varphi^H_{\mathbb{R}}) \mathrm{tr}(I_P(\pi^{p,q}_{f, \Lambda}, \varphi^{p,q}))$ \linebreak

$\times \sum_{\nu} p^{j(\Lambda(\nu^H_M)+1)} \nu(t_{\pi_p})^j$\linebreak

$\times \sum_{X_q \in \mathfrak{a}_{M, q}} e^{-\Lambda(X^G_q)} \varphi_{q,M}(\pi_q, X_q)$,\medskip
\end{center}
where $\nu$ ranges over the (twisted) transfer of the orbit of cocharacters $W.\mu$, and where in the final sum, we require $X^H_q \in j \nu^H_M logp + \mathfrak{a}^{H +}_P $.
\end{theorem}

Finally, we can write down a version of each of theorem 4.15 and 4.16 in terms of $\theta$-stable representations of $G^0(\mathbb{R})$ and $H^0(\mathbb{R})$, respectively. In fact, for our purposes it is enough to replace the index $\pi'_{\mathbb{R}}$ in the quantity $d^G_P(\pi_{\mathbb{R}}', \Lambda, \phi^G_{\mathbb{R}})$ by $\lbrace \pi'^0_{\mathbb{R}},A_{\pi'^0_\mathbb{R}}\rbrace$, namely the data corresponding to the representation $\pi'_{\mathbb{R}}$ of $G(\mathbb{R})$, and similarly for $H$. In the next section, we will compute the quantities $d^G_P(\lbrace \pi'^0_{\mathbb{R}},A_{\pi'^0_\mathbb{R}}\rbrace, \Lambda, \phi^G_{\mathbb{R}})$ in the cases we are interested in, namely $M^0 = T^0$ for both $G$ and $H$.\medskip

Recall that we had by theorem 2.8:\medskip

\begin{center}
$N(j,f^p)= T_{e}^{G}(\phi^{G})+ T_{e}^{H}(\phi^{H}).$\medskip
\end{center}
Given the results of this section, we now have

\begin{theorem} 
In the setting of theorems $2.8$ and $4.15$, we have\medskip

\begin{center}
$T_{e}^{G}(\phi) = J_{G}^{G}(\phi^{G}) + \vert W_{0}^{T} \vert \vert W_{0}^{G}\vert ^{-1}J_{T}^{G}(\phi^{G}),$\medskip
\end{center}

\begin{center}
$T_{e}^{H}(\phi) = J_{H}^{H}(\phi^{H}) + \vert W_{0}^{T} \vert \vert W_{0}^{H} \vert ^{-1}J_{T}^{H}(\phi^{H}).$
\end{center}

Thus,\medskip

\begin{center}
$N(j, f^p) = J_{G}^{G}(\phi^{G}) + J_{H}^{H}(\phi^{H})$\medskip

\hspace{15mm} $+ \frac{1}{2} (J_{T}^{G}(\phi^{G}) + J_{T}^{H}(\phi^{H})).$\medskip
\end{center}
\end{theorem}

For the remainder of this paper, we will focus on computing the parabolic part of the trace, namely $\frac{1}{2}(J^{G}_{T}(\phi^{G}) + J^{H}_{T}(\phi^{H}))$.

\section{Computations at Infinity}

\subsection{Strategy For Computing the $d$'s}

We will begin by describing informally Laumon's method in \cite{L} for computing the quantities $d$ introduced in the previous section. For the purposes of illustrating the method, we will suppose that the $L$-packet in question has two members only.

1. Let $M$ be a Levi subgroup of $G$, and let $\rho$ be a representation of $M$. We would like to compute the coefficients $d$ in the following finite expansion:\medskip
\begin{center}
$^cD^{G}_{M}(\rho, X, f)=\sum_{\Lambda \in (\mathfrak{a}^{G}_{M})^{\ast}_{\mathbb{C}}}d^{G}_{P}(\rho, \Lambda, f)e^{-\Lambda(X)}$\medskip
\end{center}
where $f = f_{\pi_{1}}+f_{\pi_{2}}$ is a pseudo-coefficient of a discrete series $L$-packet $\Pi = \lbrace \pi_{1},\pi_{2}\rbrace$ of $G$.

2. We will first compute the left-hand side of the expression in 1 for every $\rho$ (most of these quantities will be zero). We will utilize the fact that\medskip
\begin{center}
$^cD^{G}_{M}(\rho, X, f)=\sum_{\sigma}$ $^cD^{G}_{M}(\rho, X, \sigma) \mathrm{tr}(\sigma(f)),$\medskip
\end{center}
where $\sigma$ varies over representations of $G$. Since our pseudo-coefficient $f$ only `sees' the representations $\pi_1$ and $\pi_2$, the expression above simplifies to\medskip
\begin{center}
$^cD^{G}_{M}(\rho, X, f)=$ $^cD^{G}_{M}(\rho, X, \pi_{1})+$ $^cD^{G}_{M}(\rho, X, \pi_{2}),$\medskip
\end{center}
and now our task is reduced to computing the right-hand side.

3. The second main fact we will use is that there is a recursive relation\medskip
\begin{center}
$\sum_{L \in \mathcal{L}(M)}\sum_{\rho}(-1)^{\dim{(A_{L}/A_{G})}}\Phi^{L}_{M}(\rho^{\vee},\gamma e^{X^{L}}) ^{c}D^{G}_{M}(\rho, X_{L}, \sigma)=0,$\medskip
\end{center}
where $\rho$ is a representation of $L$. Now for $G=\mathbf{GU}(2,1)$ and $M = \mathbf{T}$, and for a fixed representation $\sigma$ of $G$, this simplifies to\medskip
\begin{center}
$\Phi^{G}_{T}(\sigma^{\vee},\gamma e^{X})=
\sum_{\rho}\Theta_{\rho^{\vee}}(\gamma)$ $^{c}D^{G}_{T}(\rho, X, \sigma).$\medskip
\end{center}

4. We now take the simplified equality in 3. for each of $\pi_{1}$ and $\pi_{2}$, and add them together, to obtain the combined equality\medskip
\begin{center}
$\Phi^{G}_{T}(\pi_{1}^{\vee},\gamma e^{X}) + \Phi^{G}_{T}(\pi_{2}^{\vee},\gamma e^{X})=
\sum_{\rho}\Theta_{\rho^{\vee}}(\gamma)($$^{c}D^{G}_{T}(\rho, X, \pi_{1})+$ $^{c}D^{G}_{T}(\rho, X, \pi_{2})),$\medskip
\end{center}
which in view of 2. becomes\medskip
\begin{center}
$\Phi^{G}_{T}(\pi_{1}^{\vee},\gamma e^{X})+ \Phi^{G}_{T}(\pi_{2}^{\vee},\gamma e^{X})=
\sum_{\rho}\Theta_{\rho^{\vee}}(\gamma)$$ ^cD^{G}_{M}(\rho, X, f),$\medskip
\end{center}
for our particular choice of test function.

5. Suppose now that we are able to write the left-hand side of the last equality in 4. as a linear combination of characters of $M$. Then, by the linear independence of characters, we can simply \textit{read off} the values for the expressions $^cD^{G}_{M}(\rho, X, f)$.

In what follows, we will be using a modified version of this method, one that utilizes certain character identities relating twisted characters on $G^0(\mathbb{R})$ to ordinary characters on $\mathbf{G}(\mathbb{R})$. We will compute the left-hand side of the last equality in 4. in terms of characters on $\mathbf{G}(\mathbb{R})$, and then interpret the result of this computation as a linear combination of twisted characters on $G^0(\mathbb{R})$.

\subsection{Character Identities}
The following is the heart of our argument.

\begin{Proposition}
\textit{Let} $\mathbf{G}=\mathbf{GU}(2,1)$, \textit{let} $\mathbf{T}$ \textit{the diagonal subgroup of} $\mathbf{G}$, \textit{and let $\Pi$ be the discrete series $L$-packet of} $\mathbf{G}(\mathbb{R})$ \textit{associated with the trivial representation of} $\mathbf{G}$. \textit{We have}

\begin{center}$\Phi^{\mathbf{G}}_{\mathbf{T}}\left(\gamma e^{X}, S\Theta_{\Pi} \right)= 2\left( e^{-x} e^{i\theta}+e^{-x} e^{-i\theta}-e^{-2x}\right)$\end{center} 
\textit{for $x>0$, and}
\begin{center}$\Phi^{\mathbf{G}}_{\mathbf{T}}\left(\gamma e^{X}, S\Theta_{\Pi} \right)= 2\left( e^{x} e^{i\theta}+e^{x} e^{-i\theta}-e^{2x}\right)$\end{center} 
\textit{for $x<0$.}
\end{Proposition}

\begin{proof} This computation is an application of the Harish-Chandra character formula:
\begin{center}
$\Phi_M(\gamma, S\Theta_{\varphi}) = (-1)^{q(G)} \Delta_M(X)^{-1} \varepsilon_R(X)\zeta_{\varphi}(z)$\medskip

$\times \sum_{\omega \in \Omega_G} det(\omega) \bar{c}(Q^+_{Ad(u_M)\omega \lambda} , R^+_X) e^{(Ad(u_M)\omega \lambda)(X)},$
\end{center}
where all the notations are as specified by Morel, \cite[p.~53-4]{M}. Let $G = \mathbf{GU}(2,1)$ and $M = \mathbf{T}$. In this situation, the map exp$: \mathfrak{t}_M(\mathbb{R}) \rightarrow \mathbf{T}_M(\mathbb{R})$ is surjective by \cite[p.57]{M}, and so we can, and will, assume that any $\gamma \in \mathbf{T}_{M, reg}(\mathbb{R})$ is of the form $\gamma =$ exp$(X)$ for $X \in \mathfrak{t}_M(\mathbb{R})$. In fact, we have explicitly
\begin{center}
$X = (x + y +i\theta_1, y + i\theta_2, -x + y + i\theta_1)$
\end{center} 
with $x, y, \theta_1, \theta_2 \in \mathbb{R}$, and the claim is clear. Further, we have

$\bullet \hspace{0.5cm} R = \lbrace \pm \alpha \rbrace$, the set of real roots in $\Phi(T_M, G)$, where $\alpha(diag(a, b, d)) = a/d$;

$\bullet \hspace{0.5cm} R^+ = R \cap \Phi(T_M, B) = \lbrace \alpha \rbrace$, where we fix $B$ to be the group of upper-triangular matrices in $\mathbf{GU}(2,1)$;

$\bullet \hspace{0.5cm} R^+_X = \lbrace \alpha \in R : \alpha(X) > 0\rbrace = \lbrace \alpha \rbrace$ for $X = (x, 0, -x)$ when $x > 0$, and $R^+_X = \lbrace -\alpha\rbrace$ when $x < 0$;

$\bullet \hspace{0.5cm} q(G) = \frac{1}{2}dim_{\mathbb{R}}(X) = 2$, where $X$ is the symmetric space associated with the Shimura variety of $G$;

$\bullet \hspace{0.5cm} \Delta_M = \prod_{\alpha \in \Phi(T_M, B_M)}(e^{\alpha/2} - e^{-\alpha/2}) = 1$, since $T_M  = M = B_M$, and hence, the set $\Phi(T_M, B_M)$ is empty;

$\bullet \hspace{0.5cm} \varepsilon_R(X) = (-1)^{\vert R^+_X \cap (-R^+) \vert} = 1$ for $x > 0$, and $\varepsilon_R(X) = -1$ for $x < 0$.

$\bullet \hspace{0.5cm} \zeta_{\varphi}(z) = 1$, since we are assuming that $z \equiv 1$ throughout.

$\bullet \hspace{0.5cm}$ For $\Phi(G, T) = \lbrace \pm\alpha_1, \pm\alpha_2, \pm\alpha = \pm(\alpha_1 + \alpha_2) \rbrace$, 
we compute directly:

1)\hspace{0.5cm}$\alpha_1 = \omega(\alpha)$ for $\omega = s_{\alpha}s_{\alpha_2}s_{\alpha_1}$, hence $det(\omega) = -1$;

2)\hspace{0.5cm}$\alpha_2 = \omega(\alpha)$ for $\omega = s_{\alpha_1}$, hence $det(\omega) = -1$;

3)\hspace{0.5cm}$\alpha = \omega(\alpha)$ for $\omega = 1$, hence $det(\omega) = 1$;

4)\hspace{0.5cm}$-\alpha_1 = \omega(\alpha)$ for $\omega = s_{\alpha_1}$, hence $det(\omega) = (-1)(-1) = 1$;

5)\hspace{0.5cm}$-\alpha_2 = \omega(\alpha)$ for $\omega = s_{\alpha_2}$, hence $det(\omega) = (-1)(-1) = 1$;

6)\hspace{0.5cm}$-\alpha = \omega(\alpha)$ for $\omega = s_{\alpha}$, hence $det(\omega) = -1$;

$\bullet \hspace{0.5cm}$For 
\begin{center}
$Q^+_{\nu} = \lbrace \alpha^{\vee} \in R^{\vee} : \nu(\alpha^{\vee}) > 0 \rbrace$,
\end{center}
by fixing isomorphisms with $\mathbb{R}^2$ such that
\begin{center}
$\alpha_1 = (2,0); \alpha_1^{\vee} = (1,0); \alpha_2 = (-1,\sqrt{3}); \alpha_2^{\vee} = \frac{1}{2} (-1,\sqrt{3});$

$\alpha = (1,\sqrt{3}); \alpha^{\vee} = \frac{1}{2}(1,\sqrt{3}),$
\end{center}we compute directly:

1) $Q^+_{\alpha} = \lbrace \alpha^\vee  \rbrace$, since $\alpha(\alpha^{\vee}) > 0$;

2) $Q^+_{\alpha_1} = \lbrace \alpha^\vee  \rbrace$, since $\alpha_1(\alpha^{\vee}) > 0$;

3) $Q^+_{\alpha_2} = \lbrace \alpha^\vee  \rbrace$, since $\alpha_2(\alpha^{\vee}) > 0$;

4) $Q^+_{-\alpha} = \lbrace -\alpha^\vee  \rbrace$, since $-\alpha(-\alpha^{\vee}) > 0$;

5) $Q^+_{-\alpha_1} = \lbrace -\alpha^\vee  \rbrace$, since $-\alpha_1(-\alpha^{\vee}) > 0$;

6) $Q^+_{-\alpha_2} = \lbrace -\alpha^\vee  \rbrace$, since $-\alpha_2(-\alpha^{\vee}) > 0$.

Finally, by employing the definition in Arthur, \cite[p.~273]{AL2}, we compute that
\begin{center}
$\bar{c}(Q^+_{\alpha_1}, R^+_X) = \bar{c}(Q^+_{\alpha_2}, R^+_X) = \bar{c}(Q^+_{\alpha}, R^+_X) = 0$, $\bar{c}(Q^+_{-\alpha_1}, R^+_X) = \bar{c}(Q^+_{-\alpha_2}, R^+_X) = \bar{c}(Q^+_{-\alpha}, R^+_X) = 2$,
\end{center}
for $x > 0$; and
\begin{center}
$\bar{c}(Q^+_{\alpha_1}, R^+_X) = \bar{c}(Q^+_{\alpha_2}, R^+_X) = \bar{c}(Q^+_{\alpha}, R^+_X) = 2$, $\bar{c}(Q^+_{-\alpha_1}, R^+_X) = \bar{c}(Q^+_{-\alpha_2}, R^+_X) = \bar{c}(Q^+_{-\alpha}, R^+_X) = 0$,
\end{center}
for $x < 0$.

The set-up of the Harish-Chandra formula requires that $(\lambda - \rho)(X) = \xi(X)$, and since in this paper we are working with the trivial local system, hence the trivial algebraic representation of $\mathbf{GU}(2,1)$, the requirement is that $\lambda(X) = \rho(X) = \alpha(X)$; see \cite[p.~54]{M}. Thus, in the formula above, as the sum ranges over $\Omega_G$, the quantity Ad$(u_M)\omega\lambda$ will range over the roots in $\Phi(G, T)$. Putting all of the above pieces together, we obtain

\begin{center}$\Phi^{G}_{T}\left(\gamma, S\Theta_{\Pi} \right)= 2e^{-\alpha_1(X)} + 2e^{-\alpha_2(X)} +(-1)2e^{-\alpha(X)} = 2\left( e^{-x} e^{-i\theta}+e^{-x} e^{i\theta}-e^{-2x}\right)$,\end{center} 
for $X = (x + y +i\theta_1, y + i\theta_2, -x + y + i\theta_1)$, and $\theta := \theta_1 - \theta_2$, when $x > 0$; and,

\begin{center}$\Phi^{G}_{T}\left(\gamma, S\Theta_{\Pi} \right)= (-1)\left( (-1)2e^{\alpha_1(X)} + (-1)2e^{\alpha_2(X)} + 2e^{\alpha(X)}\right) $ \medskip
$= 2\left( e^{x} e^{i\theta}+e^{x} e^{-i\theta}-e^{2x}\right)$,\end{center} 
for $X = (x + y +i\theta_1, y + i\theta_2, -x + y + i\theta_1)$, and $\theta := \theta_1 - \theta_2$, when $x < 0$.
\end{proof}

\begin{Proposition}
\textit{Let} $\mathbf{H}=\mathbf{G}(\mathbf{U}(1,1)\times \mathbf{U}(1))$, \textit{let} $\mathbf{T}$ \textit{the diagonal subgroup of} $\mathbf{H}$, \textit{and denote the three $L$-packets of} $\mathbf{H}$ \textit{associated to the given $L$-packet $\Pi$ of} $\mathbf{G}$ \textit{by} $\Pi(\rho^{+})$, $\Pi(\rho^{-})$, \textit{and} $\Pi(\rho^{0})$. \textit{We have}\medskip

\begin{center}
$\Phi^{\mathbf{H}}_{\mathbf{T}}\left(\gamma, S\Theta_{\Pi(\rho^{+})} \right)= 2e^{-x} e^{i\theta}\mu(\gamma)^{-1}$\medskip

$\Phi^{\mathbf{H}}_{\mathbf{T}}\left(\gamma, S\Theta_{\Pi(\rho^{-})} \right)= 2e^{-x} e^{-i\theta}\mu(\gamma)^{-1}$\medskip

$\Phi^{\mathbf{H}}_{\mathbf{T}}\left(\gamma, S\Theta_{\Pi(\rho^{0})} \right)= 2e^{-2x}\mu(\gamma)^{-1}$\medskip
\end{center} 
\textit{for $x>0$, and}\medskip
\begin{center}
$\Phi^{\mathbf{H}}_{\mathbf{T}}\left(\gamma, S\Theta_{\Pi(\rho^{+})} \right)= 2e^{x} e^{-i\theta}\mu(\gamma)^{-1}$\medskip

$\Phi^{\mathbf{H}}_{\mathbf{T}}\left(\gamma, S\Theta_{\Pi(\rho^{-})} \right)= 2e^{x} e^{i\theta}\mu(\gamma)^{-1}$\medskip

$\Phi^{\mathbf{H}}_{\mathbf{T}}\left(\gamma, S\Theta_{\Pi(\rho^{0})} \right)= 2e^{2x}\mu(\gamma)^{-1}$\medskip
\end{center} 
\textit{for $x<0$}.
\end{Proposition}

\begin{proof} The computation is similar to that in the proof of proposition 5.1. The main difference is that, this time, $\lambda - \rho$ will not be set equal to the trival root, since the representations of $\mathbf{H}(\mathbb{R})$ that transfer to the trivial representation of $\mathbf{G}(\mathbb{R})$ are not trivial. Rather, $\lambda - \rho$ will be equal to the highest weight of each representation, as follows. We will follow Rogawski's method of labeling the representations; in Rogawski's notation, we will take $a = 1, b = 0, c = -1$, since we want our representations on $\mathbf{H}(\mathbb{R})$ to correspond to the trivial one on $\mathbf{G}(\mathbb{R})$; see \cite[p.178]{R}. We then have the following two 1-dimensional representations of $\mathbf{T}(\mathbb{R})$:

(1) $\xi(0, 1, -1)(\gamma) = (det_0(\gamma))^{-1 - t -1}(det(\gamma))^1$

$=(e^{2i\theta_1})^{-t -2}e^{2i\theta_1+i\theta_2}$

$=e^{i(\theta_2 - \theta_1)}(e^{i\theta_1})^{-(2t+1)}$

$=\mu^{-1}(\gamma)e^{-i\theta}$.

Thus, we have
\begin{center}
$\lambda(X) = (\Lambda + \rho_B)(X) = (i(\theta_2-\theta_1) + ln(\mu(\gamma)^{-1})) + x$.
\end{center}

(2) $\xi(1,-1,0)(\gamma) = (det_0(\gamma))^{1-t}(det(\gamma))^{-1}$

$=\mu^{-1}(\gamma)e^{i\theta}$.

Thus, we have
\begin{center}
$\lambda(X) = (\Lambda + \rho_B)(X) = (i(\theta_1-\theta_2) + ln(\mu(\gamma)^{-1})) + x$.
\end{center}

Further, write $\gamma = (g, \lambda)$ where $\lambda$ is the $U(1)$-component and $g$ is the $U(1,1)$-component. We have the following 2-dimensional representation of $T$:

(3) $\xi(a,b,c)(\gamma) = ((det^{\otimes \frac{1}{2}(a + c - 2t)} \otimes std) \otimes det^{\otimes b})(g, \lambda)$

$=(det^{-t-1} \otimes std) \otimes 1)(g, \lambda)$.

Now, we have $det(g) = e^{x+i\theta_1}e^{-x+i\theta_1} = e^{2i\theta_1}$ and $\mu(\gamma) = (e^{i\theta_1})^{2t+1}$, so
\begin{center}
$det^{-t-1}(g) = (e^{2i\theta})^{-t-1} = e^{-2ti\theta_1 -i\theta_1}e^{-i\theta_1}$

$= (e^{i\theta_1})^{-(2t+1)}e^{-i\theta_1}$

$= \mu(\gamma)^{-1}e^{-i\theta_1}$.
\end{center}

Thus, the highest weight $\Lambda$ of the representation (3) will correspond to the linear form
\begin{center}
$X \mapsto (x + i\theta_1) + (-i\theta_1 + ln (\mu(\gamma)^{-1}) = x + ln (\mu(\gamma)^{-1})$.
\end{center}
Finally, $\rho_B(X) = \frac{1}{2}\alpha(X) = \frac{1}{2}(2x) = x$, and hence, all told, 
\begin{center}
$\lambda(X) = (\Lambda + \rho_B)(X) = (x + ln (\mu(\gamma)^{-1})) + x = 2x + ln (\mu(\gamma)^{-1}).$
\end{center}

Given what we have said in (1), (2), and (3), respectively, it follows that
\begin{center}
$\Phi(\gamma, S\Theta_{\rho(0,1,-1)})$
$= (-1)(1)^{-1}(1)(1)
\times \left( \det(1)(0)e^{-i\theta + x + ln(\mu(\gamma)^{-1})}
+ \det(-1)(2)e^{i\theta - x + ln(\mu(\gamma)^{-1})}\right)$

$= 2e^{-x+i\theta}\mu(\gamma)^{-1}$
\end{center}
when $x>0$;
\begin{center}
$\Phi(\gamma, S\Theta_{\rho(1,-1,0)})$
$= (-1)(1)^{-1}(1)(1)\times \left(\det(1)(0)e^{i\theta + x + ln(\mu(\gamma)^{-1})}
+ \det(-1)(2)e^{-i\theta - x + ln(\mu(\gamma)^{-1})}\right)$

$= 2e^{-x-i\theta}\mu(\gamma)^{-1}$
\end{center}
when $x>0$;
\begin{center}
$\Phi(\gamma, S\Theta_{\rho(1,0,-1)})$
$= (-1)(1)^{-1}(1)(1) \times \left( \det(1)(0)e^{2x + ln(\mu(\gamma)^{-1})}
+ \det(-1)(2)e^{-2x + ln(\mu(\gamma)^{-1})} \right)$

$= 2e^{-2x}\mu(\gamma)^{-1}$
\end{center}
when $x>0$. Similarly,
\begin{center}
$\Phi(\gamma, S\Theta_{\rho(0,1,-1)})$
$= (-1)(1)^{-1}(1)(-1)
\times \left( \det(1)(2)e^{-i\theta + x + ln(\mu(\gamma)^{-1})}
+ \det(-1)(0)e^{i\theta - x + ln(\mu(\gamma)^{-1})}\right)$

$= 2e^{x-i\theta}\mu(\gamma)^{-1}$
\end{center}
when $x<0$;
\begin{center}
$\Phi(\gamma, S\Theta_{\rho(1,-1,0)})$
$= (-1)(1)^{-1}(1)(-1)\times \left(\det(1)(2)e^{i\theta + x + ln(\mu(\gamma)^{-1})}
+ \det(-1)(0)e^{-i\theta - x + ln(\mu(\gamma)^{-1})}\right)$

$= 2e^{x+i\theta}\mu(\gamma)^{-1}$
\end{center}
when $x<0$;
\begin{center}
$\Phi(\gamma, S\Theta_{\rho(1,0,-1)})$
$= (-1)(1)^{-1}(1)(-1) \times \left( \det(1)(2)e^{2x + ln(\mu(\gamma)^{-1})}
+ \det(-1)(0)e^{-2x + ln(\mu(\gamma)^{-1})} \right)$

$= 2e^{2x}\mu(\gamma)^{-1}$
\end{center}
when $x<0$.
\end{proof}

For a $\theta$-stable representation $\pi^0$ of $G^0(\mathbb{R})$, let $\Phi^{G^0}_{M^0}\left( \hspace{5mm} ,\lbrace \pi^0, A_{\pi^0}\rbrace\right)$ denote the normalized twisted character, as defined by Morel \cite[p.~121, 125]{M}, and similarly for $H^0(\mathbb{R})$. The datum $A_{\pi^0}$ is redundant as far as the definition of the character is concerned, but we shall need it for the purposes of indexing sums when we view twisted characters of $G^0(\mathbb{R})$ as characters of $G(\mathbb{R})$. We shall have to do this in order to incorporate the following crucial character identity into the framework of the trace formula for $G$ that we set up in section 4.

\begin{Proposition}
\textit{We have for the $\theta$-discrete representation $\pi^0$ of $G^0(\mathbb{R})$ corresponding to the discrete series $L$-packet $\Pi$ of} $\mathbf{G}(\mathbb{R})$\medskip

\begin{center}
$\Phi^{G^0}_{T^0}\left(\gamma e^{X}, \lbrace \pi^0, A_{\pi^0}\rbrace \right)=\varepsilon(\mathbf{G})\Phi^{\mathbf{G}}_{\mathbf{T}}\left(N(\gamma e^{X}), S\Theta_{\Pi} \right)$,\medskip
\end{center}
\textit{where} $\varepsilon(\mathbf{G}) = 1$.
\end{Proposition}

\begin{proof} This follows from theorems 8.1.2 and 8.1.3 of \cite{M}, and from the fact that $\varepsilon(\mathbf{G})$ is the familiar sign defined by\medskip
\begin{center}
$\varepsilon(\mathbf{G}) = (-1)^{q(X)}$,\medskip
\end{center}
where $q(X) = \frac{1}{2} \dim(X)$, with $\dim(X)$ the real dimension of the symmetric space associated to $\mathbf{G}$; see \cite[p.~29]{Cl2}.
\end{proof}

\begin{Proposition}
\textit{We have for the $\theta$-discrete representation $\pi^0$ of $H^0(\mathbb{R})$ corresponding to the discrete series $L$-packet $\Pi$ of} $\mathbf{H}(\mathbb{R})$\medskip
\begin{center}
$\Phi^{H^0}_{T^0}\left(\gamma e^{X}, \lbrace \pi^0, A_{\pi^0}\rbrace \right)=\varepsilon(\mathbf{H})\Phi^{\mathbf{H}}_{\mathbf{T}}\left(N(\gamma e^{X}), S\Theta_{\Pi} \right)$,\medskip
\end{center}
\textit{where} $\varepsilon(\mathbf{H}) = -1.$
\end{Proposition}

\begin{theorem} 
We have\medskip
\begin{center}
$\Phi^{G^0}_{T^0}\left(\gamma e^{X}, \lbrace\pi^0, A_{\pi^0}\rbrace \right)$\medskip

$= \sum_{\rho^0 \in \Pi_{\theta-2}(A_{T^0}(\mathbb{R})^{0}\setminus T^0(\mathbb{R}))} \sum_{A_{\rho^0}}      \Theta_{\lbrace\rho^{0\vee}, A_{\rho^{0\vee}}\rbrace}(\gamma)$ $^cD^{G^0}_{T^0}(\lbrace\rho^0, A_{\rho^0}\rbrace, X, \lbrace\pi, A_{\pi^0}\rbrace)$,\medskip
\end{center}
where $\pi^0$ is a $\theta$-stable representation of $G^0(\mathbb{R})$; and\medskip
\begin{center}
$\Phi^{H^0}_{T^0}\left(\gamma e^{X}, \lbrace\pi^0, A_{\pi^0}\rbrace \right)$\medskip

$= \sum_{\rho^0 \in \Pi_{\theta-2}(A_{T^0}(\mathbb{R})^{0}\setminus T^0(\mathbb{R}))} \sum_{A_{\rho^0}}      \Theta_{\lbrace\rho^{0\vee}, A_{\rho^{0\vee}}\rbrace}(\gamma)$ $^cD^{H^0}_{T^0}(\lbrace\rho^0, A_{\rho^0}\rbrace, X, \lbrace\pi, A_{\pi^0}\rbrace)$,\medskip
\end{center}
where $\pi^0$ is a $\theta$-stable representation of $H^0(\mathbb{R})$.
\end{theorem}

\begin{proof} These are special cases of the recursion formula in proposition 4.3.
\end{proof}

\noindent It is perhaps worth emphasizing, once again, that we are writing the recursion formula for $G(\mathbb{R}) = G^0(\mathbb{R}) \rtimes \theta$, and that the double indexing has the effect of summing over all the admissible representations of $T(\mathbb{R})$ of the appropriate type.

\begin{cor}
We have, in the notation of theorem 5.4,\medskip
\begin{center}
$\Phi^{\mathbf{G}}_{\mathbf{T}}\left(N(\gamma e^{X}), S\Theta_{\Pi} \right)$\medskip

$= \sum_{\rho^0 \in \Pi_{\theta-2}(A_{T^0}(\mathbb{R})^{0}\setminus T^0(\mathbb{R}))} \sum_{A_{\rho^0}}      \Theta_{\lbrace\rho^{0\vee},  \hspace{1mm} A_{\rho^{0\vee}}\rbrace}(\gamma)$ $^cD^{G^0}_{T^0}(\lbrace\rho^0, A_{\rho^0}\rbrace, X, \lbrace\pi, A_{\pi^0}\rbrace).$\medskip
\end{center}
\end{cor}

\begin{proof} This follows directly from Proposition 5.3 and Theorem 5.5.
\end{proof}

\noindent Thus, we are working inside the recursion formula for $G = G^0 \rtimes \theta$, but we use the computation from the original unitary group $\mathbf{G}$ as input. We will interpret the stable character $\Phi^{\mathbf{G}}_{\mathbf{T}}\left(N(\gamma e^{X}), S\Theta_{\Pi} \right)$ of $\mathbf{G}$ as a linear combination of characters of $T^0 \rtimes \theta$, and use this interpretation to compute the $^cD^{G^0}_{T^0}(\lbrace\rho^0, A_{\rho^0}\rbrace, X, \lbrace\pi, A_{\pi^0}\rbrace)$.

\begin{cor} 
\textit{We have, in the obvious notation,}\medskip
\begin{center}
$\Phi^{\mathbf{H}}_{\mathbf{T}}\left(N(\gamma e^{X}), S\Theta_{\Pi(\rho^+)} \right)$\medskip

$= -\sum_{\rho^0 \in \Pi_{\theta-2}(A_{T^0}(\mathbb{R})^{0}\setminus T^0(\mathbb{R}))} \sum_{A_{\rho^0}}      \Theta_{\lbrace\rho^{0\vee},  \hspace{1mm} A_{\rho^{0\vee}}\rbrace}(\gamma)$ $^cD^{G^0}_{T^0}(\lbrace\rho^0, A_{\rho^0}\rbrace, X, \lbrace\pi, A_{\pi^{+,0}}\rbrace)$;\medskip
\end{center}
\textit{and}
\begin{center}
$\Phi^{\mathbf{H}}_{\mathbf{T}}\left(N(\gamma e^{X}), S\Theta_{\Pi(\rho^-)} \right)$\medskip

$=-\sum_{\rho^0 \in \Pi_{\theta-2}(A_{T^0}(\mathbb{R})^{0}\setminus T^0(\mathbb{R}))} \sum_{A_{\rho^0}}      \Theta_{\lbrace\rho^{0\vee},  \hspace{1mm} A_{\rho^{0\vee}}\rbrace}(\gamma)$ $^cD^{G^0}_{T^0}(\lbrace\rho^0, A_{\rho^0}\rbrace, X, \lbrace\pi, A_{\pi^{-,0}}\rbrace)$;\medskip
\end{center}
\textit{and}\medskip
\begin{center}
$\Phi^{\mathbf{H}}_{\mathbf{T}}\left(N(\gamma e^{X}), S\Theta_{\Pi(\rho^0)} \right)$\medskip

$=-\sum_{\rho^0 \in \Pi_{\theta-2}(A_{T^0}(\mathbb{R})^{0}\setminus T^0(\mathbb{R}))} \sum_{A_{\rho^0}}      \Theta_{\lbrace\rho^{0\vee},  \hspace{1mm} A_{\rho^{0\vee}}\rbrace}(\gamma)$ $^cD^{G^0}_{T^0}(\lbrace\rho^0, A_{\rho^0}\rbrace, X, \lbrace\pi, A_{\pi^{0,0}}\rbrace)$.
\end{center}
\end{cor}

\begin{proof} This follows directly from Proposition 5.4 and Theorem 5.5.
\end{proof}

\subsection{Explicit Computations}

We will now compute explicitly the expressions on the left-hand side of the identities in corollaries 5.6 and 5.7. We have\medskip
\begin{center}
$G^{0} = R_{\mathbb{C/R}}(\mathbb{G}_{m, \mathbb{C}} \times \mathbf{GL}_{3, \mathbb{C}})$\medskip
\end{center}
and \medskip
\begin{center}
$T^{0} = R_{\mathbb{C/R}}\mathbb{G}_{m,\mathbb{C}} \times (R_{\mathbb{C/R}}(\mathbb{G}_{m,\mathbb{C}}))^3$\medskip
\end{center}

\noindent Let $\gamma e^{X} \in T^{0}(\mathbb{R})$. We can write\medskip
\begin{center}
$\gamma e^{X}=(e^{i\theta_{0}},(diag(e^{i \theta_{1}}, e^{i \theta_{2}}, e^{i \theta_{3}}))(e^{x_{0}}, diag(e^{x_{1}}, e^{x_{2}}, e^{x_{3}}))$.\medskip
\end{center}

\noindent The action of the nontrivial element of Gal$(\mathbb{C}/\mathbb{R})$ is given by the involution\medskip
\begin{center}
$(\lambda, g)\mapsto (\bar{\lambda}, \bar{\lambda}A_{2,1}$$^{t}\bar{g}^{-1}A_{2,1}^{-1})   $,\medskip
\end{center}
where $A_{2,1}$ is the matrix used to define the particular form of the group $\mathbf{GU}(2,1)$ we are using, namely
\begin{center}
\[ \left( \begin{array}{ccc}
0 & 0 & 1 \\
0 & 1 & 0 \\
1 & 0 & 0 \end{array} \right)\]\medskip

\end{center}

\noindent Since our character formula will ignore the similitude factor, we may as well assume that the similitude component $e^{i\theta_{0}}e^{x_{0}}$ of $\gamma e^{X}$ is trivial. Then\medskip
\begin{center}
$N(\gamma e^{X})= \gamma e^{X} \theta (\gamma e^{X})$\linebreak

$= diag(e^{x_{1}-x_{3}}, 1, e^{x_{3}-x_{1}})diag(e^{i(\theta_{1}+\theta_{3})}, e^{2i\theta_{2}}, e^{i(\theta_{1}+\theta_{3})})$.\medskip
\end{center}

Let\medskip
\begin{center}
$\phi^G = \phi_{\pi^0},$\medskip
\end{center}
where $\phi_{\pi^0}$ is a twisted pseudo-coefficient of the $\theta$-discrete representation on $G^0(\mathbb{R})$ associated with the discrete series $L$-packet $\Pi$ on $\mathbf{G}(\mathbb{R})$ considered in this paper; see \cite[p.~124]{M}. Similarly, let $\phi^H$ be the following linear combination of pseudo-coefficients of the  three representations of $H^0(\mathbb{R})$ associated with the three discrete series $L$-packets on $\mathbf{H}(\mathbb{R})$ considered above:\medskip
\begin{center}
$\phi^H = \phi_{+}+\phi_{-}-\phi_{0}.$\medskip
\end{center}

The following theorems follow at once from the results in the previous subsection.
\begin{theorem} Let $\phi^G$ be as above.

(i) For $x_{1} > x_{3}$, we have\medskip
\begin{center}
$^cD^{G^0}_{T^0}(\lbrace\rho^0, A_{\rho^0}\rbrace, X, \phi^G) = 2e^{x_{3}-x_{1}}$,\medskip
\end{center}
when $\Theta_{\lbrace\rho^{0\vee},  \hspace{1mm} A_{\rho^{0\vee}}\rbrace}(\gamma) = e^{i(\theta_{1}+\theta_{3})}$ or $e^{-i(\theta_{1}+\theta_{3})}$; and\medskip
\begin{center}
$^cD^{G^0}_{T^0}(\lbrace\rho^0, A_{\rho^0}\rbrace, X, \phi^G) = -2e^{2(x_{3}-x_{1})}$,\medskip
\end{center}
when $\Theta_{\lbrace\rho^{0\vee},  \hspace{1mm} A_{\rho^{0\vee}}\rbrace}(\gamma) = 1$.

(ii) For $x_{1} < x_{3}$, we have\medskip
\begin{center}
$^cD^{G^0}_{T^0}(\lbrace\rho^0, A_{\rho^0}\rbrace, X, \phi^G)= 2e^{x_{1}-x_{3}}$,\medskip
\end{center}
when $\Theta_{\lbrace\rho^{0\vee},  \hspace{1mm} A_{\rho^{0\vee}}\rbrace}(\gamma) = e^{i(\theta_{1}+\theta_{3})}$ or $e^{-i(\theta_{1}+\theta_{3})}$; and\medskip
\begin{center}
$^cD^{G^0}_{T^0}(\lbrace\rho^0, A_{\rho^0}\rbrace, X, \phi^G) = -2e^{2(x_{1}-x_{3})}$,\medskip
\end{center}
when $\Theta_{\lbrace\rho^{0\vee},  \hspace{1mm} A_{\rho^{0\vee}}\rbrace}(\gamma) = 1$.
\end{theorem}

\begin{cor}
Let $\phi^G$ be as above.

(i) For $P^0 = P_{u}$ the parabolic subset of upper triangular matrices in $G^0$ (i.e. $x_{1} - x_{3} >0$), we have\medskip
\begin{center}
$d^{G^0}_{P^0}(\lbrace\rho^0, A_{\rho^0}\rbrace, \Lambda, \phi^G) = 2$,\medskip
\end{center}
when $\Theta_{\lbrace\rho^{0\vee},  \hspace{1mm} A_{\rho^{0\vee}}\rbrace}(\gamma) = e^{i(\theta_{1}+\theta_{3})}$ or $e^{-i(\theta_{1}+\theta_{3})}$, and  $\Lambda=(1,0,-1)$; and\medskip
\begin{center}
$d^{G^0}_{P^0}(\lbrace\rho^0, A_{\rho^0}\rbrace, \Lambda, \phi^G)= -2$,\medskip
\end{center}
when $\Theta_{\lbrace\rho^{0\vee},  \hspace{1mm} A_{\rho^{0\vee}}\rbrace}(\gamma) = 1$, and  $\Lambda=(2,0,-2)$; $d^{G}_{P}(\lbrace\rho^0, A_{\rho^0}\rbrace, \Lambda, \phi)=0$ otherwise.

(ii) For $P^0 = P_{d}$ the parabolic subset of lower triangular matrices in $G^0$ (i.e. $x_{1} - x_{3}<0$),we have\medskip
\begin{center}
$d^{G^0}_{P^0}(\lbrace\rho^0, A_{\rho^0}\rbrace, \Lambda, \phi^G) = 2$,\medskip
\end{center}
when $\Theta_{\lbrace\rho^{0\vee},  \hspace{1mm} A_{\rho^{0\vee}}\rbrace}(\gamma) = e^{i(\theta_{1}+\theta_{3})}$ or $e^{-i(\theta_{1}+\theta_{3})}$, and  $\Lambda=(-1,0,1)$; and\medskip
\begin{center}
$d^{G^0}_{P^0}(\lbrace\rho^0, A_{\rho^0}\rbrace, \Lambda, \phi^G)=-2$,\medskip
\end{center}
when $\Theta_{\lbrace\rho^{0\vee},  \hspace{1mm} A_{\rho^{0\vee}}\rbrace}(\gamma) = 1$, and  $\Lambda=(-2,0,2)$; $d^{G}_{P}(\rho, \Lambda, \phi)=0$ otherwise.
\end{cor}

\begin{theorem} Let $\phi^H$ be as above.

(i) For $x_{1} > x_{3}$, we have\medskip
\begin{center}
$^cD^{H^0}_{T^0}(\lbrace\rho^0, A_{\rho^0}\rbrace, X, \phi^H) = -2e^{x_{3}-x_{1}}$, \medskip
\end{center}
\textit{when} $\Theta_{\lbrace\rho^{0\vee},  \hspace{1mm} A_{\rho^{0\vee}}\rbrace}(\gamma) = e^{i(\theta_{1}+\theta_{3})}\mu(\gamma)^{-1}$;\medskip
\begin{center}
$^cD^{H^0}_{T^0}(\lbrace\rho^0, A_{\rho^0}\rbrace, X, \phi^H)= -2e^{x_{3}-x_{1}}$,\medskip
\end{center}
when $\Theta_{\lbrace\rho^{0\vee},  \hspace{1mm} A_{\rho^{0\vee}}\rbrace}(\gamma) =e^{-i(\theta_{1}+\theta_{3})}\mu(\gamma)^{-1}$; and\medskip
\begin{center}
$^cD^{H^0}_{T^0}(\lbrace\rho^0, A_{\rho^0}\rbrace, X, \phi^H)= 2e^{x_{3}-x_{1}}$, \medskip
\end{center}
when $\Theta_{\lbrace\rho^{0\vee},  \hspace{1mm} A_{\rho^{0\vee}}\rbrace}(\gamma) =\mu(\gamma)^{-1}$.\medskip

\pagebreak
(ii) For $x_{1} < x_{3}$, we have\medskip
\begin{center}
$^cD^{H^0}_{T^0}(\lbrace\rho^0, A_{\rho^0}\rbrace, X, \phi^H)= -2e^{x_{3}-x_{1}}$, \medskip
\end{center}
when $\Theta_{\lbrace\rho^{0\vee},  \hspace{1mm} A_{\rho^{0\vee}}\rbrace}(\gamma) = e^{i(\theta_{1}+\theta_{3})}\mu(\gamma)^{-1}$;\medskip
\begin{center}
$^cD^{H^0}_{T^0}(\lbrace\rho^0, A_{\rho^0}\rbrace, X, \phi^H)= -2e^{x_{3}-x_{1}}$,\medskip
\end{center}
when $\Theta_{\lbrace\rho^{0\vee},  \hspace{1mm} A_{\rho^{0\vee}}\rbrace}(\gamma) =e^{-i(\theta_{1}+\theta_{3})}\mu(\gamma)^{-1}$; and\medskip
\begin{center}
$^cD^{H^0}_{T^0}(\lbrace\rho^0, A_{\rho^0}\rbrace, X, \phi^H) = 2e^{x_{3}-x_{1}}$, \medskip
\end{center}
when $\Theta_{\lbrace\rho^{0\vee},  \hspace{1mm} A_{\rho^{0\vee}}\rbrace}(\gamma) =\mu(\gamma)^{-1}$.\medskip
\end{theorem}

\begin{cor}
Let $\phi^H$ be as above.

(i) For $P^0 = P_{u}$ the parabolic subset of upper triangular matrices in $H^0$ (i.e. $x_{1} - x_{3} >0$), we have\medskip
\begin{center}
$d^{H^0}_{P^0}(\lbrace\rho^0, A_{\rho^0}\rbrace, \Lambda, \phi^H)= -2$,\medskip
\end{center}
when $\Theta_{\lbrace\rho^{0\vee},  \hspace{1mm} A_{\rho^{0\vee}}\rbrace}(\gamma) = e^{i(\theta_{1}+\theta_{3})} \mu(\gamma)^{-1}$ or $e^{-i(\theta_{1}+\theta_{3})} \mu(\gamma)^{-1}$, and  $\Lambda=(1,0,-1)$; and\medskip
\begin{center}
$d^{H^0}_{P^0}(\lbrace\rho^0, A_{\rho^0}\rbrace, \Lambda, \phi^H) = 2$,\medskip
\end{center}
when $\Theta_{\lbrace\rho^{0\vee},  \hspace{1mm} A_{\rho^{0\vee}}\rbrace}(\gamma) =  \mu(\gamma)^{-1}$, and  $\Lambda=(2,0,-2)$; $d^{G}_{P}(\rho, \Lambda, \phi)=0$ otherwise.

(ii) For $P^0 = P_{d}$ the parabolic subset of lower triangular matrices in $H$ (i.e. $x_{1} - x_{3}<0$), we have\medskip
\begin{center}
$d^{H^0}_{P^0}(\lbrace\rho^0, A_{\rho^0}\rbrace, \Lambda, \phi^H) = -2$,\medskip
\end{center}
when $\Theta_{\lbrace\rho^{0\vee},  \hspace{1mm} A_{\rho^{0\vee}}\rbrace}(\gamma) = e^{i(\theta_{1}+\theta_{3})} \mu(\gamma)^{-1}$ or $e^{-i(\theta_{1}+\theta_{3})} \mu(\gamma)^{-1}$, and  $\Lambda=(-1,0,1)$; and\medskip
\begin{center}
$d^{H^0}_{P^0}(\lbrace\rho^0, A_{\rho^0}\rbrace, \Lambda, \phi^H) = 2$,\medskip
\end{center}
when $\Theta_{\lbrace\rho^{0\vee},  \hspace{1mm} A_{\rho^{0\vee}}\rbrace}(\gamma) =  \mu(\gamma)^{-1}$, and  $\Lambda=(-2,0,2)$; $d^{G}_{P}(\rho, \Lambda, \phi)=0$ otherwise.
\end{cor}

\pagebreak
\section{Computations at $p$}

We pick a rational prime $p$ which we assume to be unramified in $E$. Thus, either $p$ is split in $E$ or $p$ is inert in $E$. In this paper, we shall only need to consider the case $p$ splits in $E$; let $\mathfrak{p}$ be one of the primes above $p$. In this case, locally at $p$, the extension is trivial: $E_{\mathfrak{p}} = \mathbb{Q}_{p}$. This means that

\begin{center}\medskip
$G_{E_{\mathfrak{p}}} = (R_{E_{\mathfrak{p}}/\mathbb{Q}_{p}}\mathbf{GU}(2,1)_{E_{\mathfrak{p}}})_{E_{\mathfrak{p}}} = \mathbf{GU}(2,1)_{E_{\mathfrak{p}}}$,\medskip
\end{center}
and so for $L = \mathbb{Q}_{p^j} \supset \mathbb{Q}_{p}=E_{\mathfrak{p}}$,\medskip
\begin{center}
$\mathbf{GU}(2,1)_{L} = \mathbb{G}_{m,L} \times \mathbf{GL}_{3,L}.$\medskip
\end{center}
Also, in this case, the local Galois group is trivial, and so we can ignore the automorphism $\theta$. Overall, this means that the computation is exactly the same as it is in the case $G=\mathbf{GU}(2,1)$.

Recall that $\mu \in X_{\ast}(T)$ was defined to be the cocharacter of $T$ given by\medskip

\begin{center}
$u \mapsto (u, diag(u,u,1)) \in \mathbb{G}_{m,L} \times (\mathbb{G}_{m,L})^{3}.$\medskip
\end{center}
We will now view $\mu$ as a character of the dual torus $\hat{T}$ of $T$; viewed this way, it is the character given by

\begin{center}
$(t, (t_{1}, t_{2}, t_{3})) \mapsto t.t_{1}.t_{2}$\medskip
\end{center}
 Either way, we will regard $\mu$ as the element $(1,(1,1,0))$ of the space $\mathfrak{a}_{T} \subset \mathbb{R}^{4}$. It is then trivial to check that the orbit $W\mu$ is\medskip
\begin{center}
$\left\lbrace (1,(1,1,0)),(1,(1,0,1)),(1,(0,1,1))\right\rbrace $.\medskip
\end{center}

%%%%%%%%%%%%%%%%%%%%%%

\noindent
\textbf{The Group $G$}. First, let $(t,(t_1,t_2,t_3 )) \in \hat{T}$ be the Langlands parameter of the representation $\pi_{p}$ of $T(\mathbb{Q}_{p})$. 

(A1) For $P = P_{u}$, $\Lambda = -\alpha/2=(0,-1/2,1/2)$,
\begin{center}
$p^{j((0,-1/2,1/2)(0,1/2,-1/2)+1)}  ((1,(1,1,0))(t,(t_{1},t_{2},t_{3} )))^{j}$

$+p^{j((0,-1/2,1/2)(0,-1/2,1/2)+1)}  ((1,(1,0,1))(t,(t_{1},t_{2},t_{3} )))^{j}$

$+p^{j((0,-1/2,1/2)(0,0,0)+1)}  ((1,(0,1,1))(t,(t_{1},t_{2},t_{3} )))^{j}$

$=p^{j/2} (t.t_{1}.t_{2} )^{j}+p^{3j/2} (t.t_{1}.t_{3} )^{j}+ p^{j} (t.t_{2}.t_{3} )^{j}$

$=p^{j}(p^{-j/2} (t.t_{1}.t_{2} )^{j}+p^{j/2} (t.t_{1}.t_{3} )^{j}+ (t.t_{2}.t_{3} )^{j})$
\end{center}

(B1)  For $P = P_{d}$, $\Lambda = \alpha/2=(0,1/2,-1/2)$,
\begin{center}
$p^{j((0,1/2,-1/2)(0,1/2,-1/2)+1)}  ((1,(1,1,0))(t,(t_{1},t_{2},t_{3} )))^{j}$

$+p^{j((0,1/2,-1/2)(0,-1/2,1/2)+1)}  ((1,(1,0,1))(t,(t_{1},t_{2},t_{3} )))^{j}$

$+p^{j((0,1/2,-1/2)(0,0,0)+1)}  ((1,(0,1,1))(t,(t_{1},t_{2},t_{3} )))^{j}$

$=p^{3j/2} (t.t_{1}.t_{2} )^{j}+p^{j/2} (t.t_{1}.t_{3} )^{j}+ p^{j} (t.t_{2}.t_{3} )^{j}$

$=p^{j}(p^{j/2} (t.t_{1}.t_{2} )^{j}+p^{-j/2} (t.t_{1}.t_{3} )^{j}+ (t.t_{2}.t_{3} )^{j})$
\end{center}

(C1) For $P = P_{u}, \Lambda= -\alpha/2=(0,-1/2,1/2)$,

\begin{center}
$=p^{j}(p^{-j/2} (t.t_{1}.t_{2} )^{j}+p^{j/2} (t.t_{1}.t_{3} )^{j}+ (t.t_{2}.t_{3} )^{j})$
\end{center}

the same as (A1);

(D1) For $P = P_{d}, \Lambda=\alpha/2=(0,1/2,-1/2)$,

\begin{center}
$=p^{j}(p^{j/2} (t.t_{1}.t_{2} )^{j}+p^{-j/2} (t.t_{1}.t_{3} )^{j}+ (t.t_{2}.t_{3} )^{j})$
\end{center}

the same as (B1).

\pagebreak
(E1) For $P = P_{u}, \Lambda=-\alpha=(0,-1,1)$,

\begin{center}
$p^{j((0,-1,1)(0,1/2,-1/2)+1)}  ((1,(1,1,0))(t,(t_{1},t_{2},t_{3} )))^{j}$

$+p^{j((0,-1,1)(0,-1/2,1/2)+1)}  ((1,(1,0,1))(t,(t_{1},t_{2},t_{3} )))^{j}$

$+p^{j((0,-1,1)(0,0,0)+1)}  ((1,(0,1,1))(t,(t_{1},t_{2},t_{3} )))^{j}$

$=(t.t_{1}.t_{2} )^{j}+p^{2j} (t.t_{1}.t_{3} )^{j}+ p^{j} (t.t_{2}.t_{3} )^{j}$

$=p^{j}(p^{-j}(t.t_{1}.t_{2} )^{j}+p^{j}(t.t_{1}.t_{3} )^{j}+ (t.t_{2}.t_{3} )^{j})$
\end{center}

(F1) For $P = P_{d}, \Lambda=\alpha=(0,1,-1)$,

\begin{center}
$p^{j((0,1,-1)(0,1/2,-1/2)+1)}  ((1,(1,1,0))(t,(t_{1},t_{2},t_{3} )))^{j}$

$+p^{j((0,1,-1)(0,-1/2,1/2)+1)}  ((1,(1,0,1))(t,(t_{1},t_{2},t_{3} )))^{j}$

$+p^{j((0,1,-1)(0,0,0)+1)}  ((1,(0,1,1))(t,(t_{1},t_{2},t_{3} )))^{j}$

$=p^{2j}(t.t_{1}.t_{2} )^{j}+ (t.t_{1}.t_{3} )^{j}+ p^{j} (t.t_{2}.t_{3} )^{j}$

$=p^{j}(p^{j}(t.t_{1}.t_{2} )^{j}+p^{-j}(t.t_{1}.t_{3} )^{j}+ (t.t_{2}.t_{3} )^{j})$
\end{center}

Finally, we need to consider

\begin{center}
$\sum_{X_{q} \in a_{T,q}} e^{-\Lambda(X_{q}^{G})}\phi_{q,M}(\pi_{q}, X_{q})$,
\end{center}
where $X_{q}^{G} \in j.\nu_{T}^{G}.logp+(a_{P}^{G})^{+}$.\medskip

(A2.1) For $P = P_{u}$, $\Lambda=(0,-1/2,1/2)$, and $\nu=(1,(1,1,0))$, $\nu_{T}^{G}=(0,1/2,-1/2)$,
\begin{center}
$\sum_{X_{q}=(y,x,-x) \in a_{T,q}} e^{x}\phi_{q,T}(\pi_{q}, X_{q})$.
\end{center}
where $x>(1/2)jlogp$;

(A2.2) For $P = P_{u}$, $\Lambda=(0,-1/2,1/2)$, and $\nu=(1,(1,0,1))$, $\nu_{T}^{G}=(0,-1/2,1/2)$,
\begin{center}
$\sum_{X_{q}=(y,x,-x) \in a_{T,q}} e^{x}\phi_{q,T}(\pi_{q}, X_{q})$.
\end{center}
where  $x>(-1/2)jlogp$;

(A2.3) For $P = P_{u}$, $\Lambda=(0,-1/2,1/2)$, and $\nu=(1,(0,1,1))$, $\nu_{T}^{G}=(0,0,0)$,
\begin{center}
$\sum_{X_{q}=(y,x,-x) \in a_{T,q}} e^{x}\phi_{q,T}(\pi_{q}, X_{q})$.
\end{center}
where  $x>0$;

(B2.1) For $P = P_{d}$, $\Lambda=(0,1/2,-1/2)$, and $\nu=(1,(1,1,0))$, $\nu_{T}^{G}=(0,1/2,-1/2)$,
\begin{center}
$\sum_{X_{q}=(y,x,-x) \in a_{T,q}} e^{-x}\phi_{q,T}(\pi_{q}, X_{q})$.
\end{center}
where  $x<(1/2)jlogp$;

(B2.2) For $P = P_{d}$, $\Lambda=(0,1/2,-1/2)$, and $\nu=(1,(1,0,1))$, $\nu_{T}^{G}=(0,-1/2,1/2)$,
\begin{center}
$\sum_{X_{q}=(y,x,-x) \in a_{T,q}} e^{-x}\phi_{q,T}(\pi_{q}, X_{q})$.
\end{center}
where  $x<(-1/2)jlogp$;

(B2.3) For $P = P_{d}$, $\Lambda=(0,1/2,-1/2)$, and $\nu=(1,(0,1,1))$, $\nu_{T}^{G}=(0,0,0)$,
\begin{center}
$\sum_{X_{q}=(y,x,-x) \in a_{T,q}} e^{-x}\phi_{q,T}(\pi_{q}, X_{q})$.
\end{center}
where  $x<0$.

(C2.1) For $P = P_{u}$, $\Lambda=(0,-1/2,1/2)$, and $\nu=(1,(1,1,0))$, $\nu_{T}^{G}=(0,1/2,-1/2)$,
\begin{center}
$\sum_{X_{q}=(y,x,-x) \in a_{T,q}} e^{x}\phi_{q,T}(\pi_{q}, X_{q})$.
\end{center}
where $x>(1/2)jlogp$;

(C2.2) For $P = P_{u}$, $\Lambda=(0,-1/2,1/2)$, and $\nu=(1,(1,0,1))$, $\nu_{T}^{G}=(0,-1/2,1/2)$,
\begin{center}
$\sum_{X_{q}=(y,x,-x) \in a_{T,q}} e^{x}\phi_{q,T}(\pi_{q}, X_{q})$.
\end{center}
where  $x>(-1/2)jlogp$;

(C2.3) For $P = P_{u}$, $\Lambda=(0,-1/2,1/2)$, and $\nu=(1,(0,1,1))$, $\nu_{T}^{G}=(0,0,0)$,
\begin{center}
$\sum_{X_{q}=(y,x,-x) \in a_{T,q}} e^{x}\phi_{q,T}(\pi_{q}, X_{q})$.
\end{center}
where  $x>0$;

(D2.1) For $P = P_{d}$, $\Lambda=(0,1/2,-1/2)$, and $\nu=(1,(1,1,0))$, $\nu_{T}^{G}=(0,1/2,-1/2)$,
\begin{center}
$\sum_{X_{q}=(y,x,-x) \in a_{T,q}} e^{-x}\phi_{q,T}(\pi_{q}, X_{q})$.
\end{center}
where  $x<(1/2)jlogp$;

(D2.2) For $P = P_{d}$, $\Lambda=(0,1/2,-1/2)$, and $\nu=(1,(1,0,1))$, $\nu_{T}^{G}=(0,-1/2,1/2)$,
\begin{center}
$\sum_{X_{q}=(y,x,-x) \in a_{T,q}} e^{-x}\phi_{q,T}(\pi_{q}, X_{q})$.
\end{center}
where  $x<(-1/2)jlogp$;

(D2.3) For $P = P_{d}$, $\Lambda=(0,1/2,-1/2)$, and $\nu=(1,(0,1,1))$, $\nu_{T}^{G}=(0,0,0)$,
\begin{center}
$\sum_{X_{q}=(y,x,-x) \in a_{T,q}} e^{-x}\phi_{q,T}(\pi_{q}, X_{q})$.
\end{center}
where  $x<0$.

(E2.1) For $P = P_{u}$, $\Lambda=(0,-1,1)$, and $\nu=(1,(1,1,0))$, $\nu_{T}^{G}=(0,1/2,-1/2)$,
\begin{center}
$\sum_{X_{q}=(y,x,-x) \in a_{T,q}} e^{2x}\phi_{q,T}(\pi_{q}, X_{q})$.
\end{center}
where $x>(1/2)jlogp$;

(E2.2) For $P = P_{u}$, $\Lambda=(0,-1,1)$, and $\nu=(1,(1,0,1))$, $\nu_{T}^{G}=(0,-1/2,1/2)$,
\begin{center}
$\sum_{X_{q}=(y,x,-x) \in a_{T,q}} e^{2x}\phi_{q,T}(\pi_{q}, X_{q})$.
\end{center}
where  $x>(-1/2)jlogp$;

(E2.3) For $P = P_{u}$, $\Lambda=(0,-1,1)$, and $\nu=(1,(0,1,1))$, $\nu_{T}^{G}=(0,0,0)$,
\begin{center}
$\sum_{X_{q}=(y,x,-x) \in a_{T,q}} e^{2x}\phi_{q,T}(\pi_{q}, X_{q})$.
\end{center}
where  $x>0$;

(F2.1) For $P = P_{d}$, $\Lambda=(0,1,-1)$, and $\nu=(1,(1,1,0))$, $\nu_{T}^{G}=(0,1/2,-1/2)$,
\begin{center}
$\sum_{X_{q}=(y,x,-x) \in a_{T,q}} e^{-2x}\phi_{q,T}(\pi_{q}, X_{q})$.
\end{center}
where  $x<(1/2)jlogp$;

(F2.2) For $P = P_{d}$, $\Lambda=(0,1,-1)$, and $\nu=(1,(1,0,1))$, $\nu_{T}^{G}=(0,-1/2,1/2)$,
\begin{center}
$\sum_{X_{q}=(y,x,-x) \in a_{T,q}} e^{-2x}\phi_{q,T}(\pi_{q}, X_{q})$.
\end{center}
where  $x<(-1/2)jlogp$;

(F2.3) For $P = P_{d}$, $\Lambda=(0,1,-1)$, and $\nu=(1,(0,1,1))$, $\nu_{T}^{G}=(0,0,0)$,
\begin{center}
$\sum_{X_{q}=(y,x,-x) \in a_{T,q}} e^{-2x}\phi_{q,T}(\pi_{q}, X_{q})$.
\end{center}
where  $x<0$.

%%%%%%%%%%%%%%%%%%%%%%%%%%%%%%%%%%%%%

\noindent
\textbf{The Group $H$}. Let $(t,t_{1},(t_{2},t_{3}))\in \hat{T}$ be the Langlands parameter of the representation $\pi_{p}$ of $T(\mathbb{Q}_{p})$. By the transfer explained in \cite[p.~72]{M}, we have

\begin{center}
$(1,(1,1,0))=XX_{1} X_{2} \mapsto(Z)(Z_{1,1} )(-Z_{2,1})=(1,1,(-1,0))$
$(1,(1,0,1))=XX_{1} X_{3} \mapsto(Z)(Z_{1,1} )(-Z_{2,2})=(1,1,(0,-1))$
$(1,(0,1,1))=XX_{2} X_{3} \mapsto(Z)(-Z_{2,1} )(-Z_{2,2})=(1,0,(-1,-1))$
\end{center}
where
\begin{center}
$X((t,(t_{1},t_{2},t_{3})))=t$

$X_{j}((t,(t_{1},t_{2},t_{3})))=t_{j}$
\end{center}
for $j=1,2,3$.

\begin{center}
$Z(t,t_1,(t_2,t_3 ))=t$

$Z_{1,1}(t,t_1,(t_2,t_3 ))=t_1$

$Z_{2,1}(t,t_1,(t_2,t_3 ))=t_2$

$Z_{2,2}(t,t_1,(t_2,t_3 ))=t_3$
\end{center}

Thus,
\begin{center}
$(Z)(Z_{1,1})(-Z_{2,1})(t_0,t_1,(t_2,t_3 ))=-t.t_1.t_2$

$(Z)(Z_{1,1})(-Z_{2,2})(t_0,t_1,(t_2,t_3 ))=-t.t_1.t_3$

$(Z)(-Z_{2,1})(-Z_{2,2})(t_0,t_1,(t_2,t_3 ))=t.t_2.t_3$
\end{center}

(H-A1) For $P = P_{u}$, $\Lambda=-\alpha/2=(0,-1/2,1/2)$,
\begin{center}
$p^{j((0,-1/2,1/2)(0,1/2,-1/2)+1)}  ((1,1,(-1,0))(t,t_1,(t_2,t_3 )))^{j}$

$+p^{j((0,-1/2,1/2)(0,-1/2,1/2)+1)}  ((1,1,(0,-1))(t,t_1,(t_2,t_3 )))^{j}$

$+p^{j((0,-1/2,1/2)(0,0,0)+1)}  ((1,0,(-1,-1))(t,t_1,(t_2,t_3 )))^{j}$

$= -p^{j/2} (t.t_1.t_2 )^{j} - p^{3j/2} (t.t_1.t_3 )^{j} + p^{j} (t.t_2.t_3 )^{j}$

$= -p^{j} (p^{-j/2} (t.t_{1}.t_{2})^{j} + p^{j/2} (t.t_{1}.t_{3})^{j} - (t.t_{2}.t_{3})^{j} )$
\end{center}

(H-B1) For $P = P_{d}$, $\Lambda=\alpha/2=(0,1/2,-1/2)$,
\begin{center}
$p^{j((0,1/2,-1/2)(0,1/2,-1/2)+1)}  ((1,1,(-1,0))(t,t_1,(t_2,t_3 )))^{j}$

$+p^{j((0,1/2,-1/2)(0,-1/2,1/2)+1)}  ((1,1,(0,-1))(t,t_1,(t_2,t_3 )))^{j}$

$+p^{j((0,1/2,-1/2)(0,0,0)+1)}  ((1,0,(-1,-1))(t,t_1,(t_2,t_3 )))^{j}$

$=-p^{3j/2}(t.t_{1}.t_{2})^{j} - p^{j/2} (t.t_{1}.t_{3})^{j} + p^{j} (t.t_{2}.t_{3})^{j}$

$=-p^{j} (p^{j/2} (t.t_{1}.t_{2})^{j} + p^{-j/2} (t.t_{1}.t_{3})^{j} - (t.t_{2}.t_{3})^{j} )$
\end{center}

(H-C1) For $P = P_{u}$, $\Lambda=-\alpha/2=(0,-1/2,1/2)$,
\begin{center}
$-p^{j} (p^{-j/2}(t.t_{1}.t_{2})^{j} + p^{j/2}(t.t_{1}.t_{3})^{j} - (t.t_{2}.t_{3})^{j})$
\end{center}
the same as (A1);

(H-D1) For $P = P_{d}$, $\Lambda=\alpha/2=(0,1/2,-1/2)$,
\begin{center}
$-p^{j}(p^{j/2}(t.t_{1}.t_{2})^{j} + p^{-j/2}(t.t_{1}.t_{3})^{j} - (t.t_{2}.t_{3})^{j})$
\end{center}
the same as (B1).

(H-E1) For $P = P_u$, $\Lambda=-\alpha=(0,-1,1)$,

\begin{center}
$p^{j((0,-1,1)(0,1/2,-1/2)+1)}((1,1,(-1,0))(t,t_{1},(t_{2},t_{3})))^{j}$
$+p^{j((0,-1,1)(0,-1/2,1/2)+1)}((1,1,(0,-1))(t,t_{1},(t_{2},t_{3})))^{j}$
$+p^{j((0,-1,1)(0,0,0)+1)} ((1,0,(-1,-1))(t,t_{1},(t_{2},t_{3})))^{j}$
$=-(t.t_{1}.t_{2})^j - p^{2j}(t.t_{1}.t_{3})^j + p^{j}(t.t_{2}.t_{3})^{j}$
$=-p^{j}(p^{-j}(t.t_{1}.t_{2})^{j} + p^{j}(t.t_{1}.t_{3})^{j} - (t.t_{2}.t_{3})^{j})$
\end{center}

(H-F1) For $P = P_{d}$, $\Lambda=\alpha=(0,1,-1)$,

\begin{center}
$p^{j((0,1,-1)(0,1/2,-1/2)+1)}((1,1,(-1,0))(t,t_1,(t_2,t_3 )))^{j}$
$+p^{j((0,1,-1)(0,-1/2,1/2)+1)}((1,1,(0,-1))(t,t_1,(t_2,t_3 )))^{j}$
$+p^{j((0,1,-1)(0,0,0)+1)} ((1,0,(-1,-1))(t,t_1,(t_2,t_3 )))^{j}$
$=-p^{2j}(2t.t_{1}.t_{2})^j - (t.t_{1}.t_{3})^j + p^{j}(t.t_{2}.t_{3})^{j}$
$=-p^{j}(p^{j}(t.t_{1}.t_{2})^{j} + p^{-j}(t.t_{1}.t_{3})^{j} - (t.t_{2}.t_{3})^{j})$
\end{center}

(H-A2.1) For $P = P_{u}, \Lambda=(0,-1/2,1/2)$, $\nu=(1,1,(-1,0))$, 
$\nu_{T}^{G}=(0,(-1/2,1/2))$,
\begin{center}
$\sum_{X_{q}=(y,x,-x)\in a_{T,q}}e^{x} \phi_{q,T}(\pi_{q},X_{q})$
\end{center}
where $x>(-1/2)jlogp$.

(H-A2.2) For $P = P_{u}, \Lambda=(0,-1/2,1/2)$, $\nu=(1,1,(0,-1))$, $\nu_{T}^{G}=(0,(1/2,-1/2))$,
\begin{center}
$\sum_{X_{q}=(y,x,-x)\in a_{T,q}}e^{x} \phi_{q,T}(\pi_{q},X_{q})$
\end{center}
where $x>(1/2)jlogp$.

(H-A2.3) For $P = P_{u}, \Lambda=(0,-1/2,1/2)$, $\nu=(1,1,(-1,-1))$, $\nu_{T}^{G}=(0,(0,0))$,
\begin{center}
$\sum_{X_{q}=(y,x,-x)\in a_{T,q}}e^{x} \phi_{q,T}(\pi_{q},X_{q})$
\end{center}
where $x>0$.

(H-B2.1) For $P = P_{d}, \Lambda=(0,1/2,-1/2)$, $\nu=(1,1,(-1,0))$, $\nu_{T}^{G}=(0,(-1/2,1/2))$,
\begin{center}
$\sum_{X_{q}=(y,x,-x)\in a_{T,q}}e^{-x} \phi_{q,T}(\pi_{q},X_{q})$
\end{center}
where $x<(-1/2)jlogp$.

(H-B2.2) For $P = P_{d}, \Lambda=(0,1/2,-1/2)$, $\nu=(1,1,(0,-1))$, $\nu_{T}^{G}=(0,(1/2,-1/2))$,
\begin{center}
$\sum_{X_{q}=(y,x,-x)\in a_{T,q}}e^{-x} \phi_{q,T}(\pi_{q},X_{q})$
\end{center}
where $x<(-1/2)jlogp$.

(H-B2.3) For $P = P_{d}, \Lambda=(0,1/2,-1/2)$, $\nu=(1,1,(-1,-1))$, $\nu_{T}^{G}=(0,(0,0))$,
\begin{center}
$\sum_{X_{q}=(y,x,-x)\in a_{T,q}}e^{-x} \phi_{q,T}(\pi_{q},X_{q})$
\end{center}
where $x<0$.

(H-C2.1) For $P = P_{u}, \Lambda=(0,-1/2,1/2)$, $\nu=(1,1,(-1,0))$, $\nu_{T}^{G}=(0,(-1/2,1/2))$,
\begin{center}
$\sum_{X_{q}=(y,x,-x)\in a_{T,q}}e^{x} \phi_{q,T}(\pi_{q},X_{q})$
\end{center}
where $x>(-1/2)jlogp$.

(H-C2.2) For $P = P_{u}, \Lambda=(0,-1/2,1/2)$, $\nu=(1,1,(0,-1))$, $\nu_{T}^{G}=(0,(1/2,-1/2))$,
\begin{center}
$\sum_{X_{q}=(y,x,-x)\in a_{T,q}}e^{x} \phi_{q,T}(\pi_{q},X_{q})$
\end{center}
where $x>(1/2)jlogp$.

(H-C2.3) For $P = P_{u}, \Lambda=(0,-1/2,1/2)$, $\nu=(1,1,(-1,-1))$, $\nu_{T}^{G}=(0,(0,0))$,
\begin{center}
$\sum_{X_{q}=(y,x,-x)\in a_{T,q}}e^{x} \phi_{q,T}(\pi_{q},X_{q})$
\end{center}
where $x>0$.

(H-D2.1) For $P = P_{d}, \Lambda=(0,1/2,-1/2)$, $\nu=(1,1,(-1,0))$, $\nu_{T}^{G}=(0,(-1/2,1/2))$,
\begin{center}
$\sum_{X_{q}=(y,x,-x)\in a_{T,q}}e^{-x} \phi_{q,T}(\pi_{q},X_{q})$
\end{center}
where $x<(-1/2)jlogp$.

(H-D2.2) For $P = P_{d}, \Lambda=(0,1/2,-1/2)$, and $\nu=(1,1,(0,-1))$, $\nu_{T}^{G}=(0,(1/2,-1/2))$,
\begin{center}
$\sum_{X_{q}=(y,x,-x)\in a_{T,q}}e^{-x} \phi_{q,T}(\pi_{q},X_{q})$
\end{center}
where $x<(1/2)jlogp$.

(H-D2.3) For $P = P_{d}, \Lambda=(0,1/2,-1/2)$, $\nu=(1,1,(-1,-1))$, $\nu_{T}^{G}=(0,(0,0))$,
\begin{center}
$\sum_{X_{q}=(y,x,-x)\in a_{T,q}}e^{-x} \phi_{q,T}(\pi_{q},X_{q})$
\end{center}
where $x<0$.

(H-E2.1) For $P = P_{u}, \Lambda=(0,-1,1)$, $\nu=(1,1,(-1,0))$, $\nu_{T}^{G}=(0,(-1/2,1/2))$,
\begin{center}
$\sum_{X_{q}=(y,x,-x)\in a_{T,q}}e^{2x} \phi_{q,T}(\pi_{q},X_{q})$
\end{center}
where $x>(-1/2)jlogp$.

(H-E2.2) For $P = P_{u}, \Lambda=(0,-1,1)$, $\nu=(1,1,(0,-1))$, $\nu_{T}^{G}=(0,(1/2,-1/2))$,
\begin{center}
$\sum_{X_{q}=(y,x,-x)\in a_{T,q}}e^{2x} \phi_{q,T}(\pi_{q},X_{q})$
\end{center}
where $x>(1/2)jlogp$.

(H-E2.3) For $P = P_{u}, \Lambda=(0,-1,1)$, $\nu=(1,1,(-1,-1))$, $\nu_{T}^{G}=(0,(0,0))$,
\begin{center}
$\sum_{X_{q}=(y,x,-x)\in a_{T,q}}e^{2x} \phi_{q,T}(\pi_{q},X_{q})$
\end{center}
where $x>0$.

(H-F2.1) For $P = P_{d}, \Lambda=(0,1,-1)$, $\nu=(1,1,(-1,0))$, $\nu_{T}^{G}=(0,(-1/2,1/2))$,
\begin{center}
$\sum_{X_{q}=(y,x,-x)\in a_{T,q}}e^{-2x} \phi_{q,T}(\pi_{q},X_{q})$
\end{center}
where $x<(-1/2)jlogp$.

(H-F2.2) For $P = P_{d}, \Lambda=(0,1,-1)$, $\nu=(1,1,(0,-1))$, $\nu_{T}^{G}=(0,(1/2,-1/2))$,
\begin{center}
$\sum_{X_{q}=(y,x,-x)\in a_{T,q}}e^{-2x} \phi_{q,T}(\pi_{q},X_{q})$
\end{center}
where $x<(1/2)jlogp$.

(H-F2.3) For $P = P_{d}, \Lambda=(0,1,-1)$, $\nu=(1,1,(-1,-1))$, $\nu_{T}^{G}=(0,(0,0))$,
\begin{center}
$\sum_{X_{q}=(y,x,-x)\in a_{T,q}}e^{-2x} \phi_{q,T}(\pi_{q},X_{q})$
\end{center}
where $x<0$.

\section{Stabilization of Parabolic Terms}

\textbf{First Major Sum for} $G$

\begin{center}
$\sum_{\pi : \pi^{\vee}_{\mathbb{R}}=e^{-i\theta}} 2\mathrm{tr} i_{P_u}(\pi_{f, \Lambda =(-1/2,1/2)}^{p,q}, \phi^{p,q})$
$p^{j}(p^{-j/2} (t.t_1.t_2 )^j \sum_{X_{q}=(y,x,-x)\in a_{T,q}: x>(1/2)jlogp} e^x \phi_{q,T} (\pi_q, X_q)$
$+p^{j/2} (t.t_1.t_3 )^j \sum_{X_{q}=(y,x,-x)\in a_{T,q}: x>(-1/2)jlogp} e^x \phi_{q,T} (\pi_q, X_q)$ 
$+(t.t_2.t_{3})^j \sum_{X_{q}=(y,x,-x)\in a_{T,q}: x>0} e^x \phi_{q,T} (\pi_q, X_q))$
\end{center}
\begin{center}
$+2\mathrm{tr} i_{P_d}(\pi_{f, \Lambda =(1/2,-1/2)}^{p,q}, \phi^{p,q})$
$p^{j}(p^{j/2} (t.t_1.t_2 )^j \sum_{X_{q}=(y,x,-x)\in a_{T,q}: x<(1/2)jlogp} e^{-x} \phi_{q,T} (\pi_q, X_q)$
$+p^{-j/2} (t.t_1.t_3)^j \sum_{X_{q}=(y,x,-x)\in a_{T,q}: x<(-1/2)jlogp} e^{-x} \phi_{q,T} (\pi_q, X_q)$ 
$+(t.t_2.t_3)^j \sum_{X_{q}=(y,x,-x)\in a_{T,q}: x<0} e^{-x} \phi_{q,T} (\pi_q, X_q))$
\end{center}

\begin{center}
$=\sum_{\pi : \pi^{\vee}_{\mathbb{R}}=e^{-i\theta}}2\mathrm{tr} i_{P_u}(\pi_{f, \Lambda =(-1/2,1/2)}^{p,q}, \phi^{p,q})$
$p^{j}(p^{j/2} (t.t_1.t_3)^j \sum_{X_{q}=(y,x,-x)\in a_{T,q}} e^x \phi_{q,T} (\pi_q, X_q)$ 
$+(t.t_2.t_3)^j \sum_{X_{q}=(y,x,-x)\in a_{T,q}: x>0} e^{\vert x \vert} \phi_{q,T} (\pi_q, X_q))$
\end{center}
\begin{center}
$+2\mathrm{tr} i_{P_d}(\pi_{f, \Lambda =(1/2,-1/2)}^{p,q}, \phi^{p,q})$

$p^{j}(p^{j/2} (t.t_1.t_{2})^j \sum_{X_{q}=(y,x,-x)\in a_{T,q}} e^{-x} \phi_{q,T} (\pi_q, X_q)$
$+(t.t_2.t_{3})^j \sum_{X_{q}=(y,x,-x)\in a_{T,q}: x<0} e^{\vert x \vert} \phi_{q,T} (\pi_q, X_q))$
\end{center}

\begin{center}
$=\sum_{\pi : \pi^{\vee}_{\mathbb{R}}=e^{-i\theta}} 2p^{3j/2} ((t.t_1.t_3)^j + (t.t_1.t_2)^j)\mathrm{tr} I_{T}(\pi_{f, \Lambda =(-1/2,1/2)}^{p}, \phi^p)$
 
$+\sum_{\pi : \pi^{\vee}_{\mathbb{R}}=e^{-i\theta}} 2p^{j} (t.t_2.t_{3})^j \mathrm{tr} I_{T}(\pi_{f, \Lambda =(-1/2,1/2)}^{p,q}, \phi^{p,q})$

$\sum_{X_{q}=(y,x,-x)\in a_{T,q}} e^{\vert x \vert} \phi_{q,T} (\pi_q, X_q))$
\end{center}

\textbf{Second Major Sum for}  $G$

\begin{center}
$\sum_{\pi : \pi^{\vee}_{\mathbb{R}}=e^{i\theta}}2\mathrm{tr} i_{P_u}(\pi_{f, \Lambda =(-1/2,1/2)}^{p,q}, \phi^{p,q})$
$p^{j}(p^{-j/2} (t.t_1.t_{2})^j \sum_{X_{q}=(y,x,-x)\in a_{T,q}: x>(1/2)jlogp} e^x \phi_{q,T} (\pi_q, X_q)$
$+p^{j/2} (t.t_1.t_{3})^j \sum_{X_{q}=(y,x,-x)\in a_{T,q}: x>(-1/2)jlogp} e^x \phi_{q,T} (\pi_q, X_q)$ 
$+(t.t_{2}.t_{3})^j \sum_{X_{q}=(y,x,-x)\in a_{T,q}: x>0} e^x \phi_{q,T} (\pi_q, X_q))$
\end{center}
\begin{center}
$+2\mathrm{tr} i_{P_d}(\pi_{f, \Lambda =(1/2,-1/2)}^{p,q}, \phi^{p,q})$
$p^{j}(p^{j/2} (t.t_{1}.t_{2})^j \sum_{X_{q}=(y,x,-x)\in a_{T,q}: x<(1/2)jlogp} e^{-x} \phi_{q,T} (\pi_q, X_q)$
$+p^{-j/2} (t.t_{1}.t_{3})^j \sum_{X_{q}=(y,x,-x)\in a_{T,q}: x<(-1/2)jlogp} e^{-x} \phi_{q,T} (\pi_q, X_q)$ 
$+(t.t_{2}.t_{3})^j \sum_{X_{q}=(y,x,-x)\in a_{T,q}: x<0} e^{-x} \phi_{q,T} (\pi_q, X_q))$
\end{center}

\begin{center}
$=\sum_{\pi : \pi^{\vee}_{\mathbb{R}}=e^{i\theta}}2\mathrm{tr} i_{P_u}(\pi_{f, \Lambda =(-1/2,1/2)}^{p,q}, \phi^{p,q})$
$p^{j}(p^{j/2} (t.t_1.t_{3})^j \sum_{X_{q}=(y,x,-x)\in a_{T,q}} e^x \phi_{q,T} (\pi_q, X_q)$ 
$+(t.t_{2}.t_{3})^j \sum_{X_{q}=(y,x,-x)\in a_{T,q}: x>0} e^{\vert x \vert} \phi_{q,T} (\pi_q, X_q))$
\end{center}
\begin{center}
$+2\mathrm{tr} i_{P_d}(\pi_{f, \Lambda =(1/2,-1/2)}^{p,q}, \phi^{p,q})$

$p^{j}(p^{j/2} (t.t_1.t_2 )^j \sum_{X_{q}=(y,x,-x)\in a_{T,q}} e^{-x} \phi_{q,T} (\pi_q, X_q)$

$+(t.t_{2}.t_{3})^j \sum_{X_{q}=(y,x,-x)\in a_{T,q}: x<0} e^{\vert x \vert} \phi_{q,T} (\pi_q, X_q))$
\end{center}

\begin{center}
$=\sum_{\pi : \pi^{\vee}_{\mathbb{R}}=e^{i\theta}} 2p^{3j/2} ((t.t_{1}.t_{2})^j + (t.t_{1}.t_{3})^j)\mathrm{tr} I_{T}(\pi_{f, \Lambda =(-1/2,1/2)}^{p}, \phi^p)$
 
$+\sum_{\pi : \pi^{\vee}_{\mathbb{R}}=e^{i\theta}} 2p^{j} (t.t_2.t_{3})^j \mathrm{tr} I_{T}(\pi_{f, \Lambda =(-1/2,1/2)}^{p,q}, \phi^{p,q})$

$\sum_{X_{q}=(y,x,-x)\in a_{T,q}} e^{\vert x \vert} \phi_{q,T} (\pi_q, X_q))$
\end{center}

\textbf{Third Major Sum for}  $G$

\begin{center}
$\sum_{\pi : \pi^{\vee}_{\mathbb{R}}=1}(-2)\mathrm{tr} i_{P_u}(\pi_{f, \Lambda =(-1,1)}^{p,q}, \phi^{p,q})$
$p^{j}(p^{-j} (t.t_{1}.t_{2})^j \sum_{X_{q}=(y,x,-x)\in a_{T,q}: x>(1/2)jlogp} e^{2x} \phi_{q,T} (\pi_q, X_q)$
$+p^{j} (t.t_{1}.t_{3})^j \sum_{X_{q}=(y,x,-x)\in a_{T,q}: x>(-1/2)jlogp} e^{2x} \phi_{q,T} (\pi_q, X_q)$ 
$+(t.t_{2}.t_{3})^j \sum_{X_{q}=(y,x,-x)\in a_{T,q}: x>0} e^{2x} \phi_{q,T} (\pi_q, X_q))$
\end{center}
\begin{center}
$+(-2)\mathrm{tr} i_{P_d}(\pi_{f, \Lambda =(1,-1)}^{p,q}, \phi^{p,q})$
$p^{j}(p^{j} (t.t_{1}.t_{2})^j \sum_{X_{q}=(y,x,-x)\in a_{T,q}: x<(1/2)jlogp} e^{-2x} \phi_{q,T} (\pi_q, X_q)$
$+p^{-j} (t.t_{1}.t_{3})^j \sum_{X_{q}=(y,x,-x)\in a_{T,q}: x<(-1/2)jlogp} e^{-2x} \phi_{q,T} (\pi_q, X_q)$ 
$+(t.t_{2}.t_{3})^j \sum_{X_{q}=(y,x,-x)\in a_{T,q}: x<0} e^{-2x} \phi_{q,T} (\pi_q, X_q))$
\end{center}

\begin{center}
$=\sum_{\pi : \pi^{\vee}_{\mathbb{R}}=1}(-2)\mathrm{tr} i_{P_u}(\pi_{f, \Lambda =(-1,1)}^{p,q}, \phi^{p,q})$
$p^{j}(p^{j} (t.t_1.t_{3})^j \sum_{X_{q}=(y,x,-x)\in a_{T,q}} e^{2x} \phi_{q,T} (\pi_q, X_q)$ 
$+(t.t_{2}.t_{3})^j \sum_{X_{q}=(y,x,-x)\in a_{T,q}: x>0} e^{\vert 2x \vert} \phi_{q,T} (\pi_q, X_q))$
\end{center}
\begin{center}
$+(-2)\mathrm{tr} i_{P_d}(\pi_{f, \Lambda =(1,-1)}^{p,q}, \phi^{p,q})$

$p^{j}(p^{j} (t.t_1.t_2 )^j \sum_{X_{q}=(y,x,-x)\in a_{T,q}} e^{-2x} \phi_{q,T} (\pi_q, X_q)$

$+(t.t_{2}.t_{3})^j \sum_{X_{q}=(y,x,-x)\in a_{T,q}: x<0} e^{\vert 2x \vert} \phi_{q,T} (\pi_q, X_q))$
\end{center}

\begin{center}
$=\sum_{\pi : \pi^{\vee}_{\mathbb{R}}=1} (-2)p^{2j} ((t.t_1.t_2)^j + (t.t_1.t_3)^j)\mathrm{tr} I_{T}(\pi_{f, \Lambda =(-1/2,1/2)}^{p}, \phi^p)$
 
$+\sum_{\pi : \pi^{\vee}_{\mathbb{R}}=e^{i\theta}}(-2p^{j} (t.t_2.t_{3})^j)\mathrm{tr} I_{T}(\pi_{f, \Lambda =(-1/2,1/2)}^{p,q}, \phi^{p,q})$

$\sum_{X_{q}=(y,x,-x)\in a_{T,q}} e^{\vert 2x \vert} \phi_{q,T} (\pi_q, X_q))$
\end{center}

\textbf{First Major Sum for}  $H$

\begin{center}
$\sum_{\pi : \pi^{\vee}_{\mathbb{R}}=\mu^{-1}e^{-i\theta}}(-2)\mathrm{tr} i_{P_u}(\pi_{h, \Lambda =(-1/2,1/2)}^{p,q}, \varphi^{p,q})$
$p^{j}(-p^{-j/2} (t.t_1.t_2 )^j \sum_{X_{q}=(y,x,-x)\in a_{T,q}: x>(-1/2)jlogp} e^x \varphi_{q,T} (\pi_q, X_q)$
$-p^{j/2} (t.t_1.t_3 )^j \sum_{X_{q}=(y,x,-x)\in a_{T,q}: x>(1/2)jlogp} e^x \varphi_{q,T} (\pi_q, X_q)$ 
$+(t.t_2.t_3 )^j \sum_{X_{q}=(y,x,-x)\in a_{T,q}: x>0} e^x \varphi_{q,T} (\pi_q, X_q))$
\end{center}
\begin{center}
$+(-2)\mathrm{tr} i_{P_d}(\pi_{f, \Lambda =(1/2,-1/2)}^{p,q}, \varphi^{p,q})$
$p^{j}(-p^{j/2} (t.t_1.t_2 )^j \sum_{X_{q}=(y,x,-x)\in a_{T,q}: x<(-1/2)jlogp} e^{-x} \varphi_{q,T} (\pi_q, X_q)$
$-p^{-j/2} (t.t_1.t_3 )^j \sum_{X_{q}=(y,x,-x)\in a_{T,q}: x<(1/2)jlogp} e^{-x} \varphi_{q,T} (\pi_q, X_q)$ 
$+(t.t_2.t_3 )^j \sum_{X_{q}=(y,x,-x)\in a_{T,q}: x<0} e^{-x} \varphi_{q,T} (\pi_q, X_q))$
\end{center}

\begin{center}
$=\sum_{\pi : \pi^{\vee}_{\mathbb{R}}=\mu^{-1}e^{-i\theta}}2\mathrm{tr} i_{P_u}(\pi_{h, \Lambda =(-1/2,1/2)}^{p,q}, \varphi^{p,q})$
$p^{j}(p^{-j/2} (t.t_1.t_2)^j \sum_{X_{q}=(y,x,-x)\in a_{T,q}} e^x \varphi_{q,T} (\pi_q, X_q)$ 
$-(t.t_2.t_3 )^j \sum_{X_{q}=(y,x,-x)\in a_{T,q}: x>0} e^{\vert x \vert} \varphi_{q,T} (\pi_q, X_q))$
\end{center}
\begin{center}
$+2\mathrm{tr} i_{P_d}(\pi_{f, \Lambda =(1/2,-1/2)}^{p,q}, \varphi^{p,q})$

$p^{j}(p^{-j/2} (t.t_1.t_{3})^j \sum_{X_{q}=(y,x,-x)\in a_{T,q}} e^{-x} \varphi_{q,T} (\pi_q, X_q)$
$-(t.t_2.t_{3})^j \sum_{X_{q}=(y,x,-x)\in a_{T,q}: x<0} e^{\vert x \vert} \varphi_{q,T} (\pi_q, X_q))$
\end{center}

\begin{center}
$=\sum_{\pi : \pi^{\vee}_{\mathbb{R}}=\mu^{-1}e^{-i\theta}} 2p^{j/2} ((t.t_1.t_{2})^j + (t.t_1.t_{3})^j) \mathrm{tr} I_{T}(\pi_{h, \Lambda =(-1/2,1/2)}^{p}, \varphi^{p})$

$-\sum_{\pi : \pi^{\vee}_{\mathbb{R}}=\mu^{-1}e^{-i\theta}} 2p^{j}(t.t_2.t_3 )^j \mathrm{tr} I_{T}(\pi_{h, \Lambda =(-1/2,1/2)}^{p,q}, \varphi^{p,q})$

$\sum_{X_{q}=(y,x,-x)\in a_{T,q}} e^{\vert x \vert} \varphi_{q,T} (\pi_q, X_q))$
\end{center}

\pagebreak
\textbf{Second Major Sum for}  $H$

\begin{center}
$\sum_{\pi : \pi^{\vee}_{\mathbb{R}}=\mu^{-1}e^{i\theta}}(-2)\mathrm{tr} i_{P_u}(\pi_{h, \Lambda =(-1/2,1/2)}^{p,q}, \varphi^{p,q})$
$p^{j}(-p^{-j/2} (t.t_1.t_{2})^j \sum_{X_{q}=(y,x,-x)\in a_{T,q}: x>(-1/2)jlogp} e^x \varphi_{q,T} (\pi_q, X_q)$
$-p^{j/2} (t.t_1.t_{3})^j \sum_{X_{q}=(y,x,-x)\in a_{T,q}: x>(1/2)jlogp} e^x \varphi_{q,T} (\pi_q, X_q)$ 
$+(t.t_{2}.t_{3})^j \sum_{X_{q}=(y,x,-x)\in a_{T,q}: x>0} e^x \varphi_{q,T} (\pi_q, X_q))$
\end{center}
\begin{center}
$+(-2)\mathrm{tr} i_{P_d}(\pi_{f, \Lambda =(1/2,-1/2)}^{p,q}, \varphi^{p,q})$
$p^{j}(-p^{j/2} (t.t_1.t_{2})^j \sum_{X_{q}=(y,x,-x)\in a_{T,q}: x<(-1/2)jlogp} e^{-x} \varphi_{q,T} (\pi_q, X_q)$
$-p^{-j/2} (t.t_1.t_{3})^j \sum_{X_{q}=(y,x,-x)\in a_{T,q}: x<(1/2)jlogp} e^{-x} \varphi_{q,T} (\pi_q, X_q)$ 
$+(t.t_{2}.t_{3})^j \sum_{X_{q}=(y,x,-x)\in a_{T,q}: x<0} e^{-x} \varphi_{q,T} (\pi_q, X_q))$
\end{center}

\begin{center}
$=\sum_{\pi : \pi^{\vee}_{\mathbb{R}}=\mu^{-1}e^{i\theta}}2\mathrm{tr} i_{P_u}(\pi_{h, \Lambda =(-1/2,1/2)}^{p,q}, \varphi^{p,q})$
$p^{j}(p^{-j/2} (t.t_1.t_2)^j \sum_{X_{q}=(y,x,-x)\in a_{T,q}} e^x \varphi_{q,T} (\pi_q, X_q)$ 
$-(t.t_2.t_3 )^j \sum_{X_{q}=(y,x,-x)\in a_{T,q}: x>0} e^{\vert x \vert} \varphi_{q,T} (\pi_q, X_q))$
\end{center}
\begin{center}
$+2\mathrm{tr} i_{P_d}(\pi_{f, \Lambda =(1/2,-1/2)}^{p,q}, \varphi^{p,q})$

$p^{j}(p^{-j/2} (t.t_1.t_{3})^j \sum_{X_{q}=(y,x,-x)\in a_{T,q}} e^{-x} \varphi_{q,T} (\pi_q, X_q)$
$-(t.t_2.t_{3})^j \sum_{X_{q}=(y,x,-x)\in a_{T,q}: x<0} e^{\vert x \vert} \varphi_{q,T} (\pi_q, X_q))$
\end{center}

\begin{center}
$=\sum_{\pi : \pi^{\vee}_{\mathbb{R}}=\mu^{-1}e^{i\theta}} 2p^{j/2} ((t.t_1.t_{2})^j + (t.t_1.t_{3})^j) \mathrm{tr} I_{T}(\pi_{h, \Lambda =(-1/2,1/2)}^{p}, \varphi^{p})$

$-\sum_{\pi : \pi^{\vee}_{\mathbb{R}}=\mu^{-1}e^{i\theta}} 2p^{j}(t.t_2.t_3 )^j \mathrm{tr} I_{T}(\pi_{h, \Lambda =(-1/2,1/2)}^{p,q}, \varphi^{p,q})$

$\sum_{X_{q}=(y,x,-x)\in a_{T,q}} e^{\vert x \vert} \varphi_{q,T} (\pi_q, X_q))$
\end{center}

\textbf{Third Major Sum for}  $H$

\begin{center}
$\sum_{\pi : \pi^{\vee}_{\mathbb{R}}=\mu^{-1}}2\mathrm{tr} i_{P_u}(\pi_{h, \Lambda =(-1,1)}^{p,q}, \varphi^{p,q})$
$p^{j}(-p^{-j} (t.t_1.t_{2})^j \sum_{X_{q}=(y,x,-x)\in a_{T,q}: x>(-1/2)jlogp} e^{2x} \varphi_{q,T} (\pi_q, X_q)$
$-p^{j} (t.t_1.t_{3})^j \sum_{X_{q}=(y,x,-x)\in a_{T,q}: x>(1/2)jlogp} e^{2x} \varphi_{q,T} (\pi_q, X_q)$ 
$+(t.t_2.t_{3})^j \sum_{X_{q}=(y,x,-x)\in a_{T,q}: x>0} e^{2x} \varphi_{q,T} (\pi_q, X_q))$
\end{center}

\begin{center}
$+2\mathrm{tr} i_{P_d}(\pi_{h, \Lambda =(1,-1)}^{p,q}, \varphi^{p,q})$
$p^{j}(-p^{j} (t.t_1.t_{2})^j \sum_{X_{q}=(y,x,-x)\in a_{T,q}: x<(-1/2)jlogp} e^{-2x} \varphi_{q,T} (\pi_q, X_q)$
$-p^{-j} (t.t_1.t_{3})^j \sum_{X_{q}=(y,x,-x)\in a_{T,q}: x<(1/2)jlogp} e^{-2x} \varphi_{q,T} (\pi_q, X_q)$ 
$+(t.t_2.t_{3})^j \sum_{X_{q}=(y,x,-x)\in a_{T,q}: x<0} e^{-2x} \varphi_{q,T} (\pi_q, X_q))$
\end{center}

\begin{center}
$=\sum_{\pi : \pi^{\vee}_{\mathbb{R}}=\mu^{-1}}2\mathrm{tr} i_{P_u}(\pi_{h, \Lambda =(-1/2,1/2)}^{p,q}, \varphi^{p,q})$
$p^{j}(-p^{-j} (t.t_1.t_2)^j \sum_{X_{q}=(y,x,-x)\in a_{T,q}} e^{2x} \varphi_{q,T} (\pi_q, X_q)$ 
$+(t.t_2.t_3 )^j \sum_{X_{q}=(y,x,-x)\in a_{T,q}: x>0} e^{\vert 2x \vert} \varphi_{q,T} (\pi_q, X_q))$
\end{center}
\begin{center}
$+2\mathrm{tr} i_{P_d}(\pi_{f, \Lambda =(1/2,-1/2)}^{p,q}, \varphi^{p,q})$

$p^{j}(-p^{-j} (t.t_1.t_{3})^j \sum_{X_{q}=(y,x,-x)\in a_{T,q}} e^{-2x} \varphi_{q,T} (\pi_q, X_q)$
$+(t.t_2.t_{3})^j \sum_{X_{q}=(y,x,-x)\in a_{T,q}: x<0} e^{\vert 2x \vert} \varphi_{q,T} (\pi_q, X_q))$
\end{center}

\begin{center}
$=\sum_{\pi : \pi^{\vee}_{\mathbb{R}}=\mu^{-1}} (-2)((t.t_1.t_{2})^j + (t.t_1.t_{3})^j) \mathrm{tr} I_{T}(\pi_{h, \Lambda =(-1/2,1/2)}^{p}, \varphi^{p})$

$+\sum_{\pi : \pi^{\vee}_{\mathbb{R}}=\mu^{-1}} 2p^{j}(t.t_2.t_3 )^j \mathrm{tr} I_{T}(\pi_{h, \Lambda =(-1/2,1/2)}^{p,q}, \varphi^{p,q})$

$\sum_{X_{q}=(y,x,-x)\in a_{T,q}} e^{\vert 2x \vert} \varphi_{q,T} (\pi_q, X_q))$
\end{center}

\begin{center}

\pagebreak
SIDE BY SIDE:
\end{center}

\textbf{First Sum for} $G$

\begin{center}
$\sum_{\pi : \pi^{\vee}_{\mathbb{R}}=e^{-i\theta}} 2p^{3j/2} ((t.t_1.t_2)^j + (t.t_1.t_3)^j)\mathrm{tr} I_{T}(\pi_{f, \Lambda =(-1/2,1/2)}^{p}, \phi^p)$
 
$+\sum_{\pi : \pi^{\vee}_{\mathbb{R}}=e^{-i\theta}} 2p^{j} (t.t_2.t_{3})^j \mathrm{tr} I_{T}(\pi_{f, \Lambda =(-1/2,1/2)}^{p,q}, \phi^{p,q})$

$\sum_{X_{q}=(y,x,-x)\in a_{T,q}} e^{\vert x \vert} \phi_{q,T} (\pi_q, X_q))$
\end{center}

\textbf{First Sum for} $H$

\begin{center}
$\sum_{\pi : \pi^{\vee}_{\mathbb{R}}=\mu^{-1}e^{-i\theta}} 2p^{j/2} ((t.t_1.t_{2})^j + (t.t_1.t_{3})^j) \mathrm{tr} I_{T}(\pi_{h, \Lambda =(-1/2,1/2)}^{p}, \varphi^{p})$

$-\sum_{\pi : \pi^{\vee}_{\mathbb{R}}=\mu^{-1}e^{-i\theta}} 2p^{j}(t.t_2.t_3 )^j \mathrm{tr} I_{T}(\pi_{h, \Lambda =(-1/2,1/2)}^{p,q}, \varphi^{p,q})$

$\sum_{X_{q}=(y,x,-x)\in a_{T,q}} e^{\vert x \vert} \varphi_{q,T} (\pi_q, X_q))$\medskip
\end{center}

\begin{center}
PUT TOGETHER:\medskip
\end{center}

\begin{center}
$\sum_{\pi : \pi^{\vee}_{\mathbb{R}}=e^{-i\theta}} 2p^{j}(p^{j/2} + p^{-j/2})((t.t_1.t_2)^j + (t.t_1.t_3)^j)\mathrm{tr} I_{T}(\pi_{f, \Lambda =(-1/2,1/2)}^{p}, \phi^p)$\medskip
\end{center}

\begin{center}
SIDE BY SIDE:
\end{center}

\textbf{Second Sum for} $G$

\begin{center}
$=\sum_{\pi : \pi^{\vee}_{\mathbb{R}}=e^{i\theta}} 2p^{3j/2} ((t.t_{1}.t_{2})^j + (t.t_{1}.t_{3})^j)\mathrm{tr} I_{T}(\pi_{f, \Lambda =(-1/2,1/2)}^{p}, \phi^p)$
 
$+\sum_{\pi : \pi^{\vee}_{\mathbb{R}}=e^{i\theta}} 2p^{j}(t.t_2.t_{3})^j \mathrm{tr} I_{T}(\pi_{f, \Lambda =(-1/2,1/2)}^{p,q}, \phi^{p,q})$

$\sum_{X_{q}=(y,x,-x)\in a_{T,q}} e^{\vert x \vert} \phi_{q,T} (\pi_q, X_q))$
\end{center}

\textbf{Second Sum for} $H$
\begin{center}
$=\sum_{\pi : \pi^{\vee}_{\mathbb{R}}=\mu^{-1}e^{i\theta}} 2p^{j/2} ((t.t_1.t_{2})^j + (t.t_1.t_{3})^j) \mathrm{tr} I_{T}(\pi_{h, \Lambda =(-1/2,1/2)}^{p}, \varphi^{p})$

$-\sum_{\pi : \pi^{\vee}_{\mathbb{R}}=\mu^{-1}e^{i\theta}} 2p^{j}(t.t_2.t_3 )^j \mathrm{tr} I_{T}(\pi_{h, \Lambda =(-1/2,1/2)}^{p,q}, \varphi^{p,q})$

$\sum_{X_{q}=(y,x,-x)\in a_{T,q}} e^{\vert x \vert} \varphi_{q,T} (\pi_q, X_q))$\medskip
\end{center}

\begin{center}
PUT TOGETHER:\medskip
\end{center}

\begin{center}
$\sum_{\pi : \pi^{\vee}_{\mathbb{R}}=e^{i\theta}} 2p^{j}(p^{j/2} + p^{-j/2})((t.t_1.t_2)^j + (t.t_1.t_3)^j)\mathrm{tr} I_{T}(\pi_{f, \Lambda =(-1/2,1/2)}^{p}, \phi^p)$\medskip
\end{center}

\begin{center}
SIDE BY SIDE:
\end{center}

\textbf{Third Sum for} $G$
\begin{center}
$=\sum_{\pi : \pi^{\vee}_{\mathbb{R}}=1} (-2)p^{2j} ((t.t_1.t_2)^j + (t.t_1.t_3)^j)\mathrm{tr} I_{T}(\pi_{f, \Lambda =(-1,1)}^{p}, \phi^p)$
 
$-\sum_{\pi : \pi^{\vee}_{\mathbb{R}}=1} 2p^{j} (t.t_2.t_{3})^j \mathrm{tr} I_{T}(\pi_{f, \Lambda =(-1,1)}^{p,q}, \phi^{p,q})$

$\sum_{X_{q}=(y,x,-x)\in a_{T,q}} e^{\vert 2x \vert} \phi_{q,T} (\pi_q, X_q))$
\end{center}

\textbf{Third Sum for} $H$
\begin{center}
$=\sum_{\pi : \pi^{\vee}_{\mathbb{R}}=\mu^{-1}} (-2)((t.t_1.t_{2})^j + (t.t_1.t_{3})^j) \mathrm{tr} I_{T}(\pi_{h, \Lambda =(-1,1)}^{p}, \varphi^{p})$

$+\sum_{\pi : \pi^{\vee}_{\mathbb{R}}=\mu^{-1}} 2p^{j}(t.t_2.t_3 )^j \mathrm{tr} I_{T}(\pi_{h, \Lambda =(-1,1)}^{p,q}, \varphi^{p,q})$

$\sum_{X_{q}=(y,x,-x)\in a_{T,q}} e^{\vert 2x \vert} \varphi_{q,T} (\pi_q, X_q))$\medskip
\end{center}

\begin{center}

PUT TOGETHER:\medskip
\end{center}
\begin{center}
$\sum_{\pi : \pi^{\vee}_{\mathbb{R}}=1} (-2)p^{j}(p^{j} + p^{-j})((t.t_1.t_2)^j + (t.t_1.t_3)^j)\mathrm{tr} I_{T}(\pi_{f, \Lambda =(-1,1)}^{p}, \phi^p)$
\end{center}

\pagebreak
\subsection{Conclusion}

Let $K$ be the compact open subgroup of $\mathbf{G} = \mathbf{GU}(2,1)$ introduced in subsection 2.2 and let $K_G$ be the subgroup of $R_{E/\mathbb{Q}}\mathbf{G}_E$ that is associated to $K$ in the manner described in \cite[p.~134]{M}. Following Laumon in \cite[p.~338]{L}, we can now use class field theory and the results of Eichler-Shimura to identify the first component of each of the three combined sums above as the trace on a Galois module that is associated to the automorphic representation that appears in the second part of each sum. Thus, we obtain three Gal$(\bar{E}/E) \times C^{\infty}_c(G(\mathbb{A}_f) // K_G)$ -modules, call them $M_1$, $M_2$, and $M_3$, and we form the virtual Gal$(\bar{E}/E) \times C^{\infty}_c(G(\mathbb{A}_f) // K_G)$ -module $M = M_1 + M_2 - M_3$. 
The parabolic part of the trace will be identified as the trace on this virtual module.

Let $F$ be a number field large enough to contain the field of definition of the three Galois modules above. We may then consider the virtual Gal$(\bar{E}/E) \times C^{\infty}_c(\mathbf{G}(\mathbb{A}_f) // K)$ -module
\begin{center}
$W_{\lambda} = W_{\ell} \otimes_{\mathbb{Q}_{\ell}} F_{\lambda} = \sum_{i\geqslant0} (-1)^{i} H^i_c(S_K(\mathbf{G}) \otimes_{E} \bar{E}, F_{\lambda}).$\medskip
\end{center}
for a finite place $\lambda$ of $F$ and a prime $\ell$ such that $\lambda$ divides $\ell$.

We can now prove our main result:

\begin{theorem} 
\textit{Let $\lambda$ be a finite place of the number field $F$ above, and let $\ell$ be the prime such that $\lambda$ divides $\ell$. Let $p$ be a prime different from $\ell$, good with respect to $K$, and split in $E$. Then for every function $f^p \in C^{\infty}_c(\mathbf{G}(\mathbb{A}^p_f)//K^p)_{\mathbb{Q}}$ and for every $j > 0$, the parabolic part of the trace}
\begin{center}
$\mathrm{tr}(\Phi^p_j \times f^p, W_{\lambda})$
\end{center}
is given by

\begin{center}
$\frac{1}{2}(J_{T}^{G}(\phi^{G}) + J_{T}^{H}(\phi^{H}))$\medskip

$=\sum_{\pi : \pi_{\mathbb{R}}(g)=e^{i\theta}} (1 + p^j)p^{j/2}((t.t_1.t_2)^j + (t.t_1.t_3)^j)\mathrm{tr} I_{T}(\pi_{f, \Lambda =(-1/2,1/2)}^{p}, \phi^{p})$\medskip

$+\sum_{\pi : \pi_{\mathbb{R}}(g)=e^{-i\theta}} (1 + p^j)p^{j/2}((t.t_1.t_2)^j + (t.t_1.t_3)^j)\mathrm{tr} I_{T}(\pi_{f, \Lambda =(-1/2,1/2)}^{p}, \phi^{p})$\medskip

$-\sum_{\pi : \pi_{\mathbb{R}}(g)=1} p^{j}(p^{j} + p^{-j})((t.t_1.t_2)^j + (t.t_1.t_3)^j)\mathrm{tr} I_{T}(\pi_{f, \Lambda =(-1,1)}^{p}, \phi^{p}),$\medskip
\end{center}
where $\pi$ ranges over cuspidal automorphic representations of $T = T^0 \rtimes \theta$.
\end{theorem}

\begin{proof} The computations above establish the claim for test functions $\phi^G$ and $\phi^H$ on $G$ and $H$, respectively, that satisfy the assumptions of theorem 4.12.  In particular, the component at infinity $\phi^G_{\mathbb{R}}$ is chosen as in section 5.3. so that it is associated to the function $
f_{\mathbb{R}}$ on $\mathbf{GU}(2,1)$, and similarly for $\phi^H_{\mathbb{R}}$ and $f^H_{\mathbb{R}}$. In both cases, both of these functions can be adjusted by changing the choice of pseudocoefficients so as to assure that the assumptions of 4.12 are satisfied; this follows from \cite[p.~320]{L} and \cite[p.~124]{M}. In order to remove these assumptions, we can imitate Laumon's argument \cite[p.~344]{L}. It is here that we need the identification of the first part of each summand above as the trace of Frob$_p^j$ on a Galois module. Thus, we obtain the equality of $2(J_{T}^{G}(\phi^{G}) + J_{T}^{H}(\phi^{H}))$ with the displayed expression for $\phi^G$ and $\phi^H$ that are arbitrary away from $p$.\footnote{Note that in this first part of the argument, we do not need to view the functions $\phi^{p, \infty}$ away from $p$ and infinity as being associated with functions on the unitary groups, and so we need not worry about the transference of the conditions required by 4.12.} Finally, for all $q \neq p$, we can now start with any test function $f_q$ on $\mathbf{GU}(2,1)$, a transfer $h_q$ of $f_q$ to $\mathbf{H}$, and by \cite[p.~136]{M}, we can find an associated function for each one. By the first part of our argument, these associated functions are not required to satisfy any further assumptions, and so the results holds for any $f^p$.
\end{proof}

%define test function; say that by the same argument as in L; earlier, for the cancellations, need to say by using char identities as in L%

%\bibliography {bib/network,bib/naming}    % bibliography references
\bibliographystyle {uclathesis}

\bibliographystyle{amsplain}

\end{document}